\documentclass[review]{elsarticle}
\usepackage{hyperref}
\hypersetup{
	colorlinks=true,
	linkcolor=black,
	citecolor=black,
	urlcolor=blue
}

\usepackage{color}
\usepackage{graphicx}
\usepackage{subcaption}
\usepackage[normalem]{ulem}
\let\isout\sout \renewcommand{\sout}[1]{\ifmmode\text{\isout{\ensuremath{{\color{blue}#1}}}}\else\isout{{\color{blue}#1}}\fi}
\usepackage{hyperref}
\usepackage{multirow}
\usepackage{amsmath,amstext,amssymb}
\usepackage{amsthm}
\usepackage{tabularx,longtable,multirow,diagbox}
\usepackage{enumerate}
\usepackage{mathrsfs}

\usepackage{caption}
\usepackage{booktabs}
\usepackage{geometry}
\usepackage{algorithm}
\usepackage{algorithmicx}
\usepackage{algpseudocode}

\numberwithin{equation}{section}

\theoremstyle{plain}
\newtheorem{Theorem}{Theorem}[section]

\newtheorem{Definition}{Definition}[section]
\newtheorem{Lemma}[Theorem]{Lemma}

\theoremstyle{remark}
\newtheorem{Remark}[Theorem]{Remark}

\begin{document}
	\begin{frontmatter}
		
		\title{High-Order Exponential Integrators with Improved Uniform Accuracy for Charged-Particle in a Perpendicular Strong Magnetic Field}
		\author[nju]{Zhihao Qi\fnref{myfootnote}}
		\ead{zhihaoqi@smail.nju.edu.cn}
		
		\author[nju]{Weibing Deng\corref{mycorrespondingauthor}\fnref{myfootnote1}}
		\ead{wbdeng@nju.edu.cn}
		\cortext[mycorrespondingauthor]{Corresponding author}
		\fntext[myfootnote1]{The work of this author was
			supported by the National Key R\&D Program of China (2024YFA1012600), and by the NSF of China grant 12171237.}
		\address[nju]{School of Mathematics,
			Nanjing University, Nanjing 210093, People's Republic of China}
		
		\begin{abstract}
			This paper considers a class of charged-particle dynamics problems in which the particle is subjected to a magnetic force, with a magnetic flux density inversely proportional to a small parameter $0<\varepsilon\ll 1$, and a nonlinear electric force. The resulting highly oscillatory behavior poses significant challenges for numerical computation. To enhance the performance of exponential integrators (EIs), this paper employs a technique that linearizes the ordinary differential equation through a dimension-raising approach.  Based on this technique, a new family of EIs is developed that achieves arbitrarily high order. For short-time simulations on the interval $[0,T]$, it is rigorously proved that the proposed method--which employs auxiliary polynomials of degree $k$ and a time step $\Delta t$--satisfies two distinct error bounds: $O(\varepsilon \Delta t^{k+1})$ and $O(\varepsilon^{k+2})$. The latter bound guarantees that the algorithm stays accurate even when the step size is of order $O(1)$. Furthermore, when a large step size $\varepsilon^{-1}\Delta t$ is used to simulate the long-term dynamics over $[0,\varepsilon^{-1}T]$, the numerical scheme attains a uniform convergence rate of $O(\Delta t^{k+1})$. Several numerical experiments confirm these theoretical results.
		\end{abstract}
		
	\end{frontmatter}
	

\section{Introduction}\label{Sec1}
High-precision simulation of charged-particle motion in electromagnetic fields has important applications across many fields, including plasma physics, nuclear physics, inertial confinement fusion, and aerospace engineering. Charged-particle dynamics (CPD) constitutes a fundamental issue in numerous physical problems involving field–particle interactions and has attracted extensive and sustained attention over the past few decades \cite{C2015,AKNI2006}. For example, in the kinetic modeling of collisionless plasmas, the Vlasov equation is often used together with the electromagnetic field equations to describe the collective behavior of charged particles through a distribution function in phase space \cite{T2004,FE1998,FXS2018}. When this function is discretized using particle shape functions within the Particle-in-Cell (PIC) framework \cite{XQ2018,L2014,KKMS2017,WWYZ2026}, its time evolution can be interpreted as the motion of a large number of charged macroparticles in an electromagnetic field. In practical computations, it is necessary to advance the position of each particle by solving the CPD equation.

This paper investigates the two-dimensional dynamics of a charged particle subject to a uniform, perpendicular, and strong magnetic field. The governing equations are given by
\begin{equation}\label{CPD2D-homo}
  \begin{aligned}
      &\ddot{\mathbf{y}}=\frac{1}{\varepsilon}J\dot{\mathbf{y}}+\mathbf{E}(\mathbf{y}),\quad t\in \left[0,\frac{T}{\varepsilon^\nu}\right], \\
      &\mathbf{y}(0)=\mathbf{y}_0,\quad \dot{\mathbf{y}}(0)=\dot{\mathbf{y}}_0,
  \end{aligned}
\end{equation}
where $\mathbf{y}(t)\in\mathbb{R}^2$ denotes the position vector of the charged particle, and $\mathbf{y}_0,\dot{\mathbf{y}}_0\in\mathbb{R}^2$ are the initial position and velocity, respectively. The dimensionless parameter $\varepsilon$ characterizes the strength of the magnetic field. In the regime $0<\varepsilon\ll 1$, the cyclotron frequency becomes very large, rendering (\ref{CPD2D-homo}) a prototypical highly oscillatory differential equation \cite{HLW2013, WYW2013}. The function $\mathbf{E}:\mathbb{R}^2\to\mathbb{R}^2$ represents a prescribed nonlinear electric field, while $J$ denotes the second-order identity symplectic matrix defined by
\begin{equation*}
  J=\begin{bmatrix} 0 & 1 \\ -1 & 0 \end{bmatrix}.
\end{equation*}
$T>0$ is a fixed positive constant, and the parameter $\nu\geq 0$ determines the scale of the final time.

In this work, two cases are considered: $\nu=0$ and $\nu=1$. The case $\nu=0$ corresponds to simulations over a fixed time interval independent of $\varepsilon$, in which the motion is dominated by high-frequency cyclotron rotation on a small scale, accompanied by a limited drift effect. When $\nu=1$, the interpretation of $\varepsilon$ is characterized by the following two assumptions in plasma physics \cite{CLMZ2017,FXS2018}: (\romannumeral1) the ratio of the plasma frequency to the cyclotron frequency is $\varepsilon$; (\romannumeral2) the product of the macroscopic time scale and the plasma frequency is equal to $\varepsilon^{-1}$. For single-particle motion in particular, the particle undergoes high-frequency cyclotron rotation with period $O(\varepsilon)$ while simultaneously exhibiting slower dynamics—known as the “guiding-center $E\times B$ drift”—on the $O(\varepsilon^{-1})$ scale. Consequently, the equation (\ref{CPD2D-homo}) exhibits pronounced multiscale characteristics. A visual illustration of charged-particle trajectories under different magnetic fields in the case $\nu=1$ is presented in Figure \ref{fig0}. It can be observed that all particle trajectories exhibit broadly similar behavior. The stronger magnetic field confines the particles more effectively, resulting in a smaller cyclotron radius and a longer characteristic motion time.

\begin{figure}[htb]
	\centering
	\includegraphics[width=0.3\linewidth]{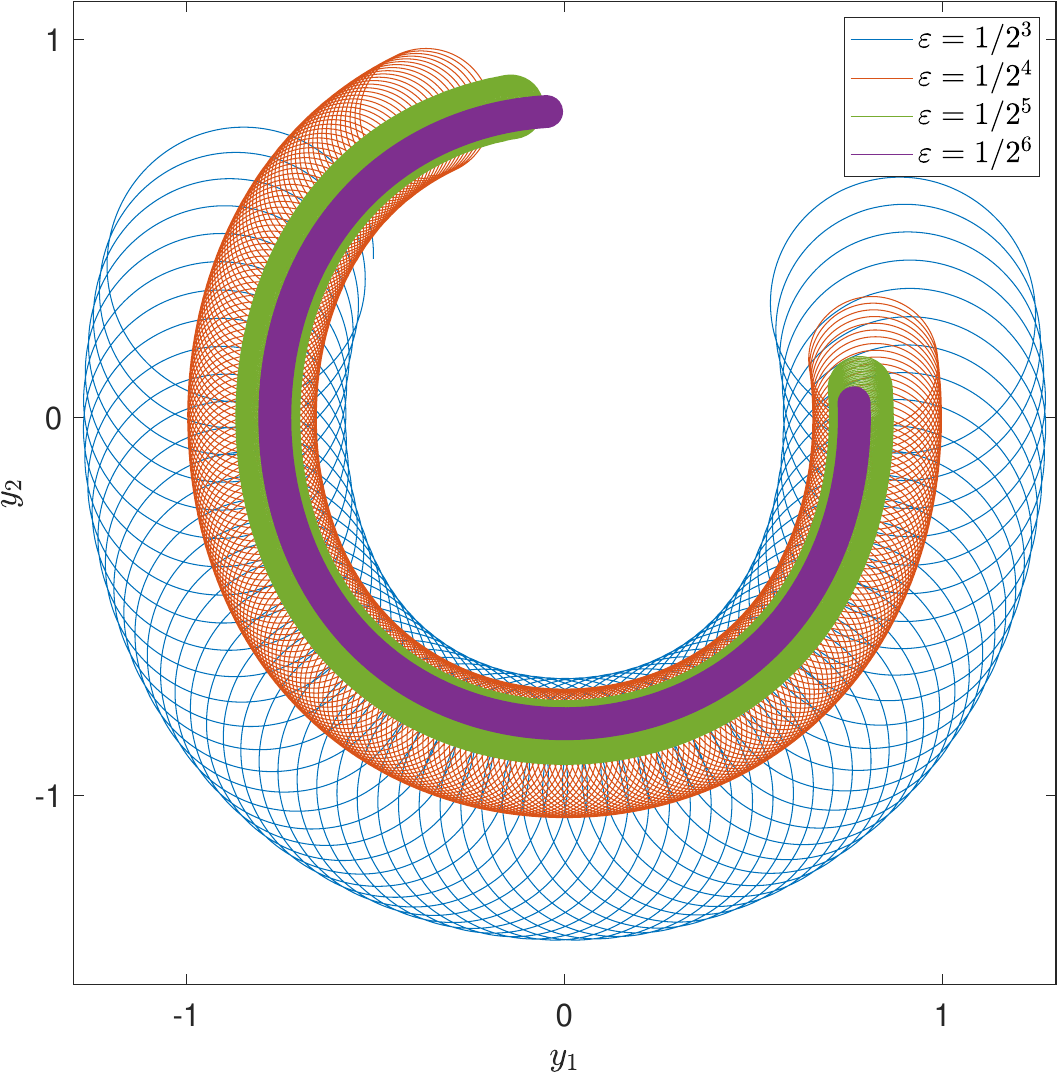}
  \caption{Trajectories of a charged particle in an electromagnetic field for different magnetic flux densities with $\nu=1$. This schematic illustration shows the simulation results over the time interval $[0,6/\varepsilon]$, using the electric field function and initial conditions from numerical Example \ref{Ex3} in Section \ref{Sec4}.}\label{fig0}
\end{figure}

As a classical model in electromagnetism, the CPD (\ref{CPD2D-homo}) has a long research history from the physical perspective \cite{N1983,AKNI2006,BL2004,T2004,CB2009}. From a computational viewpoint, numerical methods can be classified into two categories according to the magnetic field regime: normal magnetic fields ($\varepsilon\approx 1$) and strong magnetic fields ($0<\varepsilon\ll 1$). Early studies primarily focused on the former, including the well-known Boris algorithm \cite{B1970} and related works \cite{QZXLST2013,HL2018}, which are widely used in engineering applications because of their long-term stability and high efficiency. In addition, subsequent studies developed volume-preserving methods \cite{HSLQ2015}, symplectic methods \cite{HZSLQ2017,T2016,ZQTLHX2016}, energy-preserving methods \cite{BMR2019,LW2019}, splitting methods \cite{KKKO2015}, and other related algorithms \cite{HL2017,HL2020}. However, when these schemes are applied directly to the parameter regime $0<\varepsilon\ll 1$, their performance is often severely restricted, mainly in two respects. First, the stringent stability condition is imposed, requiring the time step to resolve the highest frequency, i.e., $h<O(\varepsilon)$, which leads to a substantial computational burden. Second, the numerical error of most classical schemes often depends on negative powers of $\varepsilon$.

To address these limitations, several methods with improved or uniform accuracy have been developed in recent years, with the primary objective of designing schemes whose error is independent of $\varepsilon$. Representative examples include the filtered Boris algorithm \cite{HLW2020} and some splitting methods \cite{WZ2021,Y2024}; however, these approaches still suffer from certain step-size restrictions. Other numerical methods have also been incorporated into the PIC framework, such as asymptotic-preserving methods \cite{CC2023,FR2016,FRT2026}, multi-revolution composition methods \cite{CCZ2018,CCLMZ2020}, micro-macro methods \cite{CCLMZ2020}, and two-scale formulation methods \cite{CLMZ2017,CCLMZ2019,WZ2023}. These methods permit the use of large time steps that are not constrained by $\varepsilon$. Nevertheless, most of them are at most second-order convergent. Although the construction of higher-order schemes while preserving the remarkable property of uniform accuracy has attracted attention in recent years \cite{WJ2023}, this topic still demands further investigation and deeper insight. Moreover, most existing studies on uniformly accurate methods focus on short-time simulations with $\nu=0$, and research on the case $\nu=1$ remains limited \cite{CLMZ2017,FHS2015,FRT2026,FXS2018}.

This work focuses on improving traditional exponential integrators (EIs) when applied to (\ref{CPD2D-homo}) in strong magnetic fields. It is well known that EIs have attracted widespread attention since the beginning of this century due to their advantages in solving stiff first-order equations \cite{CM2002,K2005,T2006,CCO2008,MW2017}. The review \cite{HO2010} provides a basic introduction. Several studies have also combined EIs with linearization techniques \cite{RG1997,DBCOJ2007,KO2013}, such as the exponential Rosenbrock-type methods (ERM) \cite{CO2009,HOS2009,LL2023}. Some recent works applying EIs to (\ref{CPD2D-homo}) have aimed to improve computational efficiency and structure-preserving properties \cite{FHS2015,LW2022,NJT2024,NJT2025}. Although these studies have explored the potential of EIs for this problem, they do not conduct in-depth investigations into the uniformity of the error with respect to $\varepsilon$, nor have they extended this property to high-order schemes. Our goal is to propose a new class of EIs that achieves high-order uniform convergence, permits time steps unrestricted by $\varepsilon$, and remains applicable to both parameter choices $\nu=0$ and $\nu=1$.

Our application of EIs is based on the local linear extension technique proposed in \cite{QDZ2025} very recently, whose main idea is to achieve high-order linearization of nonlinear components through a dimension-raising approach. Specifically, based on the numerical solution at the $n$-th time step, a finite number of new auxiliary variables are introduced, each defined as a polynomial in the state variables of degree at most $k$. The governing equations for these auxiliary variables are then derived, yielding a new system of ordinary differential equations in a higher-dimensional space. The EI is subsequently applied to this extended system, and the solution is then projected back to the original space to construct the numerical approximation at next time step.

We present a detailed theoretical analysis of this method applied to (\ref{CPD2D-homo}). For the case $\nu=0$, we prove that, when auxiliary polynomials of degree at most $k$ and time step $\Delta t$ are used, the method achieves error bounds of $\min\{O(\varepsilon^{k+2}),O(\varepsilon \Delta t^{k+1})\}$ for $\mathbf{y}$ and $\min\{O(\varepsilon^{k+1}),O(\Delta t^{k+1})\}$ for $\dot{\mathbf{y}}$. The convergence order is improved by directly increasing the degree $k$, without imposing additional order conditions. Thus, the proposed method can be regarded as a design framework capable of achieving arbitrarily high-order convergence. Furthermore, a notable and practically useful feature of the new EIs is that, in addition to the conventional bound expressed in powers of $\Delta t$, there exists a distinct bound expressed in powers of $\varepsilon$. This ensures that, even when the method uses large step sizes up to $O(1)$, the resulting error remains significantly small. Furthermore, for the case $\nu=1$, if the step size is chosen as $\varepsilon^{-1}\Delta t$, the numerical scheme attains a convergence order of $O(\Delta t^{k+1})$. This highlights another key feature of the method: it permits large step sizes of $O(\varepsilon^{-1})$ while still guaranteeing a uniformly high convergence order. Consequently, long-term problems can be solved efficiently using a finite and relatively small number of time steps, thereby greatly enhancing computational efficiency.

The remainder of this paper is organized as follows. In Section \ref{Sec2}, we define the local linear extension variables and the associated system, on the basis of which new EIs are constructed. Section \ref{Sec3} presents a rigorous convergence analysis under different parameter regimes, establishing the high-order and uniform accuracy of the proposed methods. In Section \ref{Sec4}, numerical experiments are presented to verify the theoretical convergence results. The final section concludes the paper and provides an outlook on future research directions.

\section{Numerical Schemes}\label{Sec2}
We begin by rescaling the velocity variable and introducing $\mathbf{x}=[\mathbf{y}^\top,\varepsilon\dot{\mathbf{y}}^\top]^\top$, which transforms the CPD equation (\ref{CPD2D-homo}) into the following equivalent first-order system:
\begin{equation}\label{P}
\begin{aligned}
  &\dot{\mathbf{x}}=\frac{1}{\varepsilon}A_1\mathbf{x}+\varepsilon\mathbf{f}(\mathbf{x}),\quad t\in \left[0,\frac{T}{\varepsilon^\nu}\right],\\
  &\mathbf{x}(0)=\mathbf{x}_0,
\end{aligned}
\end{equation}
where
\begin{equation}\label{Afx0}
  A_1=\left[\begin{array}{cc}
  \mathbf{0} & I_2 \\ \mathbf{0} & J             \end{array}\right],\quad
  \mathbf{f}(\mathbf{x})=\left[\begin{array}{c}
  \mathbf{0} \\ \mathbf{E}(\mathbf{y})
  \end{array}\right],\quad
  \mathbf{x}_0=\left[\begin{array}{c}
  \mathbf{y}_0 \\ \varepsilon\dot{\mathbf{y}}_0 \end{array}\right].
\end{equation}
Here, $I_m$ denotes the $m$-dimensional identity matrix.

To define higher-dimensional variables, we first introduce some notation for multi-indices. For a nonnegative integer $j$, let
\begin{equation*}
  \mathbb{I}^{[[j]]}:=\left\{\alpha=[\alpha_1,\alpha_2,\alpha_3,\alpha_4]^\top:\sum_{l=1}^4\alpha_l=j~\text{with}~\alpha_1,\cdots,\alpha_4\in\{0,\cdots,j\}\right\}.
\end{equation*}
We call an element $\alpha\in\mathbb{I}^{[[j]]}$ a multi-index and refer to $j$ as its length, denoted by $|\alpha|=j$. Thanks to the component-wise structure of multi-indices, we can equip $\mathbb{I}^{[[j]]}$ with the lexicographic order. This order assigns a unique integer label to each multi-index; equivalently, it specifies a bijection $\tau^{[[j]]}:\mathbb{I}^{[[j]]}\to\{1,\dots,D^{[[j]]}\}$, where $D^{[[j]]}=|\mathbb{I}^{[[j]]}|$ is the cardinality of $\mathbb{I}^{[[j]]}$. For $j=1$, we set
\begin{equation*}
  \tau^{[[1]]}(e_l^4)=l,\quad l=1,2,3,4,
\end{equation*}
with $e_l^4$ the $l$-th standard unit vector in $\mathbb{R}^4$. Let $\mathbb{I}^{[k]}=\bigcup_{j=0}^k \mathbb{I}^{[[j]]}$ be the set of all multi-indices of length at most $k$, and let $D^{[k]}$ be its cardinality. We then define an ordering on $\mathbb{I}^{[k]}$ through a bijection $\tau^{[k]}:\mathbb{I}^{[k]}\to\{1,\dots,D^{[k]}\}$ by arranging the multi-indices in increasing order of length. With this convention, the zero multi-index (of length $0$) comes first: $$\tau^{[k]}([0,0,0,0]^\top)=1.$$
For a multi-index $\alpha\in\mathbb{I}^{[[j]]}$ with $j\ge 1$, we set
 $$
\tau^{[k]}(\alpha)=D^{[j-1]}+\tau^{[[j]]}(\alpha).$$
In particular, this yields
\begin{equation}\label{tau_k}
  \tau^{[k]}(e_l^4)=l+1,\quad l=1,2,3,4.
\end{equation}

For a given point $\hat{\mathbf{x}}\in\mathbb{R}^4$, the multi-index $\alpha\in\mathbb{I}^{[k]}$ induces a polynomial variable in $\mathbf{x}$ defined by
\begin{equation}\label{LLEV01}
  (\mathbf{x}-\hat{\mathbf{x}})^\alpha:=\prod_{l=1}^4(x_l-\hat{x}_l)^{\alpha_l},
\end{equation}
where $x_i$ and $\hat{x}_i$ denote the $i$-th components of $\mathbf{x}$ and $\hat{\mathbf{x}}$, respectively. With these preparations, we introduce the following definition of local linear extension variables.
\begin{Definition}\label{LLEV}
For an integer $0\leq j\leq k$, the sets of degree-$j$ polynomials and polynomials of degree at most $k$ of the form (\ref{LLEV01}) are denoted, respectively, by
\begin{equation}\label{LLEV02}
  P_\mathbf{x}^{[[j,\hat{\mathbf{x}}]]}=\left\{(\mathbf{x}-\hat{\mathbf{x}})^\alpha:\alpha\in\mathbb{I}^{[[j]]}\right\},  P_\mathbf{x}^{[k,\hat{\mathbf{x}}]}=\bigcup_{j=0}^k P_\mathbf{x}^{[[j,\hat{\mathbf{x}}]]}=\left\{(\mathbf{x}-\hat{\mathbf{x}})^\alpha:\alpha\in\mathbb{I}^{[k]}\right\}.
\end{equation}
We define a $D^{[k]}$-dimensional vector $\mathbf{x}^{[k,\hat{\mathbf{x}}]}$, whose $i$-th component is given by $(\mathbf{x}-\hat{\mathbf{x}})^{(\tau^{[k]})^{-1}(i)}\in P_\mathbf{x}^{[k,\hat{\mathbf{x}}]}$. This vector is referred to as the $k$-th order local linear extension variable of $\mathbf{x}$ at $\hat{\mathbf{x}}$.
\end{Definition}
By the multi-index ordering scheme, the entries of the vector $\mathbf{x}^{[k,\hat{\mathbf{x}}]}$ are arranged in ascending order of degree. From (\ref{tau_k}) and (\ref{LLEV01}) we see that the first entry is the constant $1$, while the second through fifth entries are first-degree polynomials satisfying $\Pi_{\mathbf{x}}\mathbf{x}^{[k,\hat{\mathbf{x}}]}=\mathbf{x}-\hat{\mathbf{x}}$. Here $\Pi_{\mathbf{x}}:\mathbb{R}^{D^{[k]}}\to\mathbb{R}^4$ denotes the projection onto coordinates $2, 3, 4$, and $5$ of a $D^{[k]}$-dimensional vector.

For a multi-index $\alpha\in\mathbb{I}^{[k]}$ we define the associated partial differential operator of order $|\alpha|$ with respect to $\mathbf{x}$ by
$$
\frac{\partial^\alpha}{\partial\mathbf{x}^\alpha}:=\frac{\partial^{|\alpha|}}{\partial x_1^{\alpha_1}\cdots\partial x_4^{\alpha_4}}.
$$
Note that the electric field $\mathbf{E}$ depends only on $\mathbf{y}$, the first two components of $\mathbf{x}$. Hence the $k$-th order Taylor expansion of $\mathbf{E}$ around a point $\hat{\mathbf{y}}\in\mathbb{R}^2$ can be written as
\begin{equation}\label{ParDiff01}
  \mathbf{E}(\mathbf{y})=\sum_{j=0}^{k}\sum_{|\alpha|=j}\frac{\mathrm{I}_{\{\alpha_3+\alpha_4=0\}}}{\alpha_1!\alpha_2!}\frac{\partial^\alpha \mathbf{E}(\hat{\mathbf{y}})}{\partial\mathbf{y}^\alpha}(\mathbf{y}-\hat{\mathbf{y}})^\alpha + \mathbf{r}^{k+1}(\mathbf{y};\hat{\mathbf{y}}),
\end{equation}
where $\mathrm{I}_{\{\cdot\}}$ denotes the indicator function that selects the multi-indices corresponding to $\mathbf{y}$, taking the value $1$ when the condition holds and $0$ otherwise. Consequently, only multi-indices with $\alpha_3=\alpha_4=0$ contribute to the sum. The term $\mathbf{r}^{k+1}(\mathbf{y};\hat{\mathbf{y}})$ represents the remainder of the $(k+1)$-th order Taylor expansion at $\hat{\mathbf{y}}$, satisfying
\begin{equation}\label{ParDiff04}
  \mathbf{r}^{k+1}(\mathbf{y};\hat{\mathbf{y}})=\sum_{|\alpha|=k+1}\frac{\mathrm{I}_{\{\alpha_3+\alpha_4=0\}}}{\alpha_1!\alpha_2!}\frac{\partial^\alpha \mathbf{E}(\eta)}{\partial\mathbf{y}^\alpha}(\mathbf{y}-\hat{\mathbf{y}})^\alpha,\, \eta\in\mathbb{R}^2.
\end{equation}

We next derive the governing equation for $\mathbf{x}^{[k,\hat{\mathbf{x}}]}$. Since the variables in $P_\mathbf{x}^{[k,\hat{\mathbf{x}}]}$ are also unknown, it is necessary to derive the equations satisfied by the higher-degree polynomial variables that are additionally introduced as extension variables to close the system. Given $\hat{\mathbf{x}}\in\mathbb{R}^4$ and a multi-index $\alpha\in\mathbb{I}^{[k]}$, it follows from (\ref{P}) and (\ref{LLEV01}) that
\begin{align}\label{ParDiff02}
\frac{\mathrm{d}}{\mathrm{d}t}(\mathbf{x}-\hat{\mathbf{x}})^\alpha =& \sum_{j=1}^4\alpha_j(x_1-\hat{x}_1)\cdots(x_j-\hat{x}_j)^{\alpha_j-1}\cdots(x_4-\hat{x}_4)\frac{\mathrm{d}(x_j-\hat{x}_j)}{\mathrm{d}t}\notag\\
=& \sum_{j=1}^4\alpha_j(\mathbf{x}-\hat{\mathbf{x}})^{\alpha-e_j^4}\left(\frac{1}{\varepsilon}\sum_{l=1}^4(A_1)_{jl}x_l+\varepsilon f_j(\mathbf{x})\right)\notag\\
=& \frac{1}{\varepsilon}\sum_{j=1}^4\sum_{l=1}^4\alpha_j(A_1)_{jl}(\mathbf{x}-\hat{\mathbf{x}})^{\alpha-e_j^4+e_l^4} + \frac{1}{\varepsilon}\sum_{j=1}^4\sum_{l=1}^4\alpha_j(A_1)_{jl}\hat{x}_l(\mathbf{x}-\hat{\mathbf{x}})^{\alpha-e_j^4} \notag\\ &+\varepsilon\sum_{j=1}^4\sum_{l=0}^{k-|\alpha|+1}\sum_{|\beta|=l}\frac{\alpha_j}{\beta_1!\cdots\beta_4!}\frac{\partial^\beta f_j(\hat{\mathbf{x}})}{\partial\mathbf{x}^\beta}(\mathbf{x}-\hat{\mathbf{x}})^{\alpha+\beta-e_j^4} \notag\\
&+\varepsilon\sum_{j=1}^4\alpha_j(\mathbf{x}-\hat{\mathbf{x}})^{\alpha-e_j^4}\bar{r}_j^{k-|\alpha|+2}(\mathbf{x};\hat{\mathbf{x}}).
\end{align}
In the last equality, we perform a Taylor expansion of $f_j$ up to order $k-|\alpha|+1$, where $f_j$ and $\bar{r}_j^{k-|\alpha|+2}$ denote, respectively, the $j$-th component of $\mathbf{f}$ and its Taylor remainder. The term $\bar{\mathbf{r}}^{k-|\alpha|+2}$ takes the form
\begin{equation}\label{ParDiff03}
  \bar{\mathbf{r}}^{k-|\alpha|+2}(\mathbf{x};\hat{\mathbf{x}})
  :=\left[\begin{array}{c}
      \mathbf{0} \\
      \mathbf{r}^{k-|\alpha|+2}(\mathbf{y};\hat{\mathbf{y}})\\
  \end{array}\right].
\end{equation}
Substituting the vector and matrix structures given in (\ref{ParDiff03}) and (\ref{Afx0}) into (\ref{ParDiff02}), and partitioning the summation over $j$ into the two parts $j=1,2$ and $j=3,4$, we obtain
\begin{align}\label{Assem01}
&\frac{\mathrm{d}}{\mathrm{d}t}(\mathbf{x}-\hat{\mathbf{x}})^\alpha \notag\\
&=\frac{1}{\varepsilon}\left(\alpha_1(\mathbf{x}-\hat{\mathbf{x}})^{\alpha-e_1^4+e_3^4} +\alpha_2(\mathbf{x}-\hat{\mathbf{x}})^{\alpha-e_2^4+e_4^4} +\alpha_1\hat{x}_3(\mathbf{x}-\hat{\mathbf{x}})^{\alpha-e_1^4} +\alpha_2\hat{x}_4(\mathbf{x}-\hat{\mathbf{x}})^{\alpha-e_2^4}\right)\notag\\
&+\frac{1}{\varepsilon}\left(\alpha_3(\mathbf{x}-\hat{\mathbf{x}})^{\alpha-e_3^4+e_4^4} -\alpha_4(\mathbf{x}-\hat{\mathbf{x}})^{\alpha-e_4^4+e_3^4} + \alpha_3\hat{x}_4(\mathbf{x}-\hat{\mathbf{x}})^{\alpha-e_3^4} - \alpha_4\hat{x}_3(\mathbf{x}-\hat{\mathbf{x}})^{\alpha-e_4^4}\right)\notag\\
&+\varepsilon\sum_{l=0}^{k-|\alpha|+1}\sum_{|\beta|=l}\frac{\mathrm{I}_{\{\beta_3+\beta_4=0\}}}{\beta_1!\beta_2!}\left(\alpha_3\frac{\partial^{\beta} E_1(\hat{\mathbf{y}})}{\partial\mathbf{y}^\beta}(\mathbf{x}-\hat{\mathbf{x}})^{\alpha+\beta-e_3^4} +\alpha_4\frac{\partial^{\beta} E_2(\hat{\mathbf{y}})}{\partial\mathbf{y}^\beta}(\mathbf{x}-\hat{\mathbf{x}})^{\alpha+\beta-e_4^4}\right)\notag\\
&+\varepsilon\alpha_3(\mathbf{x}-\hat{\mathbf{x}})^{\alpha-e_3^4}r_1^{k-|\alpha|+2}(\mathbf{y};\hat{\mathbf{y}}) +\varepsilon\alpha_4(\mathbf{x}-\hat{\mathbf{x}})^{\alpha-e_4^4}r_2^{k-|\alpha|+2}(\mathbf{y};\hat{\mathbf{y}}).
\end{align}

It is clear that, except for the last two terms involving the remainders $r_j^{k-|\alpha|+2}(\mathbf{y};\hat{\mathbf{y}})(j=1,2)$, the remaining terms can be represented as linear combinations of polynomials in $P_{\mathbf{x}}^{[k,\hat{\mathbf{x}}]}$. By collecting the coefficients of identical polynomials and combining like terms, the portion of (\ref{Assem01}) excluding the remainder terms can be expressed as the product of a coefficient vector and $\mathbf{x}^{[k,\hat{\mathbf{x}}]}$. Applying (\ref{Assem01}) to differentiate each polynomial in $P_{\mathbf{x}}^{[k,\hat{\mathbf{x}}]}$, assembling these coefficient vectors from each resulting equation, and then separating the $O(\frac{1}{\varepsilon})$ and $O(\varepsilon)$ components, we obtain two coefficient matrices $A_1^{[k]}(\hat{\mathbf{x}})$ and $A_0^{[k]}(\hat{\mathbf{x}})$. Consequently, the equation satisfied by the local linear extension variable $\mathbf{x}^{[k,\hat{\mathbf{x}}]}$ takes the form
\begin{equation}\label{LLES01}
  \frac{\mathrm{d}\mathbf{x}^{[k,\hat{\mathbf{x}}]}}{\mathrm{d}t} =\frac{1}{\varepsilon}A_1^{[k]}(\hat{\mathbf{x}})\mathbf{x}^{[k,\hat{\mathbf{x}}]} +\varepsilon A_0^{[k]}(\hat{\mathbf{x}})\mathbf{x}^{[k,\hat{\mathbf{x}}]} +\varepsilon\mathbf{R}^{[k]}(\mathbf{x}^{[k,\hat{\mathbf{x}}]};\hat{\mathbf{x}}).
\end{equation}
If the $i$-th row of (\ref{LLES01}) corresponds to the equation associated with the polynomial generated by the multi-index $\alpha$, that is, $i=\tau^{[k]}(\alpha)$, then (\ref{Assem01}) implies that the $i$-th component of $\mathbf{R}^{[k]}(\mathbf{x}^{[k,\hat{\mathbf{x}}]};\hat{\mathbf{x}})$ can be written as
\begin{equation}\label{LLES02}
  R_i^{[k]}(\mathbf{x}^{[k,\hat{\mathbf{x}}]};\hat{\mathbf{x}})=\alpha_3(\mathbf{x}-\hat{\mathbf{x}})^{\alpha-e_3^4}r_1^{k-|\alpha|+2}(\mathbf{y};\hat{\mathbf{y}}) +\alpha_4(\mathbf{x}-\hat{\mathbf{x}})^{\alpha-e_4^4}r_2^{k-|\alpha|+2}(\mathbf{y};\hat{\mathbf{y}}).
\end{equation}

Since the two matrix-valued functions $\hat{\mathbf{x}}\mapsto A_i^{[k]}(\hat{\mathbf{x}})(i=0,1)$ do not admit simple closed-form expressions, their construction is carried out algorithmically. The main idea is to assign values to the corresponding matrix entries element-wise. Taking the first term in (\ref{Assem01}), $\alpha_1(\mathbf{x}-\hat{\mathbf{x}})^{\alpha-e_1^4+e_3^4}$, as an example, this indicates that, for any multi-index $\alpha\in\mathbb{I}^{[k]}$, the coefficient $\alpha_1$ should be placed in the row $\tau^{[k]}(\alpha)$ and the column $\tau^{[k]}(\alpha-e_1^4+e_3^4)$ of the matrix $A_1^{[k]}$. We refer to \cite{QDZ2025} for a complete assembly procedure of the two matrices. We remark that the two matrices are independent of both $\varepsilon$ and the time step size employed in the numerical scheme. This implies that, once assembled, they can be reused directly under new parameter settings, even when $\varepsilon$ or the step size changes.

\begin{Definition}\label{LLES}
Given a positive integer $k$ and a reference point $\hat{\mathbf{x}}$, (\ref{LLES01}) is called the $k$-th order local linear extension system of (\ref{P}) at $\hat{\mathbf{x}}$ with respect to $\mathbf{x}^{[k,\hat{\mathbf{x}}]}$.
\end{Definition}

We now present the numerical scheme for solving the CPD. Let the interval $[0,T]$ be partitioned by equidistant nodes $\{t_n\}_{n=0}^N$, i.e., $0=t_0<t_1<\cdots< t_N=T$, with $\Delta t=t_n-t_{n-1}$ for $n=1,\cdots,N$. The time step used in the algorithm is $h=\Delta t/\varepsilon^\nu$.

We construct a numerical scheme for solving (\ref{P}) based on the local linear extension system (\ref{LLES01}). Omitting the remainder term $\mathbf{R}^{[k]}$ therein and choosing the reference point as the numerical solution $\hat{\mathbf{x}}=\mathbf{X}_n$, we denote by $\mathbf{X}^{[k,\mathbf{X}_n]}$ the $k$-th order local linear extension variable that satisfies the resulting equation. From Definition \ref{LLEV}, its value at $t=t_n/\varepsilon^\nu$ is given by $\mathbf{X}^{[k,\mathbf{X}_n]}(t_n/\varepsilon^\nu)=e_1^{D^{[k]}}$. Thus, the initial value problem satisfied by $\mathbf{X}^{[k,\mathbf{X}_n]}(t)$ on the interval $[t_n/\varepsilon^\nu,t_{n+1}/\varepsilon^\nu]$ is defined as
\begin{equation}\label{LLES-num-truncation}
\begin{aligned}
  &\frac{\mathrm{d}\mathbf{X}^{[k,\mathbf{X}_n]}}{\mathrm{d}t} =\frac{1}{\varepsilon} A_1^{[k]}(\mathbf{X}_n) \mathbf{X}^{[k,\mathbf{X}_n]} +\varepsilon A_0^{[k]}(\mathbf{X}_n)\mathbf{X}^{[k,\mathbf{X}_n]}, \quad \frac{t_n}{\varepsilon^\nu}\leq t\leq\frac{t_{n+1}}{\varepsilon^\nu},\\
  &\mathbf{X}^{[k,\mathbf{X}_n]}(t_n/\varepsilon^\nu)=e_1^{D^{[k]}}.
\end{aligned}
\end{equation}
Its exact solution at $t=t_{n+1}/\varepsilon^\nu$ is obtained by direct matrix exponentiation:
\begin{equation}\label{LLEEI01}
  \mathbf{X}^{[k,\mathbf{X}_n]}(t_{n+1}/\varepsilon^\nu)=\mathrm{e}^{\left(A_1^{[k]}(\mathbf{X}_n)/\varepsilon^{1+\nu} +\varepsilon^{1-\nu}A_0^{[k]}(\mathbf{X}_n)\right)\Delta t} e_1^{D^{[k]}}.
\end{equation}
Finally, by Definition \ref{LLEV} and (\ref{LLEV01}), the numerical solution of (\ref{P}) at $t=t_{n+1}/\varepsilon^\nu$ is defined as
\begin{equation}\label{LLEEI02}
  \mathbf{X}_{n+1}:=\Pi_{\mathbf{x}}\mathbf{X}^{[k,\mathbf{X}_n]}(t_{n+1}/\varepsilon^\nu)+\mathbf{X}_n.
\end{equation}

In summary, we present a complete description of this new EI. First, the matrices $A_1^{[k]}(\cdot)$ and $A_0^{[k]}(\cdot)$ are assembled as matrix-valued functions. At each time step, the numerical solution $\mathbf{X}_n$ from the previous step is substituted to obtain $A_1^{[k]}(\mathbf{X}_n)$ and $A_0^{[k]}(\mathbf{X}_n)$. Next, we compute (\ref{LLEEI01}) and (\ref{LLEEI02}) sequentially. Finally, the numerical solutions for the particle position $\mathbf{y}(t_{n+1}/\varepsilon^\nu)$ and velocity $\dot{\mathbf{y}}(t_{n+1}/\varepsilon^\nu)$ are given, respectively, by
\begin{equation}\label{LLEEI03}
  \mathbf{Y}_{n+1}:=\left[
  \begin{array}{cc} I_2 & \mathbf{0} \end{array}\right]\mathbf{X}_{n+1},\quad \dot{\mathbf{Y}}_{n+1}:=\varepsilon^{-1}\left[
  \begin{array}{cc} \mathbf{0} & I_2 \end{array}\right]\mathbf{X}_{n+1}.
\end{equation}

The scheme defined by (\ref{LLEEI01})–(\ref{LLEEI03}) is fully explicit at each time step and requires no iterative procedure. In the subsequent theoretical analysis, we will show that the convergence order of the numerical scheme increases with $k$. This indicates that the local linear extension EIs improve their order simply by enlarging the size of the matrices $A_1^ {[k]}$ and $A_0^{[k]}$. Consequently, the above procedure may be regarded as a unified high-precision solution framework for CPD, which can theoretically attain arbitrarily high convergence order without imposing any order conditions.

\section{Convergence analysis}\label{Sec3}
\subsection{Preliminary results}\label{Sec3-1}
This subsection presents some preliminary results for the convergence analysis. In the following analysis, $C$ always denotes a constant independent of both the time step and $\varepsilon$, and $\|\cdot\|$ denotes the standard Euclidean norm for matrices and vectors unless otherwise specified.

\begin{Lemma}\label{CPD2D-dxLemma}
Assume that $\mathbf{E}$ has bounded first derivatives, i.e. $\|\nabla\mathbf{E}\|_{L^{\infty}}\leq C$. If $\nu=0$, then for any $0\leq\hat{t}<t\leq T$, we have
\begin{equation}\label{CPD2D-dxLemma-R1}
  \|\mathbf{y}(t)-\mathbf{y}(\hat{t})\|+\varepsilon\|\dot{\mathbf{y}}(t)-\dot{\mathbf{y}}(\hat{t})\|\leq C\min\{\varepsilon,t-\hat{t}\}.
\end{equation}
If $\nu=1$, then for any $0\leq\hat{t}<t\leq T$ satisfying $t-\hat{t}>c_2\varepsilon$, where $c_2>0$ is independent of $\varepsilon$, we have
\begin{equation}\label{CPD2D-dxLemma-R2}
  \|\mathbf{y}(t/\varepsilon)-\mathbf{y}(\hat{t}/\varepsilon)\| +\varepsilon\|\dot{\mathbf{y}}(t/\varepsilon)-\dot{\mathbf{y}}(\hat{t}/\varepsilon)\|\leq C(t-\hat{t}).
\end{equation}
\end{Lemma}
\begin{Remark}
{As will be shown later, the error bound of the local linear extension EIs exhibits piecewise behavior as the step size varies. The notation $c_2$ corresponds to the second transition point, whose estimation and physical interpretation are elaborated in Section \ref{Sec3-3}.}
\end{Remark}

\begin{proof}
We divide the proof into three parts.

\textbf{The boundedness.} We first derive the following bounds:
\begin{equation*}
  \|\mathbf{y}(t/\varepsilon^\nu)\|+\|\dot{\mathbf{y}}(t/\varepsilon^\nu)\|\leq C, \quad t\in[0,T].
\end{equation*}
Applying the variation-of-constants formula to (\ref{P}) on the interval $\left[\hat{t}/\varepsilon^\nu,t/\varepsilon^\nu\right]$, we have
\begin{equation}\label{CPD2D-dxLemma-01}
  \mathbf{x}(t/\varepsilon^\nu)=\mathrm{e}^{\frac{1}{\varepsilon^{1+\nu}}A_1(t-\hat{t})}\mathbf{x}(\hat{t}/\varepsilon^\nu)+ \varepsilon\int_{\hat{t}/\varepsilon^\nu}^{t/\varepsilon^\nu}\mathrm{e}^{A_1\left(\frac{t}{\varepsilon^{1+\nu}}-\frac{s}{\varepsilon}\right)}\mathbf{f}(\mathbf{x}(s))\mathrm{d}s.
\end{equation}
Projecting (\ref{CPD2D-dxLemma-01}) onto the rows corresponding to the velocity components and employing the matrix and vector structure given in (\ref{Afx0}) yields
\begin{equation}\label{CPD2D-dxLemma-02}
  \varepsilon\dot{\mathbf{y}}(t/\varepsilon^\nu)=\mathrm{e}^{\frac{1}{\varepsilon^{1+\nu}}J(t-\hat{t})}\varepsilon\dot{\mathbf{y}}(\hat{t}/\varepsilon^\nu)+\varepsilon \int_{\hat{t}/\varepsilon^\nu}^{t/\varepsilon^\nu}\mathrm{e}^{J\left(\frac{t}{\varepsilon^{1+\nu}}-\frac{s}{\varepsilon}\right)}\mathbf{E}(\mathbf{y}(s))\mathrm{d}s.
\end{equation}
Since the eigenvalues of $J$ are $\pm\mathrm{i}$, the associated matrix exponential is bounded. Setting $\hat{t}=0$ and taking norms in (\ref{CPD2D-dxLemma-02}) leads to
\begin{align}\label{CPD2D-dxLemma-03}
  \varepsilon\|\dot{\mathbf{y}}(t/\varepsilon^\nu)\|\leq& C\varepsilon\|\dot{\mathbf{y}}_0\|+\varepsilon\int_0^{t/\varepsilon^\nu}C(\|\mathbf{E}(\mathbf{y}_0)\| +\|\nabla\mathbf{E}\|_{L^{\infty}}\|\mathbf{y}(s)-\mathbf{y}_0\|)\mathrm{d}s \notag\\
  \leq& C\varepsilon\|\dot{\mathbf{y}}_0\|+\varepsilon\int_0^{t/\varepsilon^\nu} C(1+\|\mathbf{y}(s)\|)\mathrm{d}s.
\end{align}
Similarly, projecting (\ref{CPD2D-dxLemma-01}) onto the rows corresponding to $\mathbf{y}$, we obtain
\begin{equation}\label{CPD2D-dxLemma-04}
  \|\mathbf{y}(t/\varepsilon^\nu)\|\leq C(\|\mathbf{y}_0\|+\varepsilon\|\dot{\mathbf{y}}_0\|)+\varepsilon\int_0^{t/\varepsilon^\nu} C(1+\|\mathbf{y}(s)\|)\mathrm{d}s.
\end{equation}
Combining (\ref{CPD2D-dxLemma-03}) and (\ref{CPD2D-dxLemma-04}) and applying Grönwall's inequality in integral form yields
\begin{equation}\label{CPD2D-dxLemma-05}
  \|\mathbf{y}(t/\varepsilon^\nu)\| +\varepsilon\|\dot{\mathbf{y}}(t/\varepsilon^\nu)\| \leq C(\|\mathbf{y}_0\|+\varepsilon\|\dot{\mathbf{y}}_0\|) +C\varepsilon^{1-\nu}T.
\end{equation}
From (\ref{CPD2D-dxLemma-05}), the boundedness of $\mathbf{y}$ follows immediately. Substituting this result into (\ref{CPD2D-dxLemma-03}) then yields the boundedness of $\dot{\mathbf{y}}$ for the case $\nu=0$. For the case where $\nu=1$, applying integration by parts to the integral term in (\ref{CPD2D-dxLemma-02}) gives
\begin{align}\label{CPD2D-dxLemma-06}
  \int_{\hat{t}/\varepsilon^\nu}^{t/\varepsilon^\nu}\mathrm{e}^{J\left(\frac{t}{\varepsilon^{1+\nu}}-\frac{s}{\varepsilon}\right)}\mathbf{E}(\mathbf{y}(s))\mathrm{d}s =&\varepsilon J\left(\mathbf{E}(\mathbf{y}(t/\varepsilon^\nu))-\mathrm{e}^{\frac{1}{\varepsilon^{1+\nu}}J(t-\hat{t})}\mathbf{E}(\mathbf{y}(\hat{t}/\varepsilon^\nu))\right) \notag\\
  &-\varepsilon J\int_{\hat{t}/\varepsilon^\nu}^{t/\varepsilon^\nu}\mathrm{e}^{J\left(\frac{t}{\varepsilon^{1+\nu}}-\frac{s}{\varepsilon}\right)}\nabla\mathbf{E}(\mathbf{y}(s))\dot{\mathbf{y}}(s)\mathrm{d}s.
\end{align}
Substituting (\ref{CPD2D-dxLemma-06}) into (\ref{CPD2D-dxLemma-02}) and estimating with $\hat{t}=0$ and $\nu=1$, we have
\begin{equation*}
  \|\dot{\mathbf{y}}(t/\varepsilon^\nu)\|\leq C\|\dot{\mathbf{y}}_0\|+ C\varepsilon +C\varepsilon\int_{0}^{t/\varepsilon^\nu}\|\dot{\mathbf{y}}(s)\|\mathrm{d}s.
\end{equation*}
Applying the Grönwall's inequality again, we establish the boundedness of $\dot{\mathbf{y}}$ for $\nu=1$.

\textbf{The proof for velocity components in (\ref{CPD2D-dxLemma-R1}) and (\ref{CPD2D-dxLemma-R2}).} We first subtract $\varepsilon\dot{\mathbf{y}}(\hat{t}/\varepsilon^\nu)$ from both sides of (\ref{CPD2D-dxLemma-02}).
For $\nu=1$, the boundedness of $\dot{\mathbf{y}}$, together with the condition $c_2\varepsilon<t-\hat{t}$, directly yields the estimate for $\dot{\mathbf{y}}$ stated in (\ref{CPD2D-dxLemma-R2}).

For $\nu=0$, we again substitute (\ref{CPD2D-dxLemma-06}) into (\ref{CPD2D-dxLemma-02}), obtaining
\begin{align}\label{CPD2D-dxLemma-07}
  \varepsilon(\dot{\mathbf{y}}(t)-\dot{\mathbf{y}}(\hat{t}))=
  \varepsilon\left(\mathrm{e}^{\frac{1}{\varepsilon} J(t-\hat{t})}-I_2\right)\dot{\mathbf{y}}(\hat{t}) +&\varepsilon^2 J\left(\mathbf{E}(\mathbf{y}(t))-\mathrm{e}^{\frac{1}{\varepsilon}J(t-\hat{t})}\mathbf{E}(\mathbf{y}(\hat{t}))\right) \notag\\
  -&\varepsilon^2\int_{\hat{t}}^{t}\mathrm{e}^{\frac{1}{\varepsilon}J(t-s)}J\nabla\mathbf{E}(\mathbf{y}(s))\dot{\mathbf{y}}(s)\mathrm{d}s.
\end{align}
Observe that the exponential function satisfies the identity $\mathrm{e}^z-1=z\varphi_1(z)$, where $\varphi_1(z)=\frac{e^z-1}{z}$ is an analytic function, and its restriction $\varphi_1(\mathrm{i}x)$ for $x\in\mathbb{R}\setminus\{0\}$ is bounded. Thus, we have
\begin{equation}\label{CPD2D-dxLemma-08}
  \left\|\mathrm{e}^{\frac{1}{\varepsilon} J(t-\hat{t})}-I_2\right\|\leq C\min\left\{1,\frac{t-\hat{t}}{\varepsilon}\right\}.
\end{equation}
Substituting this bound and estimating the first term in (\ref{CPD2D-dxLemma-07}) immediately yields the estimate for $\dot{\mathbf{y}}$ given in (\ref{CPD2D-dxLemma-R1}).

\textbf{The proof for position components in (\ref{CPD2D-dxLemma-R1}) and (\ref{CPD2D-dxLemma-R2}).} We integrate both sides of (\ref{CPD2D-dxLemma-02}) directly over the interval $[\hat{t},t]$ and substitute (\ref{CPD2D-dxLemma-06}):
\begin{align}\label{CPD2D-dxLemma-09}
  &\mathbf{y}(t/\varepsilon^\nu) -\mathbf{y}(\hat{t}/\varepsilon^\nu) \notag\\
  &=\frac{1}{\varepsilon^\nu}\int_{\hat{t}}^t\mathrm{e}^{\frac{1}{\varepsilon^{1+\nu}}J(s-\hat{t})}\dot{\mathbf{y}}(\hat{t}/\varepsilon^\nu)\mathrm{d}s +\frac{1}{\varepsilon^\nu}\int_{\hat{t}}^t\int_{\hat{t}/\varepsilon^\nu}^{s/\varepsilon^\nu}\mathrm{e}^{J\left(\frac{s}{\varepsilon^{1+\nu}}-\frac{\sigma}{\varepsilon}\right)}\mathbf{E}(\mathbf{y}(\sigma))\mathrm{d}\sigma\mathrm{d}s \notag\\
  &=\varepsilon J\left(I_2-\mathrm{e}^{\frac{1}{\varepsilon^{1+\nu}}J(t-\hat{t})}\right)\dot{\mathbf{y}}(\hat{t}/\varepsilon^\nu) +\varepsilon J\int_{\hat{t}/\varepsilon^\nu}^{t/\varepsilon^\nu}\left(\mathbf{E}(\mathbf{y}(s))-\mathrm{e}^{\frac{1}{\varepsilon}J\left(s-\frac{\hat{t}}{\varepsilon^\nu}\right)}\mathbf{E}(\mathbf{y}(\hat{t}/\varepsilon^\nu))\right)\mathrm{d}s \notag\\
  &\quad -\varepsilon J\int_{\hat{t}/\varepsilon^\nu}^{t/\varepsilon^\nu}\int_{\hat{t}/\varepsilon^\nu}^{s}\mathrm{e}^{\frac{1}{\varepsilon}J(s-\sigma)}\nabla\mathbf{E}(\mathbf{y}(\sigma))\dot{\mathbf{y}}(\sigma)\mathrm{d}\sigma\mathrm{d}s.
\end{align}
When $\nu=0$, each of the two integrals in (\ref{CPD2D-dxLemma-09}) is bounded by $O(\varepsilon(t-\hat{t}))$. The first term is estimated using (\ref{CPD2D-dxLemma-08}), thereby establishing the stated conclusion for $\mathbf{y}$ in (\ref{CPD2D-dxLemma-R1}). When $\nu=1$, the condition $c_2\varepsilon<t-\hat{t}$ implies that the first two terms in (\ref{CPD2D-dxLemma-09}) are bounded by $O(t-\hat{t})$. For the third term, we interchange the order of integration:
\begin{align*}
  \int_{\hat{t}/\varepsilon}^{t/\varepsilon}\int_{\hat{t}/\varepsilon}^{s}\mathrm{e}^{\frac{1}{\varepsilon}J(s-\sigma)}\nabla\mathbf{E}(\mathbf{y}(\sigma))\dot{\mathbf{y}}(\sigma)\mathrm{d}\sigma\mathrm{d}s =&\int_{\hat{t}/\varepsilon}^{t/\varepsilon}\left(\int_{\sigma}^{t/\varepsilon}\mathrm{e}^{\frac{1}{\varepsilon}J(s-\sigma)}\mathrm{d}s\right)\nabla\mathbf{E}(\mathbf{y}(\sigma))\dot{\mathbf{y}}(\sigma)\mathrm{d}\sigma \\
  =&\varepsilon J\int_{\hat{t}/\varepsilon}^{t/\varepsilon}\left(I_2-\mathrm{e}^{\frac{1}{\varepsilon}J\left(\frac{t}{\varepsilon}-\sigma\right)}\right)\nabla\mathbf{E}(\mathbf{y}(\sigma))\dot{\mathbf{y}}(\sigma)\mathrm{d}\sigma.
\end{align*}
Thus, this double integral can be bounded by $O(t-\hat{t})$. Substituting this estimate into (\ref{CPD2D-dxLemma-09}), we obtain the conclusion for $\mathbf{y}$ in (\ref{CPD2D-dxLemma-R2}).
\end{proof}

The algebraic properties of the eigenvalues associated with the linear part of (\ref{LLES01}), particularly those of the dominant matrix $A_1^{[k]}$, essentially determine the performance of the EI. We next examine the block structure, similarity transformation, and spectral properties of $A_1^{[k]}(\hat{\mathbf{x}})$. The proofs are somewhat technical and are omitted here. A complete algebraic analysis is presented in \cite{QDZ2025}.

\begin{Lemma}\label{LemA1}
There exists a nonsingular matrix $S^{[k]}(\hat{\mathbf{x}})$ whose entries, as well as those of its inverse, are polynomial functions of $\hat{\mathbf{x}}$, with the respective block structures
\begin{equation}\label{LemA1-R1}
  S^{[k]}(\hat{\mathbf{x}})=
  \left[\begin{array}{ccc}
          1 &  & \\
          -\hat{\mathbf{x}} & I_4 & \\
          * & * & *
        \end{array}\right],\quad
  \left(S^{[k]}(\hat{\mathbf{x}})\right)^{-1}=
  \left[\begin{array}{ccc}
          1 &  & \\
          \hat{\mathbf{x}} & I_4 & \\
          * & * & *
        \end{array}\right],
\end{equation}
such that, under the similarity transformation induced by $S^{[k]}(\hat{\mathbf{x}})$, $A_1^{[k]}(\hat{\mathbf{x}})$ is transformed to $A_1^{[k]}(\mathbf{0})$, which is block diagonal with the structure
\begin{equation}\label{LemA1-R2}
  A_1^{[k]}(\mathbf{0})=\left(S^{[k]}(\hat{\mathbf{x}})\right)^{-1}A_1^{[k]}(\hat{\mathbf{x}})S^ {[k]}(\hat{\mathbf{x}})=\operatorname{diag}\{0,A_1,*\}.
\end{equation}
\end{Lemma}

Lemma \ref{LemA1} shows that the structurally simpler matrix $A_1^{[k]}( \mathbf{0})$ preserves the spectral properties of $A_1^{[k]}(\hat{\mathbf{x}})$. Based on Lemma \ref{LemA1}, we can derive the following result.
\begin{Lemma}\label{LELem}
If $A_1$ is defined by (\ref{Afx0}) and $A_1^{[k]}(\hat{\mathbf{x}})$ is generated in (\ref{LLES01}), then $A_1^{[k]}(\hat{\mathbf{x}})$ is diagonalizable and all its eigenvalues have zero real part.
\end{Lemma}
This lemma directly yields the following boundedness property of the matrix exponential, which plays a fundamental role in the convergence analysis.
\begin{Lemma}\label{LemAA2}
Let $\hat{\mathbf{x}}\in\mathbb{R}^4$ be a reference point contained in a bounded domain independent of $\varepsilon$, and let $A_1^{[k]}(\hat{\mathbf{x}})$ be given by the local linear extension system (\ref{LLES01}). Then
\begin{equation*}
  \left\|\mathrm{e}^{\frac{1}{\varepsilon}A_1^{[k]}(\hat{\mathbf{x}})t}\right\|\leq C
\end{equation*}
holds uniformly for all $\varepsilon$ and $t$.
\end{Lemma}

Finally, we introduce two auxiliary extension systems. For notational simplicity, let $\mathbf{x}_n(n=0,\cdots,N)$ denote the solution $\mathbf{x}(t_n/\varepsilon^\nu)$ of (\ref{P}). Consider the extension variable at $\mathbf{x}_n(n=0,\cdots,N-1)$, denoted by $\mathbf{x}^{[k,\mathbf{x}_n]}$, which satisfies the local linear extension system
\begin{equation}\label{LLES-analy}
\begin{aligned}
  &\frac{\mathrm{d}\mathbf{x}^{[k,\mathbf{x}_n]}}{\mathrm{d}t} =\frac{1}{\varepsilon} A_1^{[k]}(\mathbf{x}_n)\mathbf{x}^{[k,\mathbf{x}_n]} +\varepsilon A_0^{[k]}(\mathbf{x}_n)\mathbf{x}^{[k,\mathbf{x}_n]} +\varepsilon\mathbf{R}^{[k]}(\mathbf{x}^{[k,\mathbf{x}_n]}), \quad \frac{t_n}{\varepsilon^\nu}\leq t\leq \frac{t_{n+1}}{\varepsilon^\nu},\\
  &\mathbf{x}^{[k,\mathbf{x}_n]}(t_n/\varepsilon^\nu)=e_1^{D^{[k]}},
\end{aligned}
\end{equation}
where for brevity, we introduce $\mathbf{R}^{[k]}(\mathbf{x}^{[k,\mathbf{x}_n]}):=\mathbf{R}^{[k]}(\mathbf{x}^{[k,\mathbf{x}_n]};\mathbf{x}_n)$. According to the ordering $\tau^{[k]}$, we have $\Pi_{\mathbf{x}}\mathbf{x}^{[k,\mathbf{x}_n]}(t/\varepsilon^\nu)=\mathbf{x}(t/\varepsilon^\nu)-\mathbf{x}_n$. Specially, the solution of (\ref{P}) at $t_{n+1}/\varepsilon^\nu$ can be recovered by:
\begin{equation*}
  \mathbf{x} _{n+1}=\Pi_{\mathbf{x}}\mathbf{x}^{[k,\mathbf{x}_n]}(t_{n+1}/\varepsilon^\nu)+\mathbf{x}_n,\quad n=0,\cdots,N-1.
\end{equation*}
The truncated system for (\ref{LLES-analy}) is defined as follows:
\begin{equation}\label{LLES-analy-truncation}
\begin{aligned}
  &\frac{\mathrm{d}\tilde{\mathbf{x}}^{[k,\mathbf{x}_n]}}{\mathrm{d}t} =\frac{1}{\varepsilon}A_1^{[k]}(\mathbf{x}_n)\tilde{\mathbf{x}}^{[k,\mathbf{x}_n]} +\varepsilon A_0^{[k]}(\mathbf{x}_n)\tilde{\mathbf{x}}^{[k,\mathbf{x}_n]}, \quad \frac{t_n}{\varepsilon^\nu}\leq t\leq \frac{t_{n+1}}{\varepsilon^\nu},\\
  &\tilde{\mathbf{x}}^{[k,\mathbf{x}_n]}(t_n/\varepsilon^\nu)=e_1^{D^{[k]}}.
\end{aligned}
\end{equation}
Similar to (\ref{LLES-num-truncation}), (\ref{LLES-analy-truncation}) can be solved exactly. Further, we denote the projection of its solution onto the state space by $\tilde{\mathbf{x}}_{n+1}:=\Pi_{\mathbf{x}}\tilde{\mathbf{x}}^{[k,\mathbf{x}_n]}(t_{n+1}/\varepsilon^\nu)+\mathbf{x}_n$.

\subsection{\texorpdfstring{Improved uniform accuracy when $\nu=0$}{Improved uniform accuracy when nu=0}}\label{Sec3-2}
\begin{Theorem}\label{Thm-CPD2D-HomoB}
Assume that the electric field function $\mathbf{E}(\mathbf{y})$ has Lipschitz-continuous derivatives up to order $k$. Let $\mathbf{Y}_n$ and $\dot{\mathbf{Y}}_n$ be the numerical solutions obtained from (\ref{LLEEI01})-(\ref{LLEEI03}). Then the following numerical error estimate holds for $\nu=0$:
\begin{equation*}
\|\mathbf{y}(t_n)-\mathbf{Y}_n\|+\varepsilon\|\dot{\mathbf{y}}(t_n)-\dot{\mathbf{Y}}_n\| \leq \min\{C\varepsilon^{k+2},C\varepsilon \Delta t^{k+1}\},\quad n=0,\cdots,N.
\end{equation*}
\end{Theorem}

\begin{proof}
Applying the variation-of-constants formula to (\ref{LLES-analy}) and (\ref{LLES-analy-truncation}) yields, respectively,
\begin{equation}\label{Thm3-3-01}
	\begin{split}
  \mathbf{x}^{[k,\mathbf{x}_n]}(t)=&\mathrm{e}^{\frac{1}{\varepsilon}A_1^{[k]}(\mathbf{x}_n)(t-t_n)}e_1^{D^{[k]}} +\int_{t_n}^t\varepsilon\mathrm{e}^{\frac{1}{\varepsilon}A_1^{[k]}(\mathbf{x}_n)(t-s)}A_0^{[k]}(\mathbf{x}_n)\mathbf{x}^{[k,\mathbf{x}_n]}(s)\mathrm{d}s \\ +&\int_{t_n}^t\varepsilon\mathrm{e}^{\frac{1}{\varepsilon}A_1^{[k]}(\mathbf{x}_n)(t-s)}\mathbf{R}^{[k]}(\mathbf{x}^{[k,\mathbf{x}_n]}(s))\mathrm{d}s,
  \end{split}
\end{equation}
and
\begin{equation}\label{Lem3-6-01}
  \tilde{\mathbf{x}}^{[k,\mathbf{x}_n]}(t)=\mathrm{e}^{\frac{1}{\varepsilon}A_1^{[k]}(\mathbf{x}_n)(t-t_n)}e_1^{D^{[k]}} + \int_{t_n}^{t}\varepsilon\mathrm{e}^{\frac{1}{\varepsilon}A_1^{[k]}(\mathbf{x}_n)(t-s)}A_0^{[k]}(\mathbf{x}_n)\tilde{\mathbf{x}}^{[k,\mathbf{x}_n]}(s)\mathrm{d}s,
\end{equation}
Subtracting (\ref{Lem3-6-01}) from (\ref{Thm3-3-01}) gives
\begin{align}\label{Thm3-3-03}
  \mathcal{J}_1^{[k]}(t):=& \mathbf{x}^{[k,\mathbf{x}_n]}(t)-\tilde{\mathbf{x}}^{[k,\mathbf{x}_n]}(t)\notag\\ =&\int_{t_n}^{t}\varepsilon\mathrm{e}^{\frac{1}{\varepsilon}A_1^{[k]}(\mathbf{x}_n)(t-s)}A_0^{[k]}(\mathbf{x}_n) \left(\mathbf{x}^{[k,\mathbf{x}_n]}(s)-\tilde{\mathbf{x}}^{[k,\mathbf{x}_n]}(s)\right)\mathrm{d}s \notag\\
  &+\int_{t_n}^{t}\varepsilon\mathrm{e}^{\frac{1}{\varepsilon}A_1^{[k]}(\mathbf{x}_n)(t-s)} \mathbf{R}^{[k]}(\mathbf{x}^{[k,\mathbf{x}_n]}(s))\mathrm{d}s.
\end{align}
According to (\ref{CPD2D-dxLemma-R1}) and (\ref{ParDiff04}), $\mathbf{r}^j(j=2,\cdots,k+1)$ satisfies the following estimate
\begin{equation}\label{Thm3-3-04}
  \left\|\mathbf{r}^j\left(\mathbf{x}(s); \mathbf{x}_n\right)\right\|\leq \min\{C\varepsilon^j,C\Delta t^j\}.
\end{equation}
Using (\ref{CPD2D-dxLemma-R1}) again and noting that $\mathbf{x}=[\mathbf{y}^\top,\varepsilon\dot{\mathbf{y}}^\top]^\top$, we obtain an upper bound for the truncation term $\mathbf{R} ^{[k]}$ (see (\ref{LLES02})):
\begin{equation}\label{Thm3-3-02}
  \|\mathbf{R} ^{[k]}({\mathbf{x}^{[k,\mathbf{x}_n]}}(s))\|\leq \min\{C\varepsilon^{k+1},C\Delta t^{k+1}\}.
\end{equation}
Taking norms in (\ref{Thm3-3-03}) and substituting (\ref{Thm3-3-02}) leads to
\begin{equation*}
  \|\mathcal{J}_1^{[k]}(t_{n+1})\| \leq C \int_{t_n}^{t_{n+1}}\varepsilon\|\mathbf{x}^{[k,\mathbf{x}_n]}(s)-\tilde{\mathbf{x}}^{[k,\mathbf{x}_n]}(s)\| \mathrm{d}s + \varepsilon \Delta t\min\{C\varepsilon^{k+1},C\Delta t^{k+1}\}.
\end{equation*}
Finally, applying Grönwall's inequality in integral form yields the local error estimate
\begin{equation}\label{Thm3-3-05}
  \|\mathcal{J}_{1}^{[k]}(t_{n+1})\| \leq \min\{C\varepsilon^{k+2} \Delta t,C\varepsilon \Delta t^{k+2}\}.
\end{equation}

We next proceed to the global error estimate. Applying the nonsingular matrix given in Lemma \ref{LemA1} to perform the variable substitutions $\tilde{\mathbf{x}} ^{[k,\mathbf{0}]}=\left(S^{[k]}(\mathbf{x}_n)\right)^{-1}\tilde{\mathbf{x}}^{[k,\mathbf{x}_n]}$ and $\mathbf{X}^{[k,\mathbf{0}]}=\left(S^{[k]}(\mathbf{X}_n)\right)^{-1}\mathbf{X}^{[k,\mathbf{X}_n]}$ for (\ref{LLES-analy-truncation}) and (\ref{LLES-num-truncation}), respectively, we obtain
\begin{equation*}
  \frac{\mathrm{d}\tilde{\mathbf{x}}^{[k,\mathbf{0}]}}{\mathrm{d}t} =\frac{1}{\varepsilon}A_1^{[k]}(\mathbf{0})\tilde{\mathbf{x}}^{[k,\mathbf{0}]} +\varepsilon\bar{A}_0^{[k]}(\mathbf{x}_n)\tilde{\mathbf{x}}^{[k,\mathbf{0}]},\qquad
  \tilde{\mathbf{x}}^{[k,\mathbf{0}]}(t_n)=\left(S^{[k]}(\mathbf{x}_n)\right)^{-1}e_1^{D^{[k]}},
\end{equation*}
and
\begin{equation*}
  \frac{\mathrm{d}\mathbf{X}^{[k,\mathbf{0}]}}{\mathrm{d}t} =\frac{1}{\varepsilon} A_1^{[k]}(\mathbf{0}) \mathbf{X}^{[k,\mathbf{0}]} +\varepsilon\bar{A}_0^{[k]}(\mathbf{X}_n)\mathbf{X}^{[k,\mathbf{0}]}, \qquad
  \mathbf{X}^{[k,\mathbf{0}]}(t_n)=\left(S^{[k]}(\mathbf{X}_n)\right)^{-1}e_1^{D^{[k]}},
\end{equation*}
where $\bar{A}_0^{[k]}(\hat{\mathbf{x}})=\left(S^{[k]}(\hat{\mathbf{x}})\right)^{-1}A_0^{[k]}(\hat{\mathbf{x}})S^{[k]}(\hat{\mathbf{x}})$. Applying the variation-of-constants formula and subtracting the two equations yields
\begin{equation}\label{Thm3-3-06}
  \tilde{\mathbf{x}}^{[k,\mathbf{0}]}(t)-\mathbf{X}^{[k,\mathbf{0}]}(t) =\mathrm{e}^{\frac{1}{\varepsilon}A_1^{[k]}(\mathbf{0})(t-t_n)}\left(\left(S^{[k]}(\mathbf{x}_n)\right)^{-1}-\left(S^{[k]}(\mathbf{X}_n)\right)^{-1}\right)e_1^{D^{[k]}} +\mathcal{J}_{2}^{[k]}(t),
\end{equation}
where
\begin{align}\label{Thm3-3-07}
  \mathcal{J}_{2}^{[k]}(t)=&\int_{t_n}^{t}\varepsilon\mathrm{e}^{\frac{1}{\varepsilon}A_1^{[k]}(\mathbf{0})(t-s)} \left(\bar{A}_0^{[k]}(\mathbf{x}_n)\tilde{\mathbf{x}}^{[k,\mathbf{0}]}(s) -\bar{A}_0^{[k]}(\mathbf{X}_n)\mathbf{X}^{[k,\mathbf{0}]}(s)\right)\mathrm{d}s \notag\\
  =&\int_{t_n}^t\varepsilon\mathrm{e}^{\frac{1}{\varepsilon}A_1^{[k]}(\mathbf{0})(t-s)} \left(\bar{A}_0^{[k]}(\mathbf{x}_n)-\bar{A}_0^{[k]}(\mathbf{X}_n)\right)\tilde{\mathbf{x}}^{[k,\mathbf{0}]}(s) \mathrm{d}s \notag\\
  &+\int_{t_n}^t\varepsilon\mathrm{e}^{\frac{1}{\varepsilon}A_1^{[k]}(\mathbf{0})(t-s)} \bar{A}_0^{[k]}(\mathbf{X}_n)(\tilde{\mathbf{x}}^{[k,\mathbf{0}]}(s)-\mathbf{X}^{[k,\mathbf{0}]}(s))\mathrm{d}s.
\end{align}
We assert the Lipschitz continuity of $\bar{A}_0^{[k]}(\cdot)$. Indeed, Lemma \ref{LemA1} directly implies that $S^{[k]}(\cdot)$ and its inverse are Lipschitz, while the Lipschitz property of $A_0^{[k]}(\cdot)$ follows from its definition in terms of the partial derivatives of $\mathbf{E}$. Consequently, from (\ref{Thm3-3-06}) and (\ref{Thm3-3-07}) we have
\begin{equation*}
  \|\tilde{\mathbf{x}}^{[k,\mathbf{0}]}(t)-\mathbf{X}^{[k,\mathbf{0}]}(t)\|\leq C\|\mathbf{x}_n-\mathbf{X}_n\| + C\varepsilon\int_{t_n}^t \|\tilde{\mathbf{x}}^{[k,\mathbf{0}]}(s)-\mathbf{X}^{[k,\mathbf{0}]}(s)\|\mathrm{d}s.
\end{equation*}
Applying Grönwall's inequality in integral form gives the estimate
\begin{equation}\label{Thm3-3-08}
  \|\tilde{\mathbf{x}}^{[k,\mathbf{0}]}(t)-\mathbf{X}^{[k,\mathbf{0}]}(t)\|\leq C\mathrm{e}^{C\varepsilon(t-t_n)}\|\mathbf{x}_n-\mathbf{X}_n\|,\quad t_n\leq t\leq t_{n+1}.
\end{equation}
Using (\ref{Thm3-3-08}) to bound the second term in (\ref{Thm3-3-07}) yields an estimate for $\mathcal{J}_{2}^{[k]}(t_{n+1})$:
\begin{equation}\label{Thm3-3-09}
  \|\mathcal{J}_{2}^{[k]}(t_{n+1})\|\leq C\varepsilon \Delta t\|\mathbf{x}_n-\mathbf{X}_n\|.
\end{equation}
Note that Definition \ref{LLEV} implies that when $\hat{\mathbf{x}}=\mathbf{0}$, applying $\Pi_{\mathbf{x}}$ to the extension variable recovers the original variable; i.e., $\tilde{\mathbf{x}}_{n+1}=\Pi_{\mathbf{x}}\tilde{\mathbf{x}}^{[k,\mathbf{0}]}(t_{n+1})$ and $\mathbf{X}_{n+1}=\Pi_{\mathbf{x}}\mathbf{X}^{[k,\mathbf{0}]}(t_{n+1})$. Projecting both sides of (\ref{Thm3-3-06}) and invoking the similarity relation, along with the diagonal structure of $A_1^{[k]}(\mathbf{0})$ specified in (\ref{LemA1-R2}), we obtain by direct computation the equality
\begin{equation}\label{Thm3-3-10}
  \tilde{\mathbf{x}}_{n+1}-\mathbf{X}_{n+1}=\mathrm{e}^{\frac{1}{\varepsilon}A_1\Delta t}(\mathbf{x}_n-\mathbf{X}_n)+\Pi_{\mathbf{x}}\mathcal{J}_{2}^{[k]}(t_{n+1}).
\end{equation}
Then, projecting (\ref{Thm3-3-03}) and adding it to (\ref{Thm3-3-10}) gives the recurrence relation
\begin{equation*}
  \mathbf{x}_{n+1}-\mathbf{X}_{n+1}=\mathrm{e}^{\frac{1}{\varepsilon}A_1\Delta t}(\mathbf{x}_n-\mathbf{X}_n) +\Pi_{\mathbf{x}}\left(\mathcal{J}_{1}^{[k]}(t_{n+1})+\mathcal{J}_{2}^{[k]}(t_{n+1})\right).
\end{equation*}
Using the initial condition $\mathbf{x}_0=\mathbf{X}_0$ and solving this recurrence leads to
\begin{equation}\label{Thm3-3-11}
  \mathbf{x}_n-\mathbf{X}_n=\sum_{j=1}^{n}\mathrm{e}^{\frac{n-j}{\varepsilon}A_1\Delta t}\Pi_{\mathbf{x}}\left(\mathcal{J}_{1}^{[k]}(t_j)+\mathcal{J}_{2}^{[k]}(t_j)\right).
\end{equation}
Substituting the bounds on $\mathcal{J}_{i}^{[k]} (i=1,2)$ from (\ref{Thm3-3-05}) and (\ref{Thm3-3-09}) into (\ref{Thm3-3-11}) and applying Lemma \ref{LemAA2} yields the inequality
\begin{equation*}
  \|\mathbf{x}_n-\mathbf{X}_n\|\leq \sum_{j=1}^{n}\left(C\varepsilon \Delta t\|\mathbf{x}_{j-1}-\mathbf{X}_{j-1}\|+\min\{C\varepsilon^{k+2} \Delta t,C\varepsilon \Delta t^{k+2}\}\right).
\end{equation*}
Finally, by the discrete Grönwall's inequality, we arrive at the desired error estimate.
\end{proof}

From the viewpoint that different step sizes lead to distinct convergence behaviors of the numerical scheme, the error estimate in Theorem \ref{Thm-CPD2D-HomoB} can be decomposed into two sub-results according to the relation between $h$ and $\varepsilon$:
\begin{equation}\label{Thm-CPD2D-HomoB-R1}
  \|\mathbf{y}(t_n)-\mathbf{y}_n\|+\varepsilon\|\dot{\mathbf{y}}(t_n)-\dot{\mathbf{Y}}_n\| \leq
  \begin{cases}
      C\varepsilon \Delta t^{k+1},\quad h=\Delta t<c_1\varepsilon, \\
      C\varepsilon^{k+2},\quad h=\Delta t>c_1\varepsilon,
  \end{cases}
\end{equation}
where $c_1$ is a positive constant independent of both $\Delta t$ and $\varepsilon$. In other words, the theoretical error of the local linear extension exponential integrator displays a piecewise behavior for either fixed $\Delta t$ or fixed $\varepsilon$, with the transition occurring when $\Delta t$ and $\varepsilon$ are of comparable magnitude. The value of $c_1$ is generally associated with the cyclotron period of the charged particle, which is approximately $2\pi\varepsilon$ {(i.e., $c_1\approx 2\pi$). Here, we choose the empirical threshold $\tilde{c}_1=\frac{\pi}{2}$ in place of $c_1$, corresponding to four time nodes within one cyclotron period, since a sufficient number of time nodes within a single period is often required to resolve the highest frequency in the CPD and thereby attain $(k+1)$-th order convergence in $\Delta t$.}

\subsection{\texorpdfstring{Uniform accuracy when $\nu=1$}{Uniform accuracy when nu=1}}\label{Sec3-3}
We present a uniform convergence result of the proposed numerical scheme for the case $\nu=1$, where the long-term behavior of the charged particle is considered. The actual simulation time and step size are $T/\varepsilon$ and $h=\Delta t/\varepsilon$, respectively.
\begin{Theorem}\label{Thm-CPD2D-HomoB-LongTime}
Let $\nu=1$ and $\Delta t>c_2\varepsilon$, where $c_2$ is a positive constant independent of both $\varepsilon$ and $\Delta t$. Assume that $\mathbf{E}(\mathbf{y})$ has Lipschitz continuous derivatives up to order $k$. Then the following numerical error estimate holds:
\begin{equation*}
\|\mathbf{y}(t_n/\varepsilon)-\mathbf{Y}_n\|+\varepsilon\|\dot{\mathbf{y}}(t_n/\varepsilon)-\dot{\mathbf{Y}}_n\| \leq C\Delta t^{k+1},\quad n=0,\cdots,N.
\end{equation*}
\end{Theorem}

\begin{proof}
The proof is analogous to that of Theorem \ref{Thm-CPD2D-HomoB}. We therefore only highlight the differences. First, the time integration interval in (\ref{Thm3-3-03}) is modified to $s\in[t_n/\varepsilon,t/\varepsilon](t_n<t\leq t_{n+1})$. Under the parameter condition $\Delta t>c_2\varepsilon$, by applying (\ref{CPD2D-dxLemma-R2}) and (\ref{ParDiff04}), the $j$-th order Taylor remainder satisfies $\big\|\mathbf{r}^j(\mathbf{x}(s); \mathbf{x}_n)\big\|\leq C\Delta t^j$. Then, from (\ref{LLES02}), it follows that $\big\|\mathbf{R}^{[k]}({\mathbf{x}^{[k,\mathbf{x}_n]}}(s))\big\|\leq C\Delta t^{k+1}$. Consequently, the integral containing the truncation term in (\ref{Thm3-3-03}) is re-estimated as
\begin{equation*}
  \left\|\int_{t_n/\varepsilon}^{t/\varepsilon}\varepsilon\mathrm{e}^{\frac{1}{\varepsilon}A_1^{[k]}(\mathbf{x}_n)\left(\frac{t}{\varepsilon}-s\right)} \mathbf{R}^{[k]}(\mathbf{x}^{[k,\mathbf{x}_n]}(s))\mathrm{d}s\right\|\leq C\Delta t^{k+2}.
\end{equation*}
Thus, the integral inequality for the local error becomes
\begin{equation*}
  \|\mathcal{J}_1^{[k]}(t_{n+1})\| \leq C \int_{t_n/\varepsilon}^{t_{n+1}/\varepsilon} \varepsilon\|\mathbf{x}^{[k,\mathbf{x}_n]}(s)-\tilde{\mathbf{x}}^{[k,\mathbf{x}_n]}(s)\| \mathrm{d}s + C\Delta t^{k+2}.
\end{equation*}
Applying Grönwall's inequality yields the new local truncation estimate $$\|\mathcal{J}_1^{[k]}(t_{n+1})\|\leq C\Delta t^{k+1}.$$

In the global error estimate, the only modification again concerns the integration interval, which changes the estimate for $\mathcal{J}_{2}^{[k]}(t_{n+1})$ into
\begin{equation*}
  \|\mathcal{J}_{2}^{[k]}(t_{n+1})\|\leq C\Delta t\|\mathbf{x}_n-\mathbf{X}_n\|.
\end{equation*}

The error equality (\ref{Thm3-3-11}) remains valid. Substituting the new estimates for $\mathcal{J}_{1}^{[k]}(t_{n+1})$ and $\mathcal{J}_{2}^{[k]}(t_{n+1})$ into (\ref{Thm3-3-11}) and then applying the Grönwall's inequality, we obtain the desired error estimate.
\end{proof}

For $\nu=1$, when $\Delta t$ approaches $c_2\varepsilon$, the step size $h=\Delta t/\varepsilon$ becomes $O(1)$. In this regime, the convergence result from Section \ref{Sec3-2} for $h>c_1\varepsilon$ overlaps with the present analysis. This implies that $c_2$ serves as a second threshold in long-time simulations when $h\approx O(1)$. At the end of this subsection, we estimate the value of this threshold.

When a large step size inversely proportional to $\varepsilon$ is used, its time scale substantially exceeds the cyclotron frequency, so that the cumulative drift effect over each cyclotron period can no longer be neglected. Consequently, the spatial scale characterized by the Larmor radius becomes the key physical quantity in determining this threshold. We adopt guiding-center theory for the analysis \cite{C2015}. Under the influence of a nonuniform electric field, the drift velocity of the guiding center of a charged particle governed by (\ref{CPD2D-homo}) is given by
\begin{equation}\label{GD-01}
  \mathbf{v}_\mathbf{E}=\varepsilon\left(1+\frac{r_L^2}{4}\Delta\right)J\mathbf{E}(\mathbf{y}),
\end{equation}
where $r_L$ denotes the Larmor radius. Since the CPD model (\ref{CPD2D-homo}) normalizes the mass, charge, and magnetic field strength of the charged particle, the Larmor radius is given by $r_L=\varepsilon\|\dot{\mathbf{y}}\|$. Substituting this expression into (\ref{GD-01}), we obtain the time required for the guiding-center drift to traverse one Larmor radius as
\begin{equation*}
  \frac{r_L}{\mathbf{v}_\mathbf{E}}=\frac{\|\dot{\mathbf{y}}\|}{\|\mathbf{E}(y)\|+O(\varepsilon^2)}.
\end{equation*}
Neglecting the higher-order term of $\varepsilon$ and taking the maximum over $t$, we define the threshold $c_2$ as
\begin{equation}\label{GD-02}
  c_2:=\max_{t\in[0,\frac{T}{\varepsilon}]}\frac{\|\dot{\mathbf{y}}(t)\|}{\|\mathbf{E}(y(t))\|}.
\end{equation}
If a reliable a priori upper bound on the particle's velocity is available, then (\ref{GD-02}) can serve as a predicted value for the transition point in the piecewise convergence behavior. From the perspective of numerical computation with fixed $\varepsilon$, the step size $h=\Delta t/\varepsilon>c_2$ indicates that the drift distance of the particle exceeds its Larmor radius in a single time step. The change in position $\mathbf{y}(t+h)-\mathbf{y}(t)$ tends to be proportional to $\Delta t$, thereby exhibiting an error order in terms of $\Delta t$.

Furthermore, if the electric field $\mathbf{E}(\mathbf{y})$ is conservative, i.e., if there exists a potential function $U(\mathbf{y})$ such that $\mathbf{E}(\mathbf{y})=-\nabla U(\mathbf{y})$, then the charged particle admits a conserved energy:
\begin{equation}\label{Energy}
  H(\mathbf{y},\dot{\mathbf{y}})=\frac{1}{2}\|\dot{\mathbf{y}}\|^2 +U(\mathbf{y})\equiv H(\mathbf{y}_0,\dot{\mathbf{y}}_0).
\end{equation}
Using this conservation relation, one obtains an a priori upper bound for the velocity magnitude:
\begin{equation*}
  \|\dot{\mathbf{y}}(t)\|^2\leq 2H(\mathbf{y}_0,\dot{\mathbf{y}}_0) -2\inf\limits_{\mathbf{y}}U(\mathbf{y}).
\end{equation*}
Substituting this bound into (\ref{GD-02}) yields a threshold suitable for a priori estimation:
\begin{equation}\label{GD-03}
  \tilde{c}_2\approx\frac{\left(2H(\mathbf{y}_0,\dot{\mathbf{y}}_0) -2\inf\limits_{\mathbf{y}}U(\mathbf{y})\right)^{\frac{1}{2}}}{\inf\limits_{\mathbf{y}}\|\nabla U(\mathbf{y})\|}.
\end{equation}
The threshold $\tilde{c}_2$ depends on the global minimum of the potential and electric field magnitudes over a region. Therefore, it is more suitable for cases where the electric field varies gently over a spatial scale of $O(1)$ around the particle.

\section{Numerical experiments}\label{Sec4}
\newcounter{NumEx}
In this section, we present several examples employing the newly proposed EI to solve the CPD equation (\ref{CPD2D-homo}). We denote the local linear extension exponential integrator using polynomials up to degree $k$ by LLEEI$(k+1)$. Since the matrices $A_1^{[k]}$ and $A_0^{[k]}$ are of moderate size, the matrix exponential in (\ref{LLEEI01}) is computed by the built-in MATLAB function \texttt{expm} in all these examples \cite{SW2009}.

\refstepcounter{NumEx}\label{Ex1}
\textbf{Example \ref{Ex1}.} We first verify that the unified framework of the LLEEI schemes can achieve arbitrarily high-order uniform convergence with respect to $\varepsilon$ as $k$ increases, and show that this convergence exhibits the piecewise behavior revealed in Section \ref{Sec3-2}. The electric field is generated by a uniformly charged wire located at $y_2=0,-1\leq y_1\leq 1$:
\begin{equation*}
  \mathbf{E}(\mathbf{y})=\frac{1}{\sqrt{(y_1-1)^2+y_2^2}}\begin{bmatrix} 1 \\ \frac{1-y_1}{y_2} \end{bmatrix} +\frac{1}{\sqrt{(y_1+1)^2+y_2^2}}\begin{bmatrix} -1 \\ \frac{y_1+1}{y_2}
  \end{bmatrix}.
\end{equation*}
The initial position and velocity of the charged particle are taken as $\mathbf{y}_0=\left[\frac{1}{2},\frac{3}{4}\right]^\top$ and $\dot{\mathbf{y}}_0=[-2,-2]^\top$, respectively. We consider the case $\nu=0$ and set the final time to $T=6$. Since no explicit exact solution is available, we compute a reference solution using the classical fourth-order Runge-Kutta method with a very small time step $0.1/2^{17}$. We denote by $err(\cdot)$ the global error of the observable physical quantity evaluated at the time nodes $\{t_n\}_{n=0}^N$.

We first examine the dependence of the errors on $\Delta t$ for various values of $\varepsilon$. We verify convergence orders up to five, i.e., $k=1,2,3,4$, and present the results in Figure \ref{fig1-1}. It can be observed that, for relatively large $\varepsilon$, the LLEEI$(k+1)$ scheme exhibits $O(\Delta t^{k+1})$ convergence. For small $\varepsilon$, the error in the simulation results remains bounded over the entire range of step sizes. This is consistent with the theoretical results stated in Theorem \ref{Thm-CPD2D-HomoB}. For intermediate values of $\varepsilon$, a clear piecewise behavior of the error curves is observed. Moreover, as $\varepsilon$ decreases, the transition point shifts to smaller values of $\Delta t$, in agreement with the prediction in (\ref{Thm-CPD2D-HomoB-R1}) that the transition point is located at $\Delta t=O(\varepsilon)$.

\begin{figure}[htbp]
  \centering
  \begin{subfigure}[b]{0.45\textwidth}
    \centering
    \includegraphics[width=\linewidth]{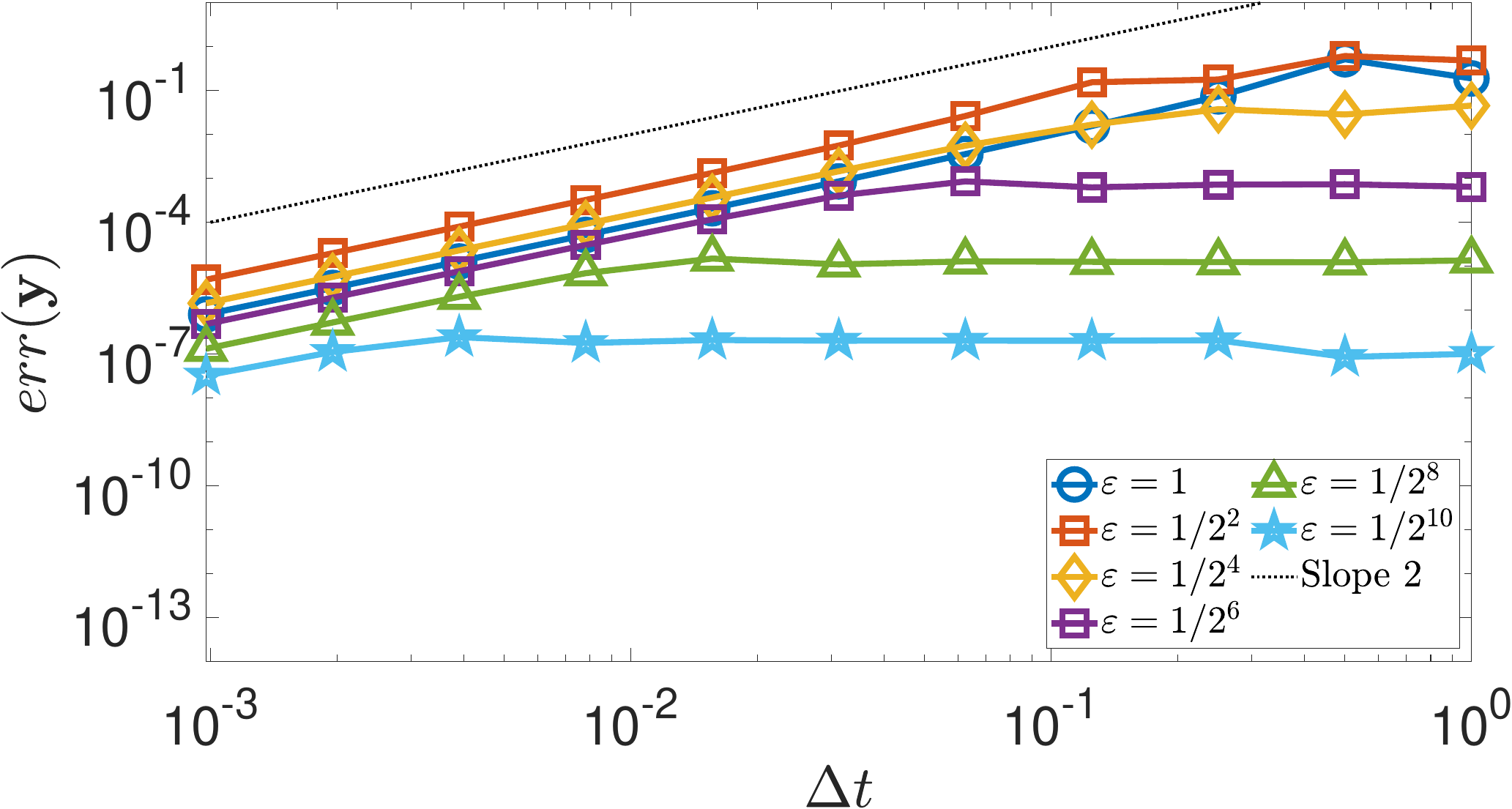}
  \end{subfigure}
  \begin{subfigure}[b]{0.45\textwidth}
    \centering
    \includegraphics[width=\linewidth]{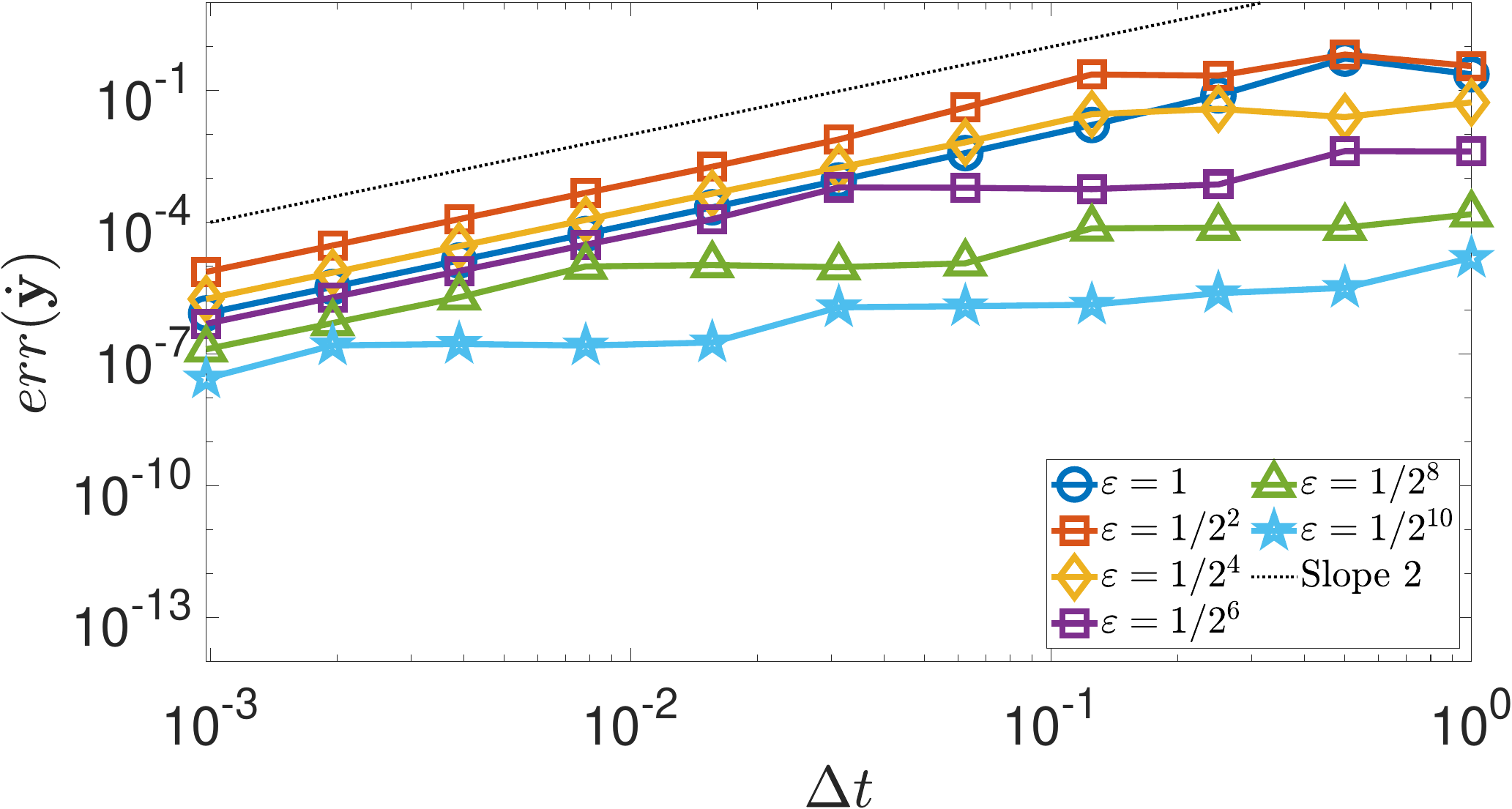}
  \end{subfigure}
  \vspace{1ex}
  \begin{subfigure}[b]{0.45\textwidth}
    \centering
    \includegraphics[width=\linewidth]{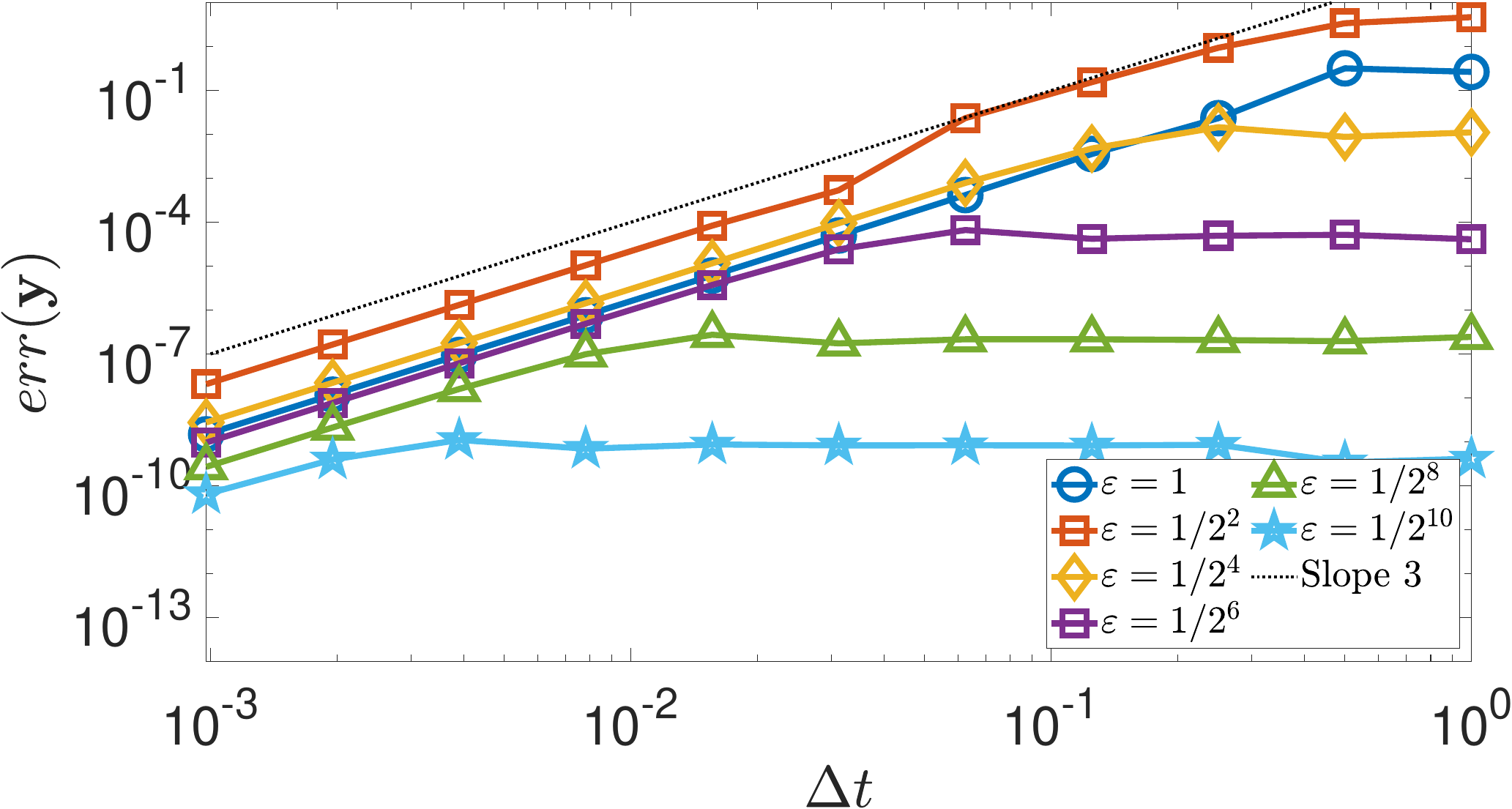}
  \end{subfigure}
  \begin{subfigure}[b]{0.45\textwidth}
    \centering
    \includegraphics[width=\linewidth]{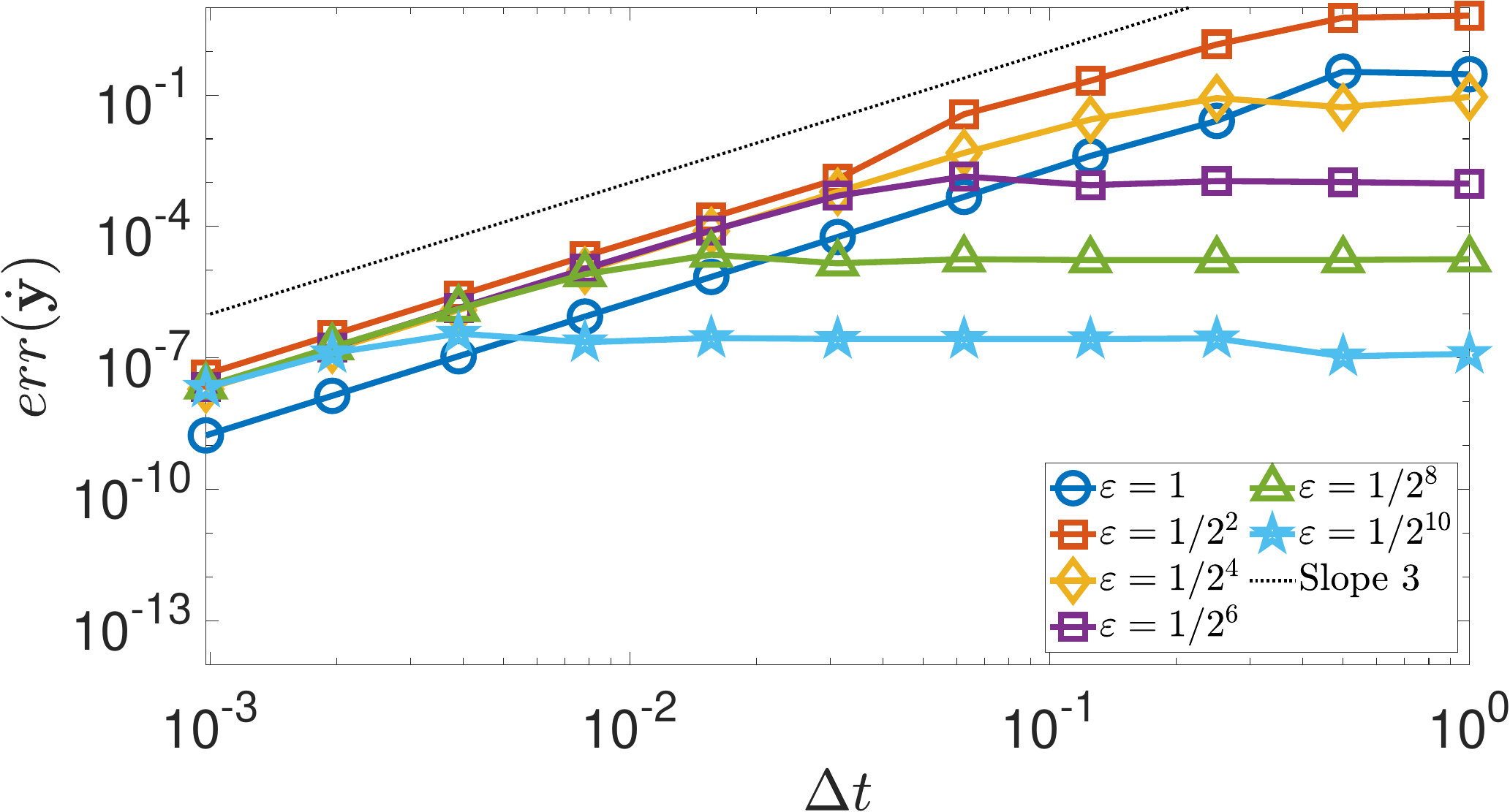}
  \end{subfigure}
  \vspace{1ex}
    \begin{subfigure}[b]{0.45\textwidth}
    \centering
    \includegraphics[width=\linewidth]{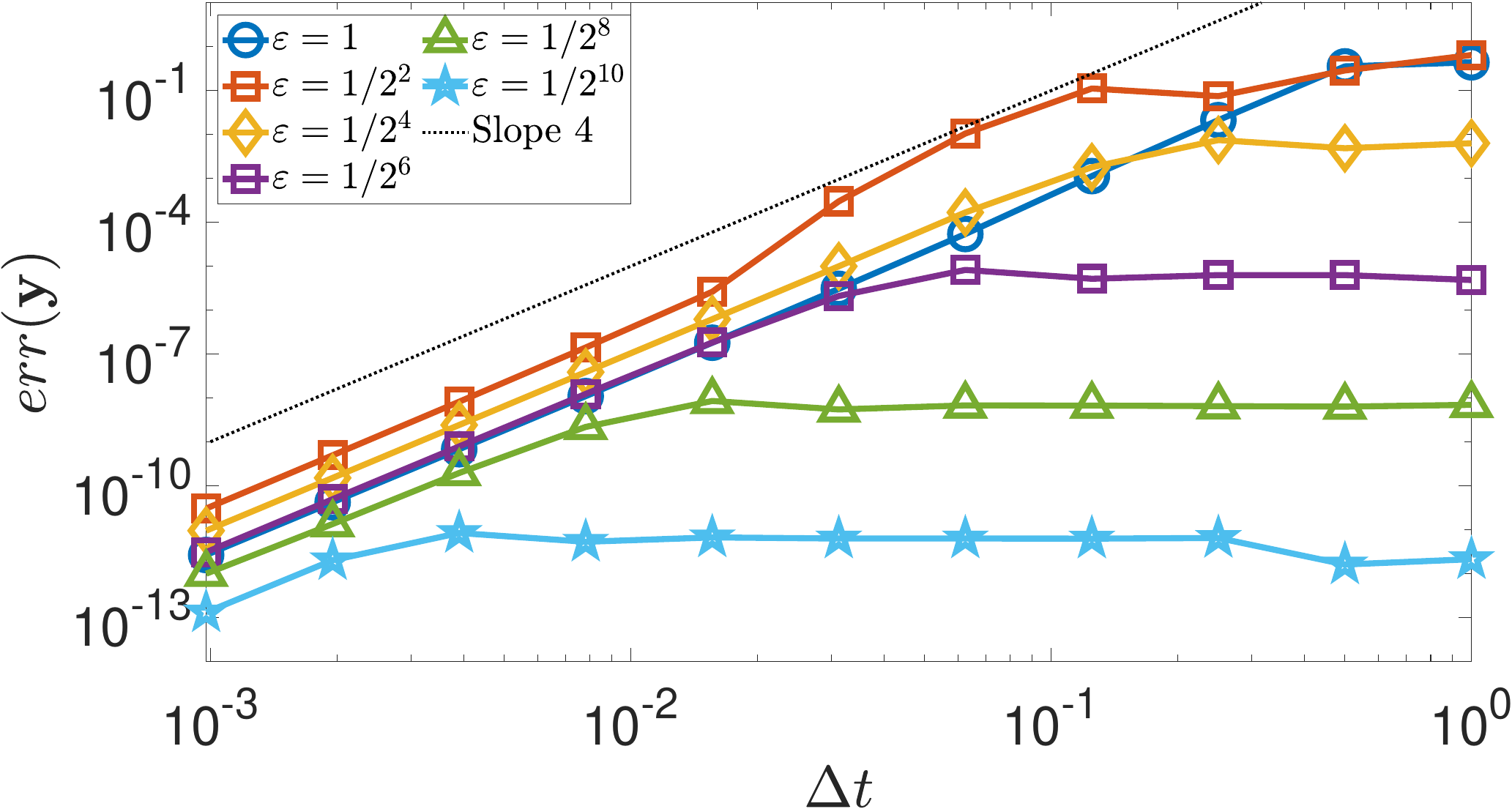}
  \end{subfigure}
  \begin{subfigure}[b]{0.45\textwidth}
    \centering
    \includegraphics[width=\linewidth]{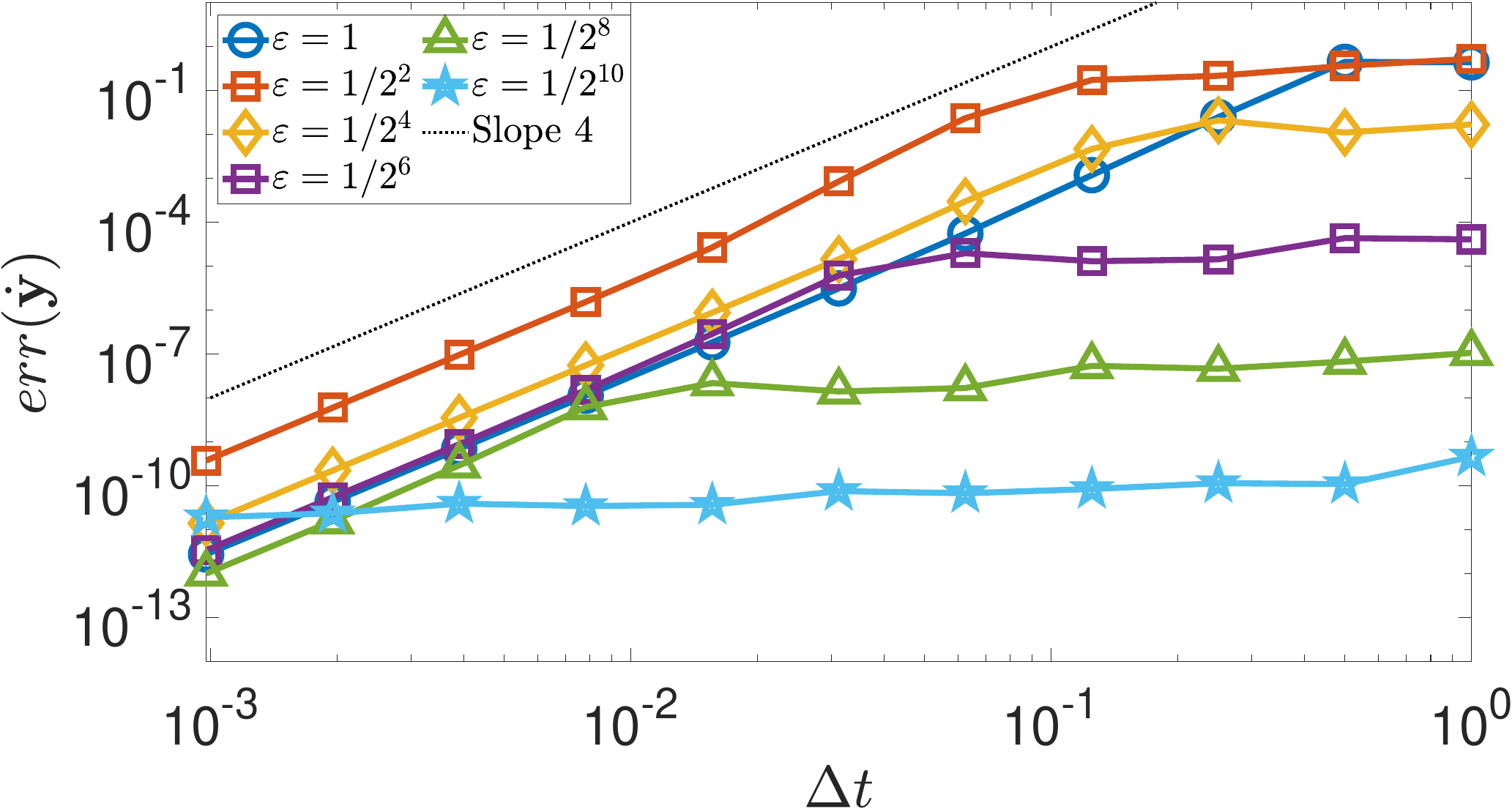}
  \end{subfigure}
  \vspace{1ex}
    \begin{subfigure}[b]{0.45\textwidth}
    \centering
    \includegraphics[width=\linewidth]{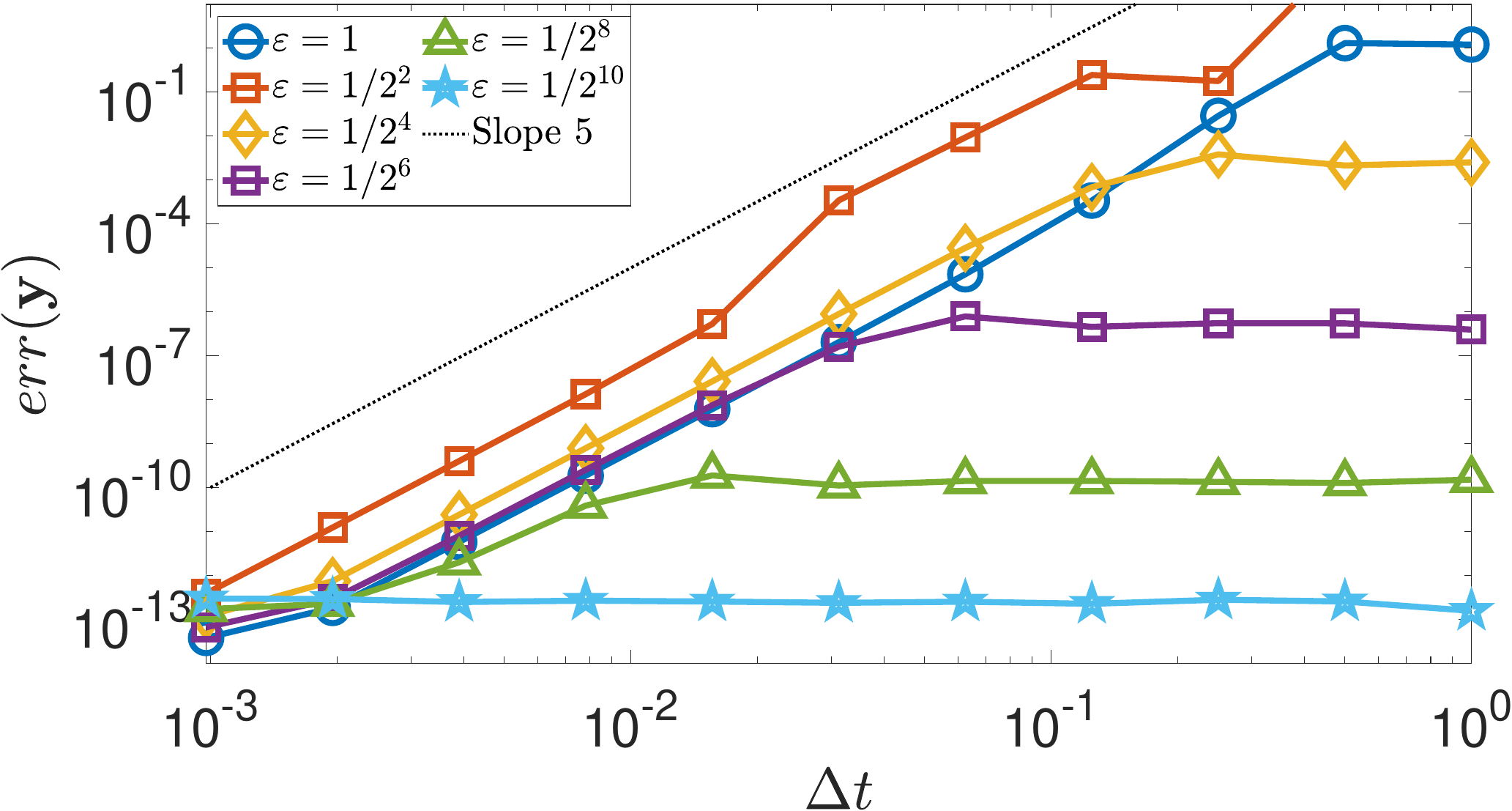}
  \end{subfigure}
  \begin{subfigure}[b]{0.45\textwidth}
    \centering
    \includegraphics[width=\linewidth]{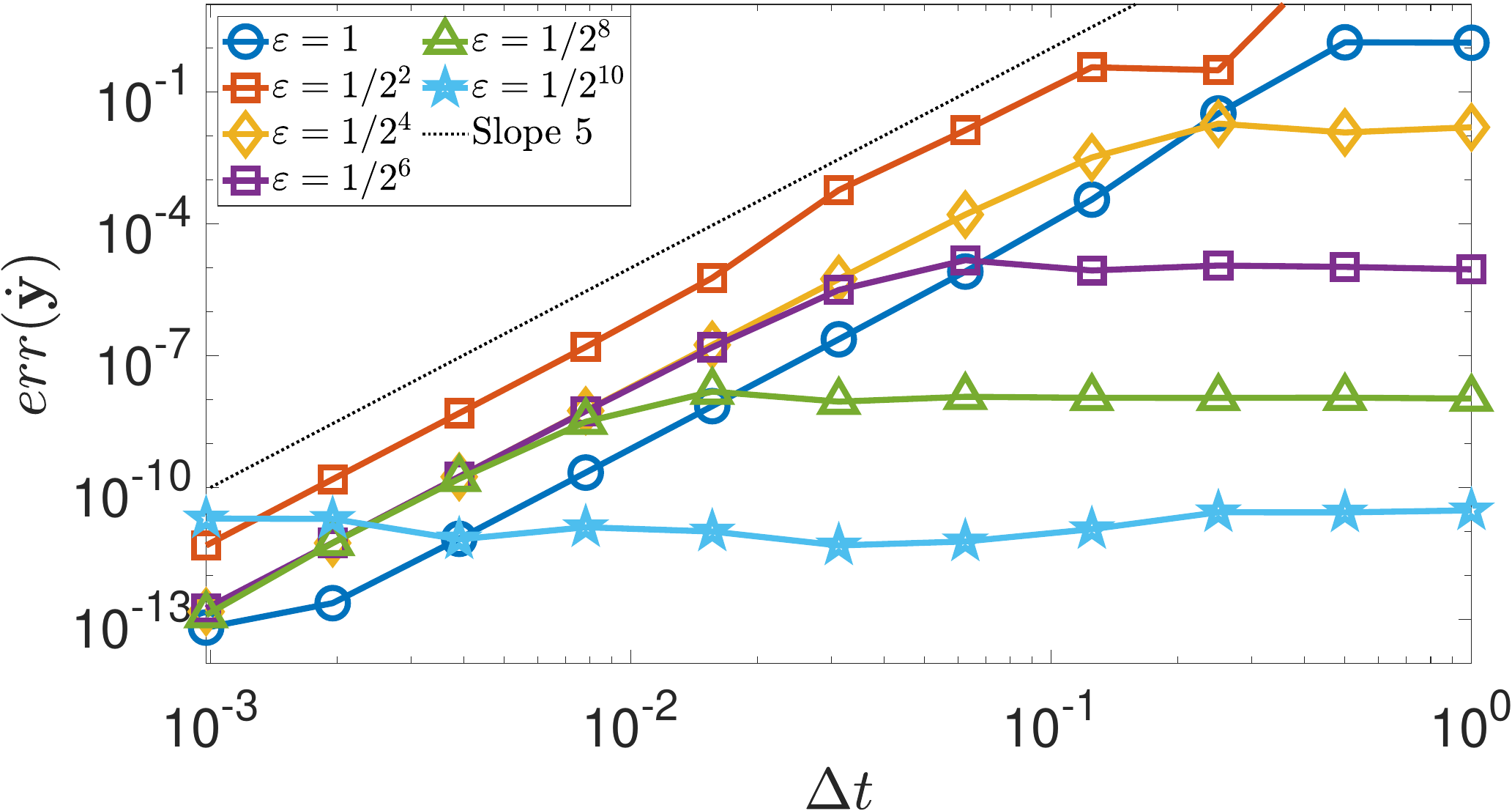}
  \end{subfigure}
  \caption{The left and right columns illustrate the dependence of the error in $\mathbf{y}$ and $\dot{\mathbf{y}}$ on $\Delta t$, respectively. From top to bottom, each row corresponds to $k=1,2,3,4$.}
  \label{fig1-1}
\end{figure}

Next, we fix several $\Delta t$ and examine the dependence of the error on $\varepsilon$ to verify its uniformity. The numerical results are presented in Figure \ref{fig1-2}. Once again, we observe piecewise convergence. For relatively small $\varepsilon$, the error in the particle position obtained by the LLEEI$(k+1)$ method exhibits convergence order $O(\varepsilon^{k+2})$, while the error in the velocity exhibits convergence order $O(\varepsilon^{k+1})$. For relatively large $\varepsilon$, the position error remains first-order convergent, whereas the velocity error remains uniformly bounded. For intermediate $\varepsilon$, the convergence is piecewise, consistent with the behavior shown in (\ref{Thm-CPD2D-HomoB-R1}).

\begin{figure}[htbp]
  \centering
  \begin{subfigure}[b]{0.45\textwidth}
    \centering
    \includegraphics[width=\linewidth]{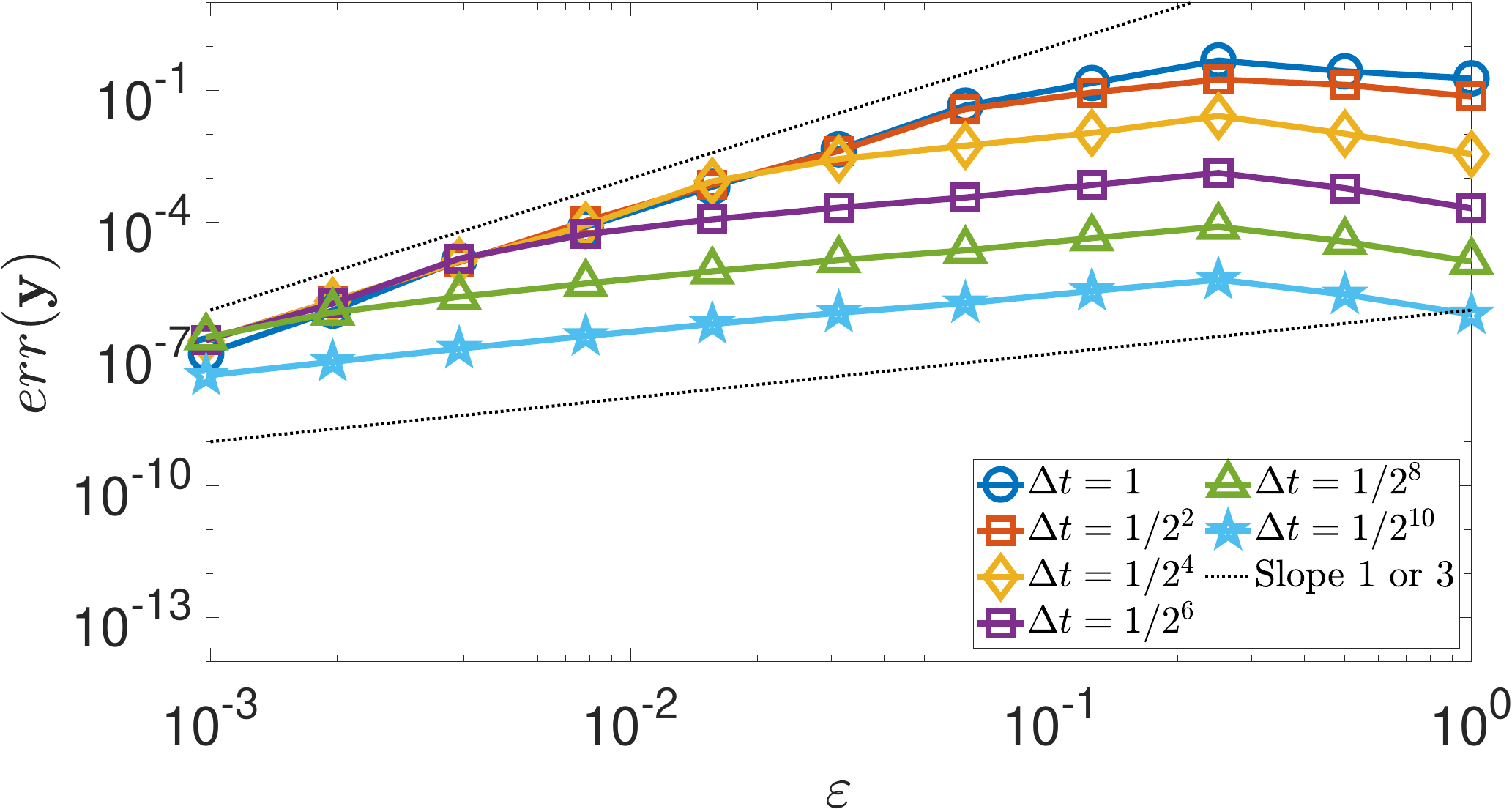}
  \end{subfigure}
  \begin{subfigure}[b]{0.45\textwidth}
    \centering
    \includegraphics[width=\linewidth]{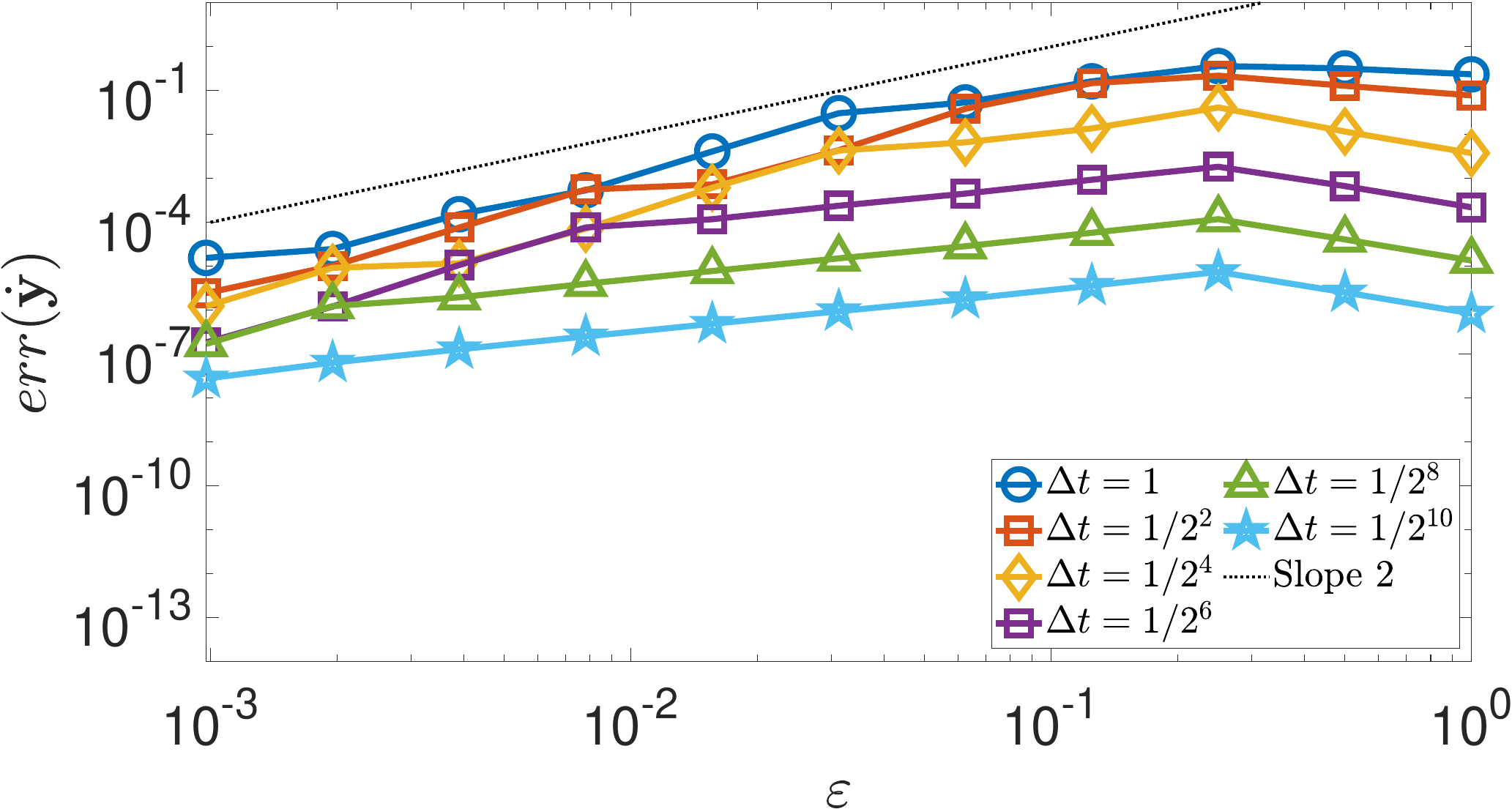}
  \end{subfigure}
  \vspace{1ex}
  \begin{subfigure}[b]{0.45\textwidth}
    \centering
    \includegraphics[width=\linewidth]{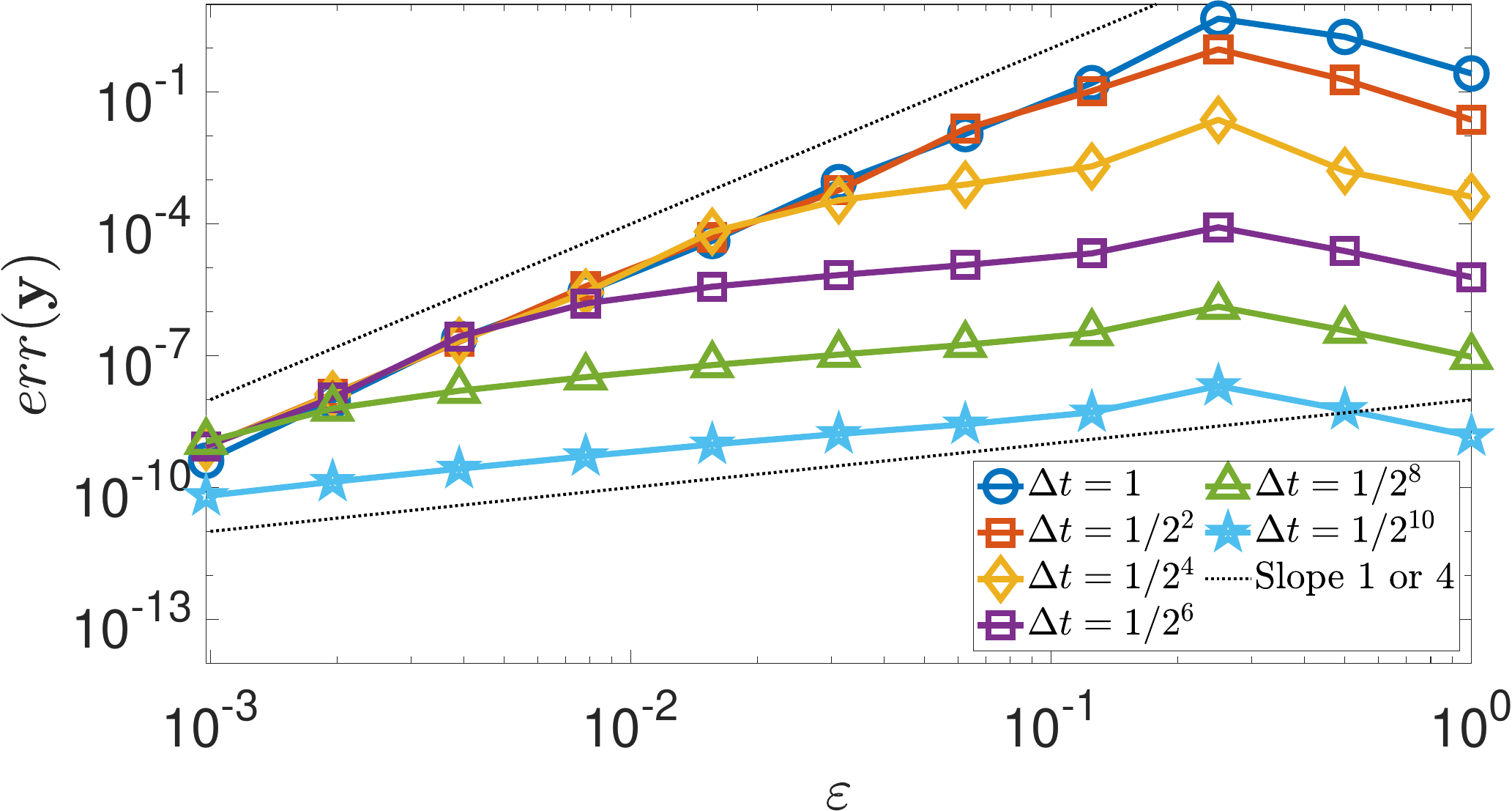}
  \end{subfigure}
  \begin{subfigure}[b]{0.45\textwidth}
    \centering
    \includegraphics[width=\linewidth]{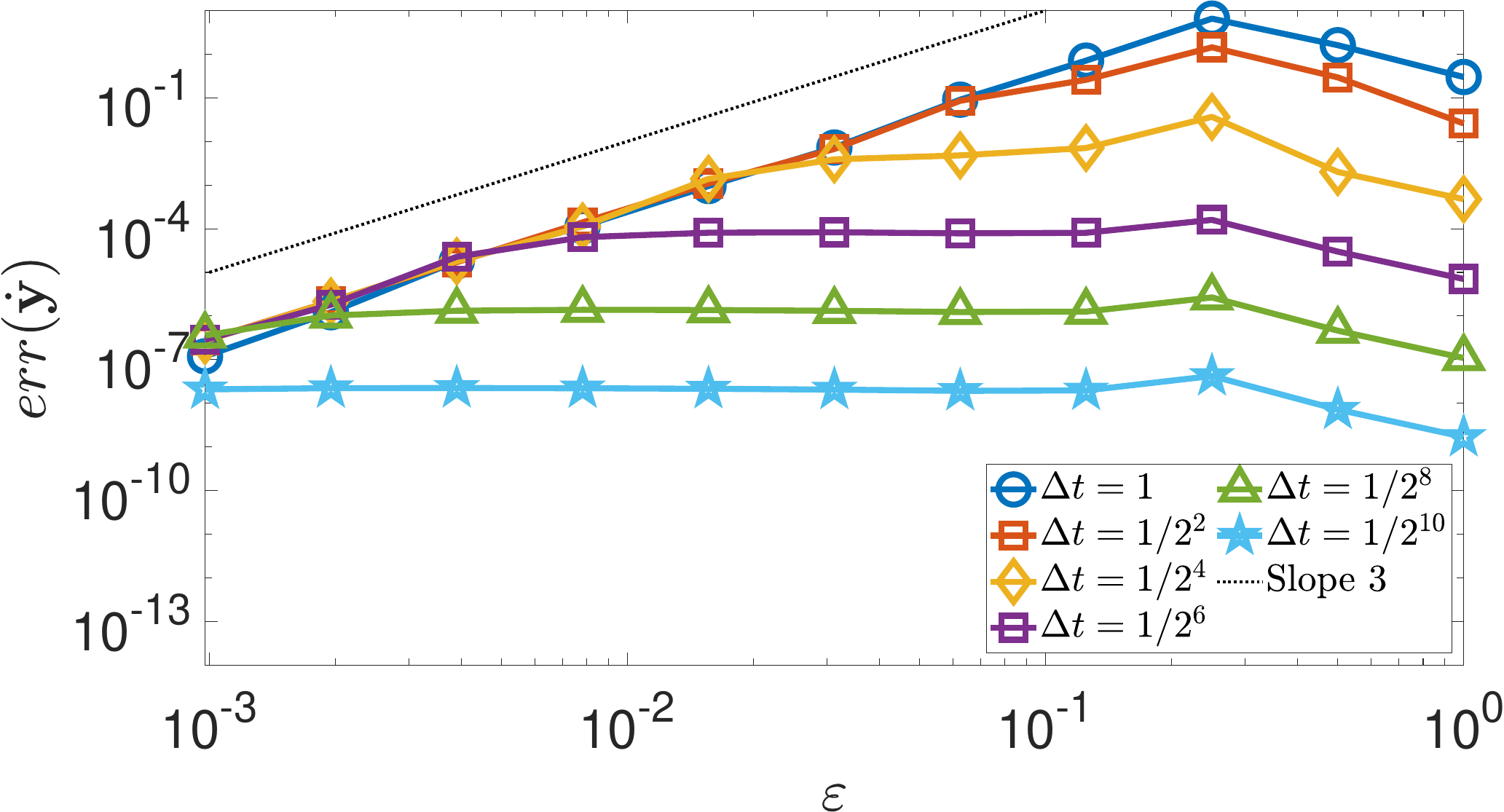}
  \end{subfigure}
  \vspace{1ex}
    \begin{subfigure}[b]{0.45\textwidth}
    \centering
    \includegraphics[width=\linewidth]{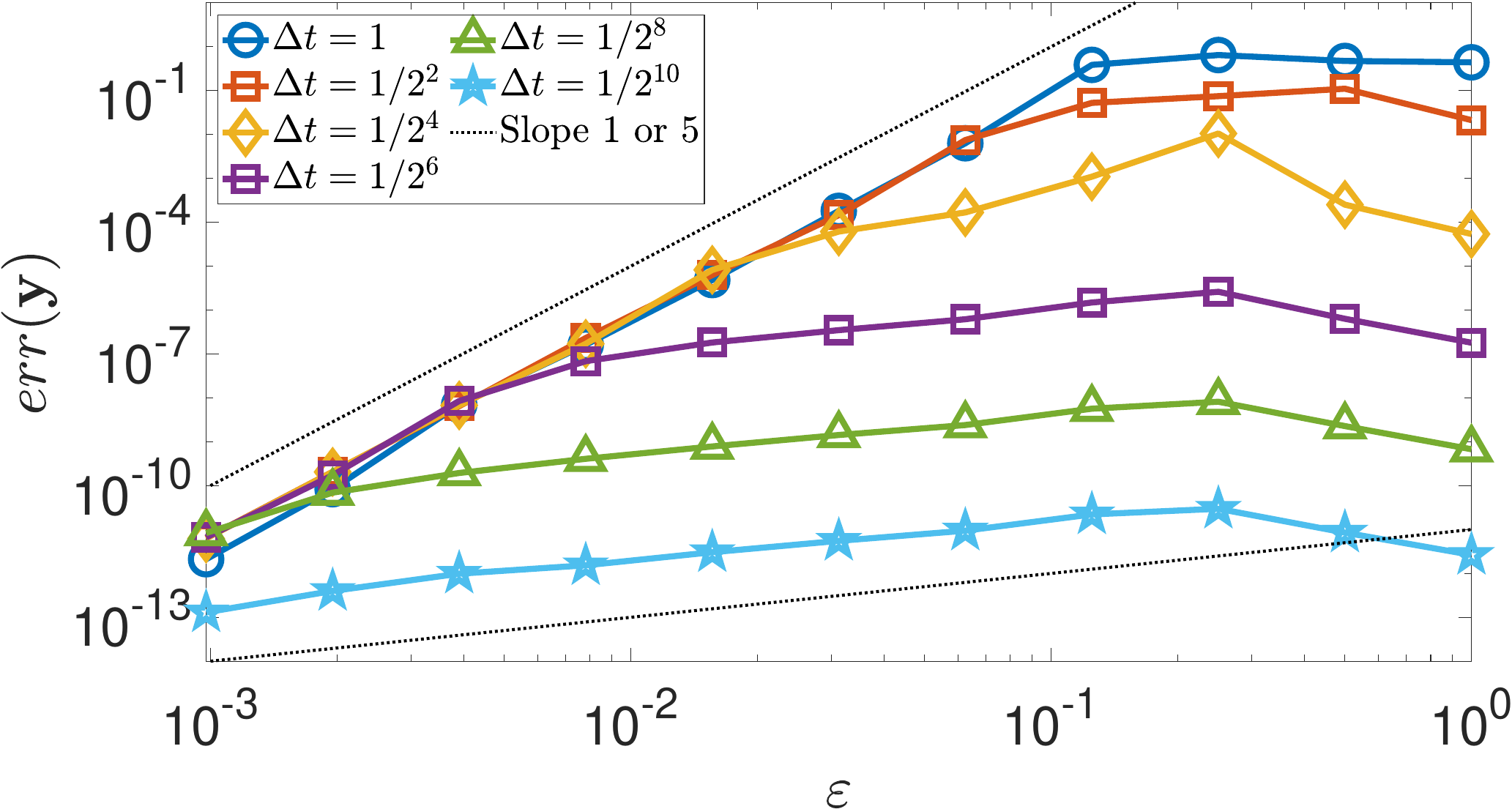}
  \end{subfigure}
  \begin{subfigure}[b]{0.45\textwidth}
    \centering
    \includegraphics[width=\linewidth]{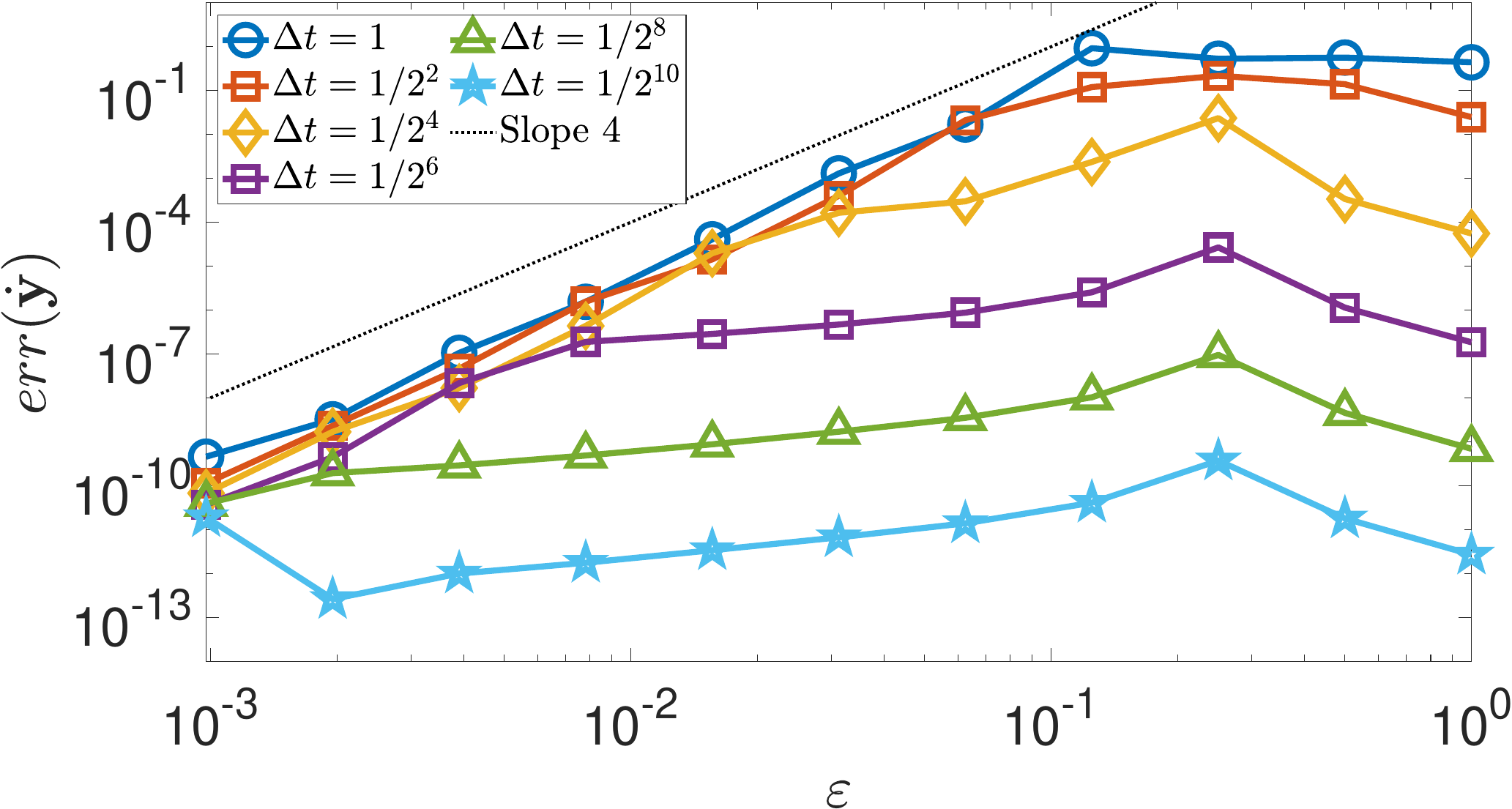}
  \end{subfigure}
  \vspace{1ex}
    \begin{subfigure}[b]{0.45\textwidth}
    \centering
    \includegraphics[width=\linewidth]{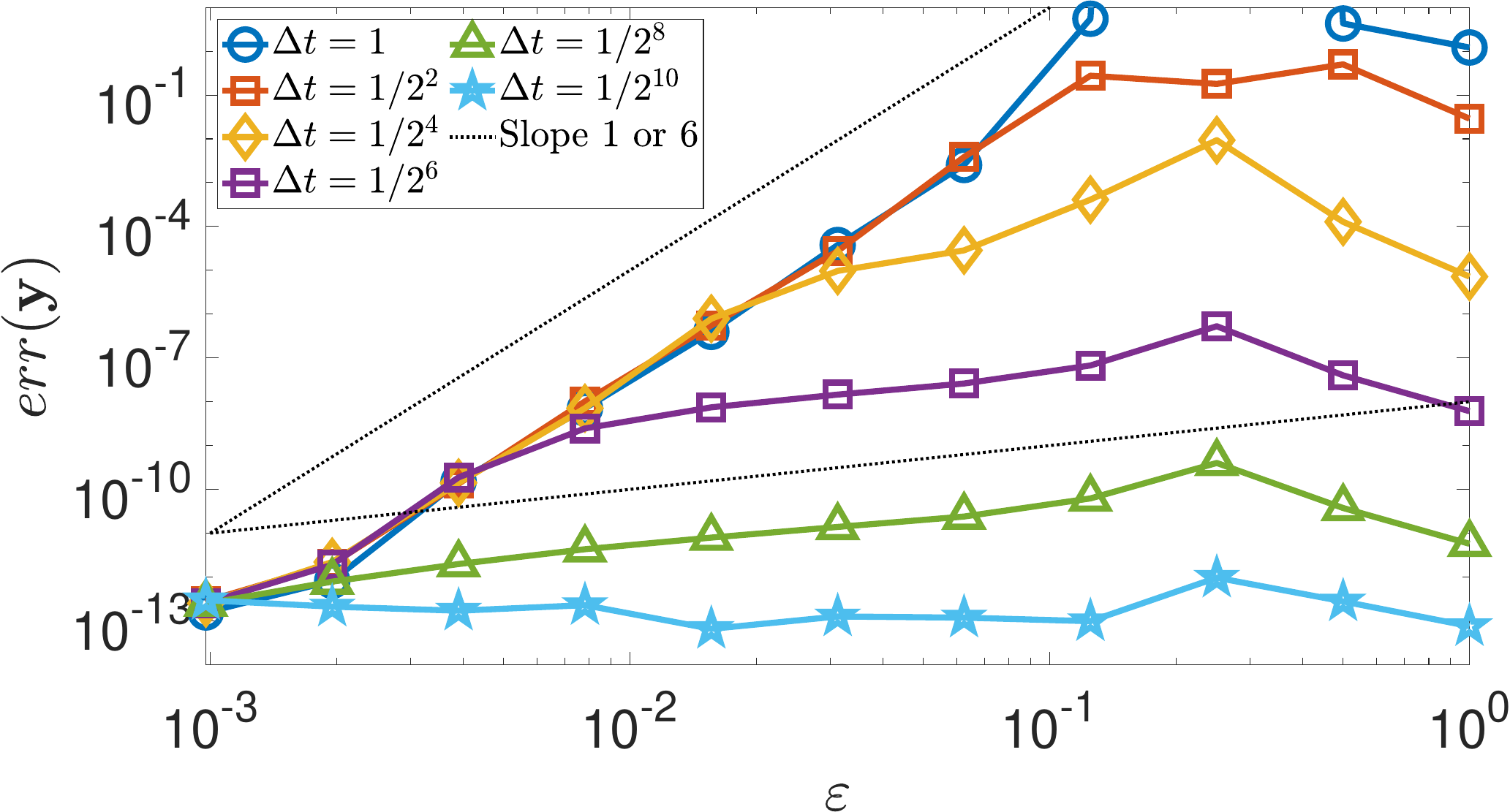}
  \end{subfigure}
  \begin{subfigure}[b]{0.45\textwidth}
    \centering
    \includegraphics[width=\linewidth]{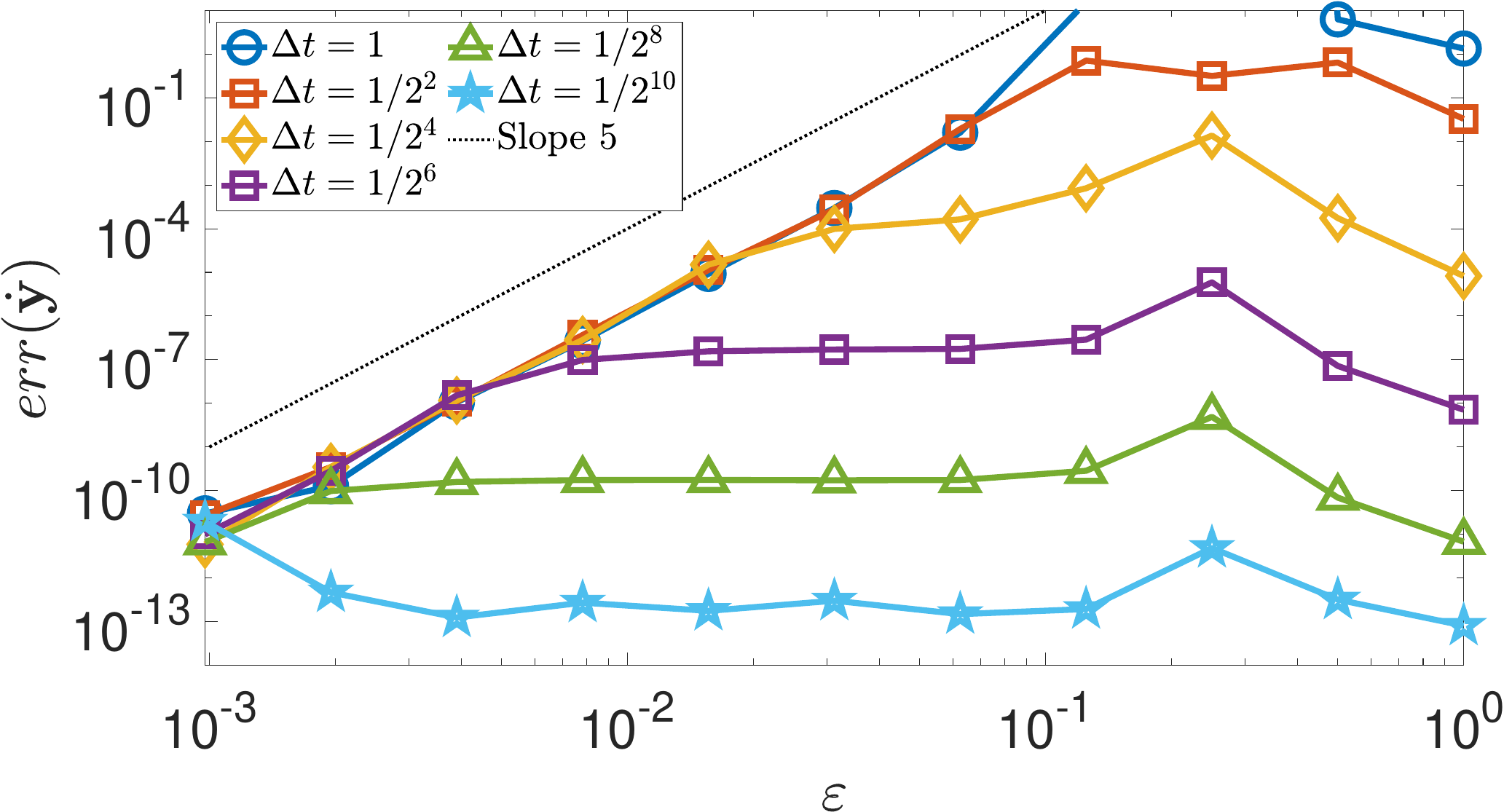}
  \end{subfigure}
  \caption{The left and right columns illustrate the dependence of the error in $\mathbf{y}$ and $\dot{\mathbf{y}}$ on $\varepsilon$, respectively. From top to bottom, each row corresponds to $k=1,2,3,4$.}
  \label{fig1-2}
\end{figure}

From the above results, it is evident that the LLEEI method exhibits several notable properties. First, it demonstrates an improved uniform convergence property: as $\varepsilon$ decreases, the numerical error decreases correspondingly; notably, for large step sizes, it decays rapidly as a power of $\varepsilon$. This distinguishes the method from most traditional approaches and meets the requirements of charged-particle simulations in strong magnetic fields. Second, we observe that the $O(\varepsilon^{k+2})$ bound, which is independent of $\Delta t$, leads to another counterintuitive yet highly useful phenomenon: even when $\Delta t$ is taken to be of $O(1)$ scale, far exceeding the cyclotron period, the numerical error remains at a very small level proportional to a power of the cyclotron frequency. This enables high-precision simulations with very large step sizes.

\refstepcounter{NumEx}\label{Ex2}
\textbf{Example \ref{Ex2}.} Consider the electric field generated by a point charge located at the origin:
\begin{equation*}
  \mathbf{E}(\mathbf{y})=\frac{1}{(y_1^2+y_2^2)^{\frac{3}{2}}}\begin{bmatrix} y_1\\y_2 \end{bmatrix}.
\end{equation*}
The initial position and velocity of the particle are chosen as $\mathbf{y}_0=\left[\frac{1}{3},\frac{1}{4}\right]^\top$ and $\dot{\mathbf{y}}_0=\left[\frac{2}{5},\frac{3}{5}\right]^\top$, respectively. The final time is set to $T=6$ with $\nu=0$. We perform a comparative study of the LLEEI methods and classical numerical schemes, particularly traditional EIs, to show that the local linear extension technique enhances numerical accuracy. The following numerical methods are considered:
\begin{itemize}
\item The second- and third-order exponential integrators from \cite{CM2002} and the fourth-order version from \cite{K2005}, denoted by ETDRK2, ETDRK3, and ETDRK4, respectively.
\item The $k$-th order Exponential Rosenbrock method from \cite{HOS2009}, denoted by EXPRB$k$. By using Lemma 1 in \cite{HOS2009} directly to compute (\ref{LLEEI01}), it can be shown that LLEEI2 and EXPRB2 are essentially equivalent.
\item The explicit splitting methods from \cite{WZ2021}. The schemes based on Lie-Trotter splitting and Strang splitting are denoted by ES1 and ES2, respectively. Using the subflow defined in these methods, the classical Yoshida splitting can also be applied directly, denoted by ES4.
\item The classical Boris algorithm.
\item The $k$-th order classical Runge-Kutta method, denoted by RK$k$.
\end{itemize}

The two rows of panels in Figure \ref{fig2-1} present numerical results for the error with respect to $\Delta t$ for the second- to fourth-order versions of each method, with $\varepsilon=1/2^4$ and $\varepsilon=1/2^8$, respectively. When $\varepsilon=1/2^4$, the problem is weakly oscillatory, and all numerical methods attain their theoretical accuracy. The differences among the exponential-type methods are small, but their error performance is superior to that of the other methods. However, in the strong magnetic field case represented by $\varepsilon=1/2^8$, the splitting methods and ETDRK methods exhibit only stability for the stiffness when large step sizes are applied, and their convergence is constrained by the oscillatory nature of the solution. The EXPRB methods, which linearize only the first-order part via the Jacobian matrix of the nonlinear term, also exhibit a reduction in convergence order. In contrast, the LLEEI method retains a significant accuracy advantage over all other tested methods across the entire range of step sizes considered.

In each panel, one or two vertical dashed lines are drawn. The dashed line in the panels of the first row and the left dashed line in the panels of the second row represent the threshold $h=\tilde{c}_1\varepsilon$ given in Section \ref{Sec3-2}. The right dashed line in the panels of the second row represents the threshold $h=\tilde{c}_2$ given in Section \ref{Sec3-3}, where the two infima in (\ref{GD-03}) are taken as their respective minimum values over all numerical solutions. In this example, $\tilde{c}_2\approx 0.14$. It can be observed that the two empirical thresholds $\tilde{c}_1$ and $\tilde{c}_2$ effectively distinguish the different convergence behaviors of the LLEEI method across the three scales. In contrast, the other methods exhibit their theoretical convergence orders only in the left region. When $h>\tilde{c}_1\varepsilon$, a comparison of the performance of the same method for different orders $k$ shows that the error curves of the traditional EIs are nearly identical, indicating that increasing the order does not improve the accuracy. By contrast, as $k$ increases, the accuracy of the LLEEI method improves accordingly, and its relative advantage over the other schemes becomes more pronounced when higher-order methods are used.

\begin{figure}[htbp]
  \centering
  \begin{subfigure}[b]{0.3\textwidth}
    \centering
    \includegraphics[width=\linewidth]{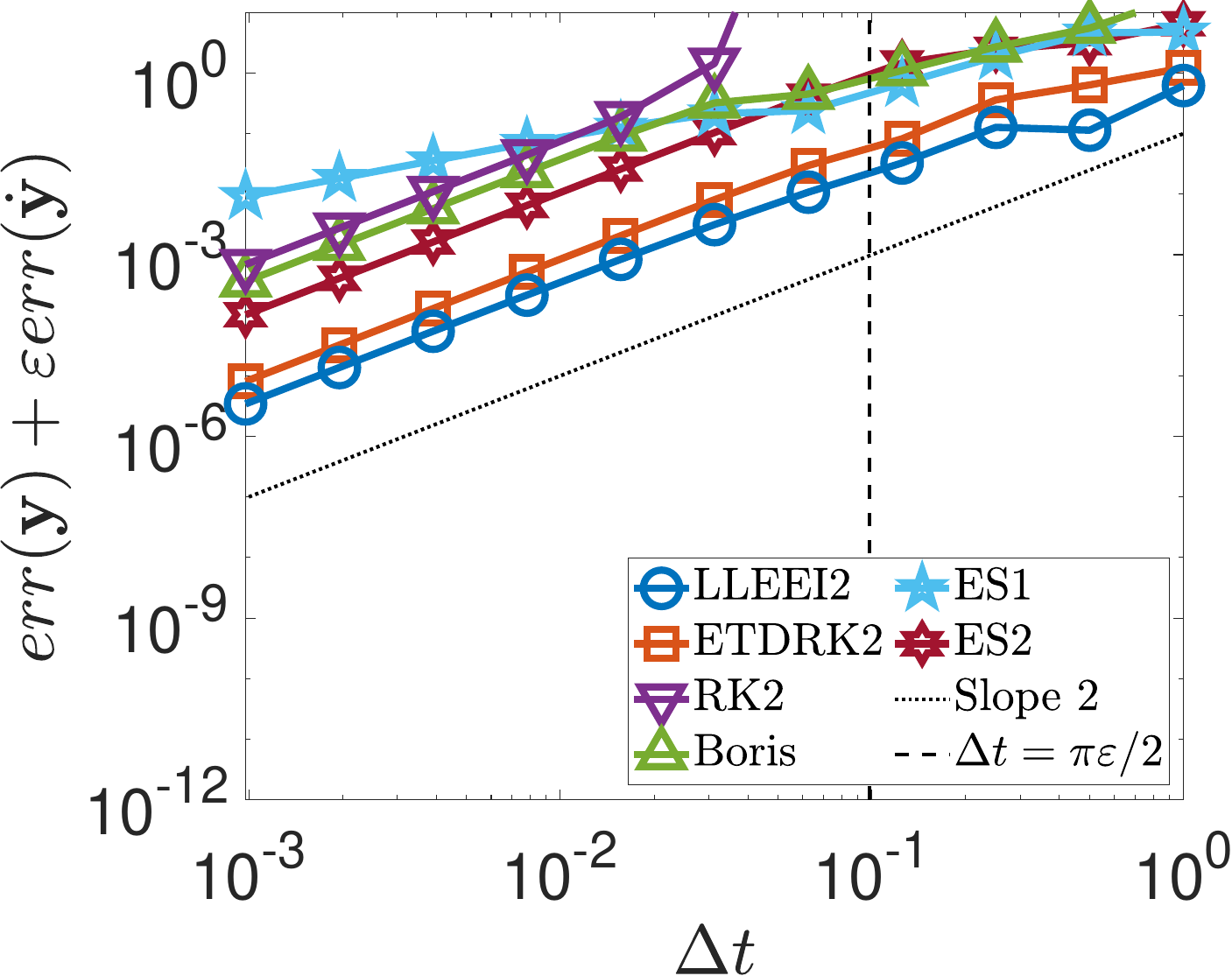}
  \end{subfigure}
  \begin{subfigure}[b]{0.3\textwidth}
    \centering
    \includegraphics[width=\linewidth]{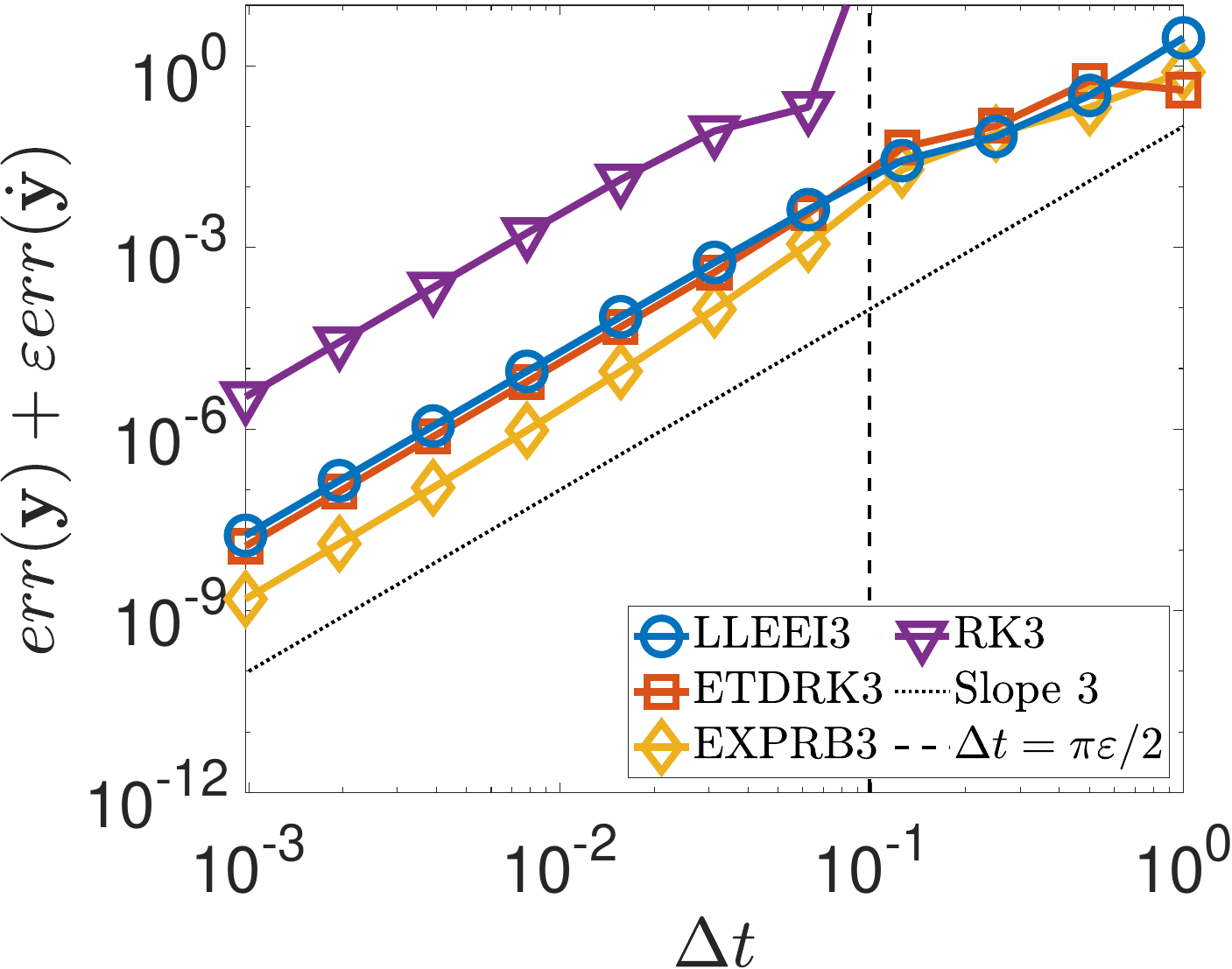}
  \end{subfigure}
  \begin{subfigure}[b]{0.3\textwidth}
    \centering
    \includegraphics[width=\linewidth]{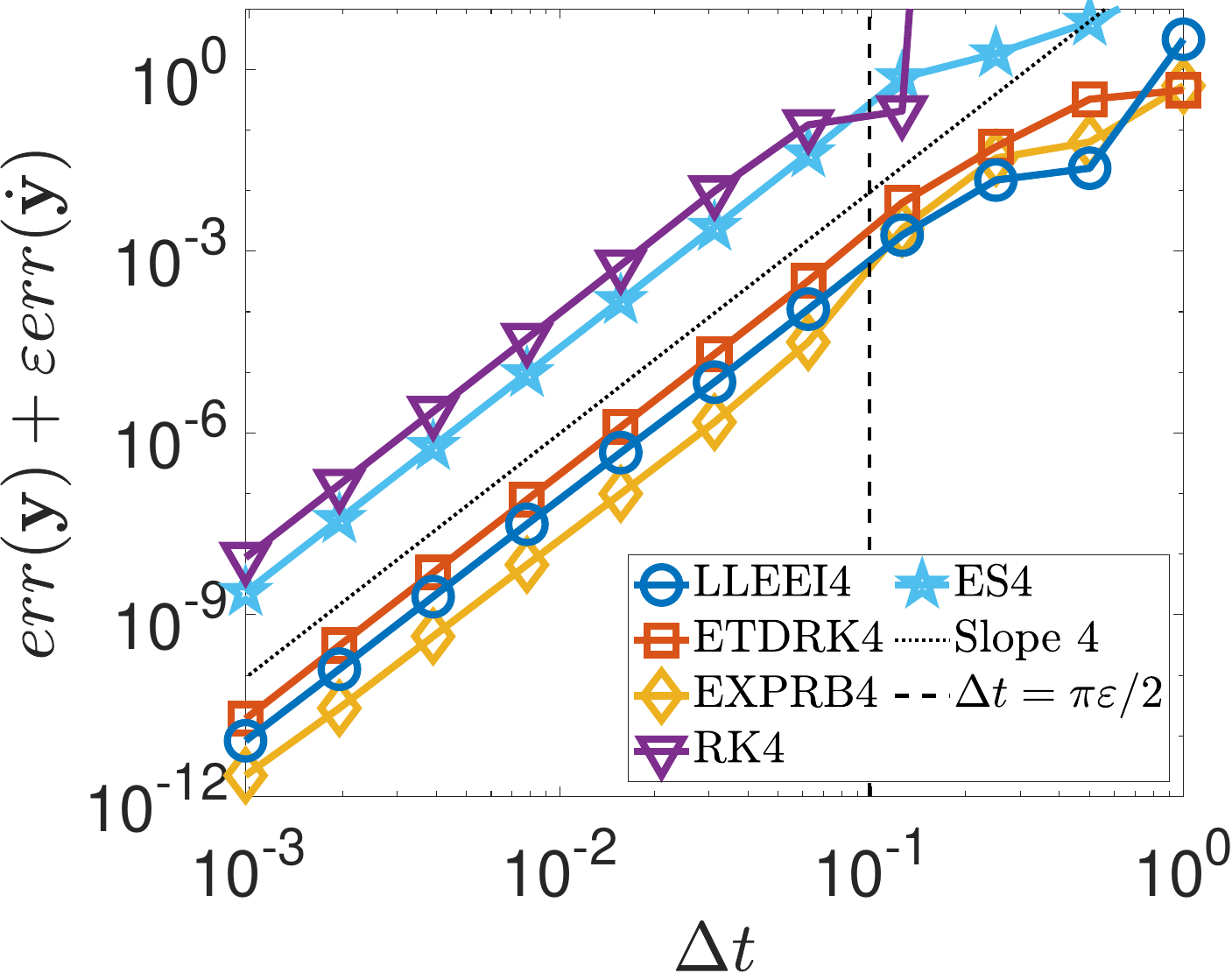}
  \end{subfigure}
  \vspace{1ex}
  \begin{subfigure}[b]{0.3\textwidth}
    \centering
    \includegraphics[width=\linewidth]{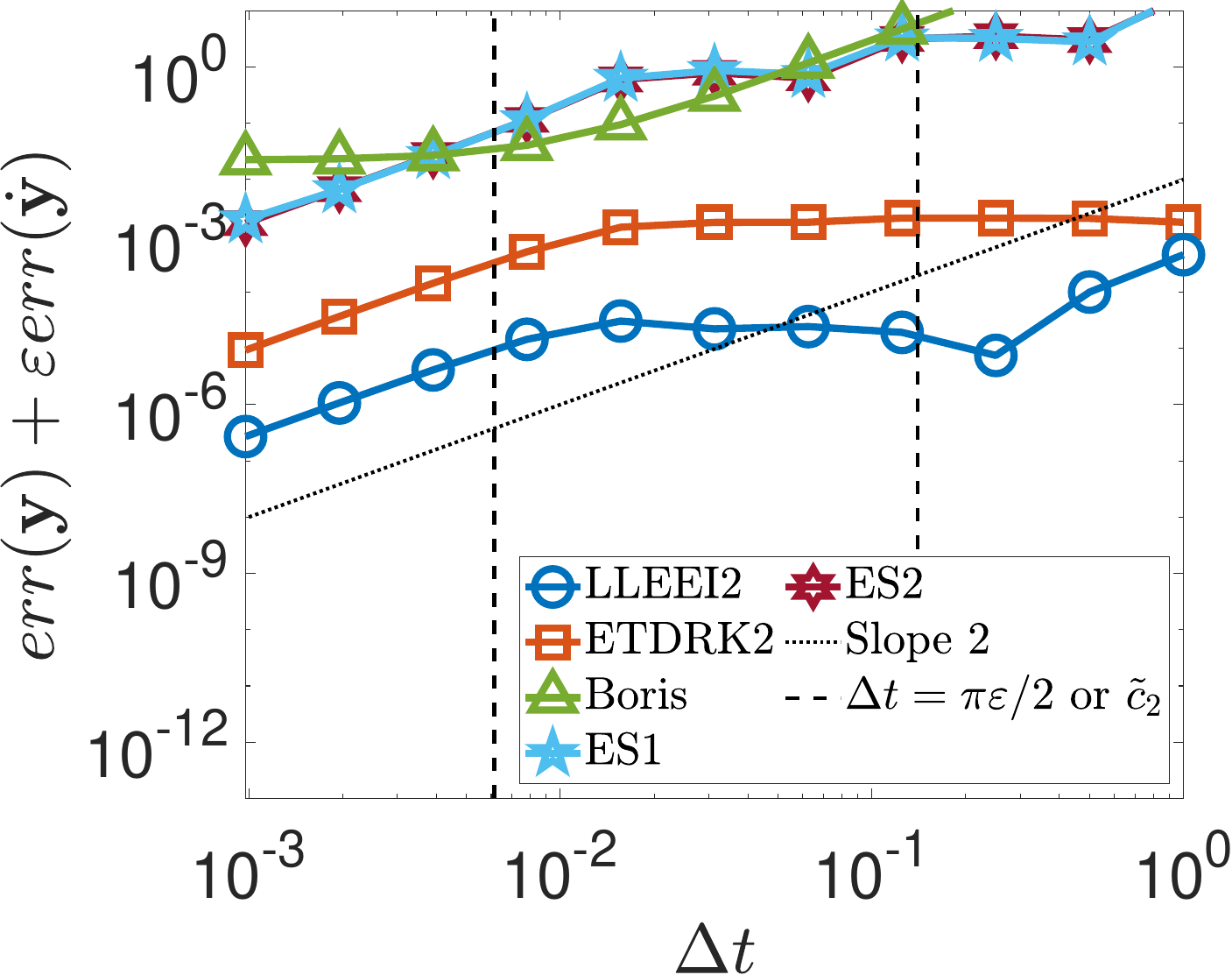}
  \end{subfigure}
  \begin{subfigure}[b]{0.3\textwidth}
    \centering
    \includegraphics[width=\linewidth]{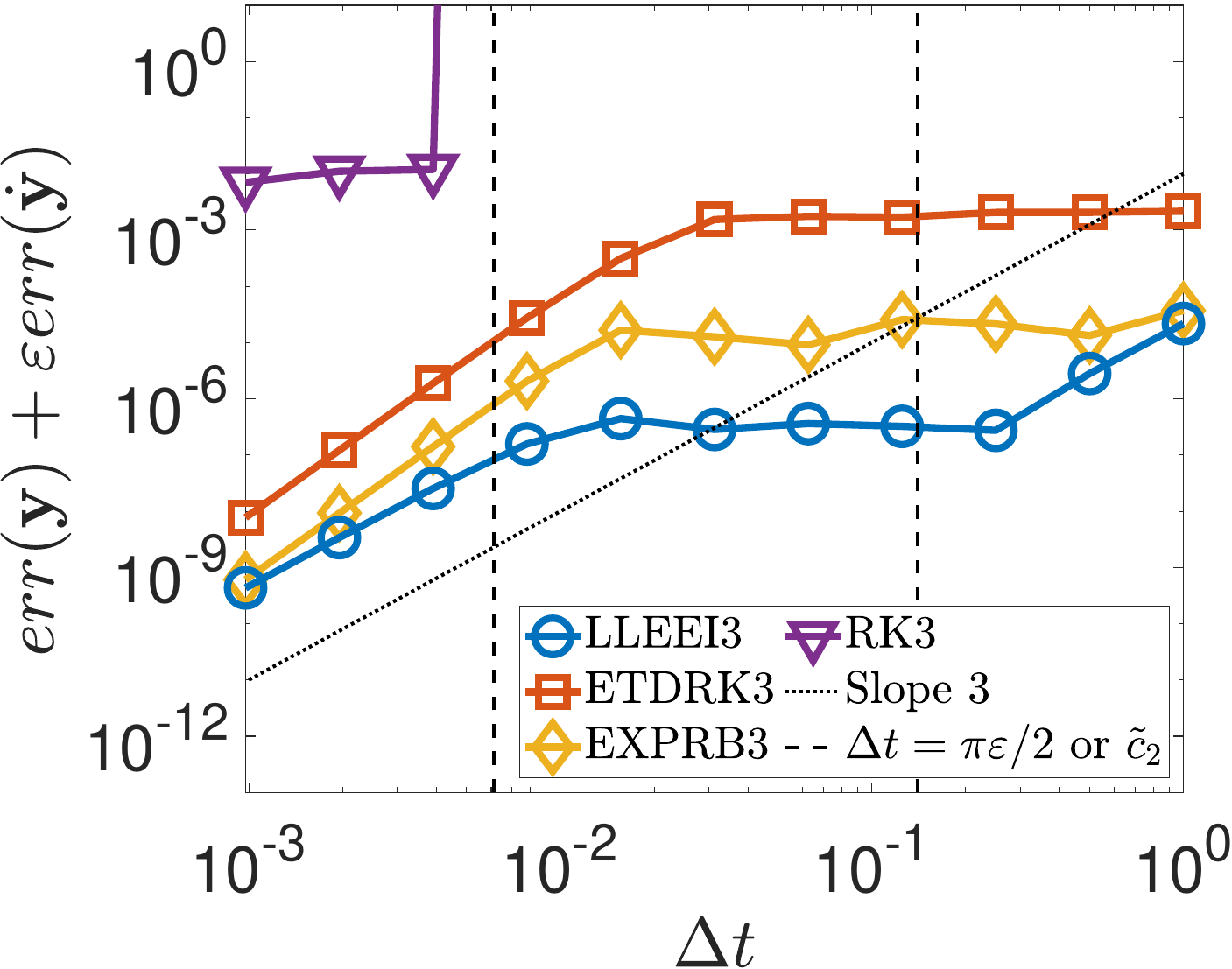}
  \end{subfigure}
    \begin{subfigure}[b]{0.3\textwidth}
    \centering
    \includegraphics[width=\linewidth]{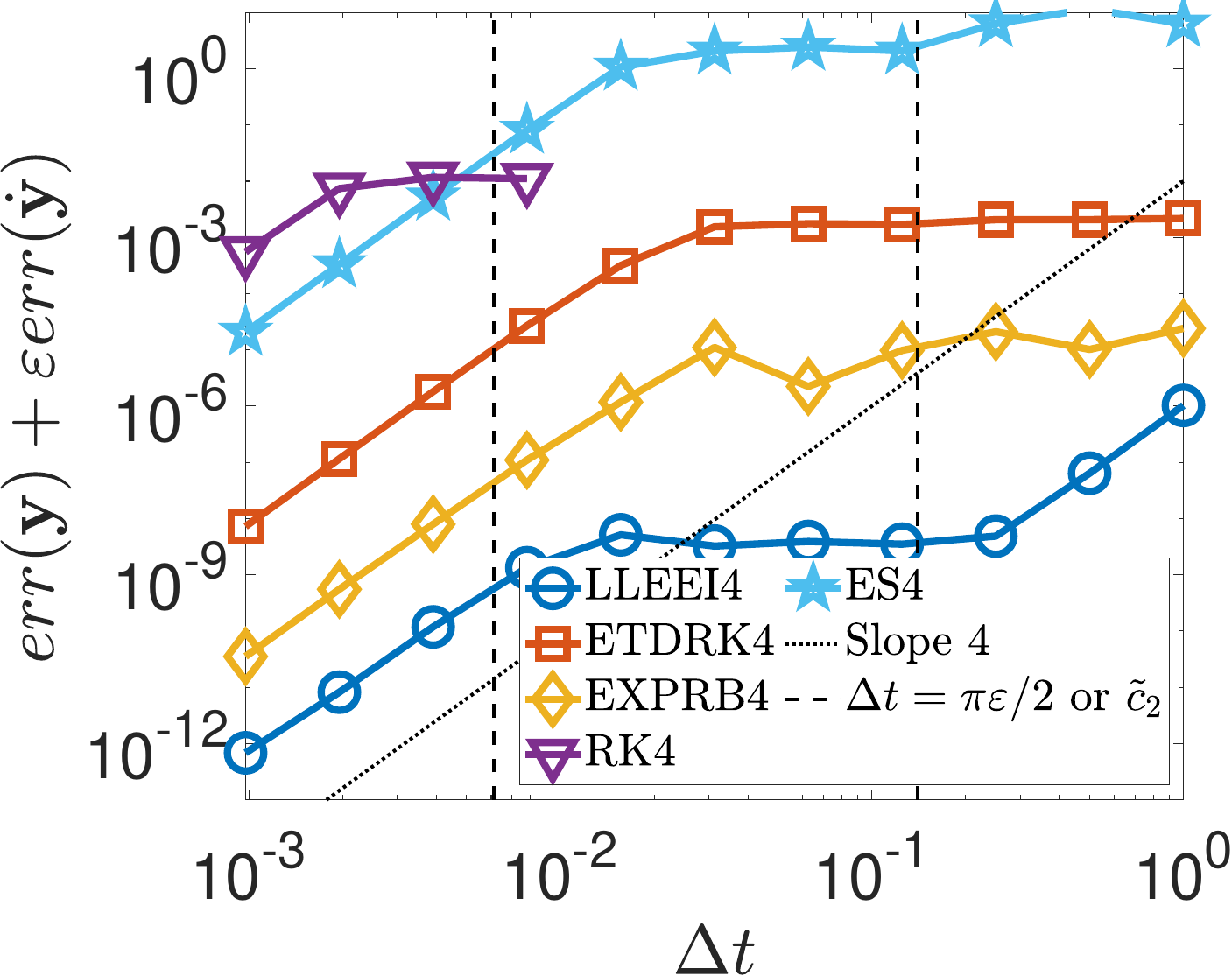}
  \end{subfigure}
  \caption{The error versus $\Delta t$. In the first row, $\varepsilon=1/2^4$; in the second row, $\varepsilon=1/2^8$. From left to right: second-order, third-order, and fourth-order schemes, respectively.}
  \label{fig2-1}
\end{figure}

Figure \ref{fig2-2} shows the error of each method as a function of $\varepsilon$ for two fixed step sizes $h=\Delta t=1/2^4$ and $h=\Delta t=1/2^8$. In each panel, a vertical dashed line is drawn to indicate the threshold $\varepsilon=\frac{h}{\tilde{c}_1}$ determined by $\tilde{c}_1$. For $\varepsilon>\frac{h}{\tilde{c}_1}$, corresponding to the small-step case, only the LLEEI method maintains uniform accuracy with respect to $\varepsilon$. The errors of the Boris method and the RK method increase significantly as $\varepsilon$ decreases. Among the splitting methods, only the first-order scheme constructed via Lie–Trotter splitting exhibits uniform convergence; increasing the order via Strang splitting or Yoshida splitting leads to a loss of uniform accuracy, in agreement with the theoretical analyses in \cite{WZ2021,Y2024}. Similar behavior is also observed in traditional EIs as $k$ increases. In the left region corresponding to the large-step case $\varepsilon<\frac{h}{\tilde{c}_1}$, the exponential-type methods produce smaller errors than the non-exponential methods. In particular, the error of the LLEEI methods decreases more rapidly as $\varepsilon$ decreases, thereby demonstrating a clear advantage over the other schemes in the highly oscillatory regime due to their improved uniform accuracy.

\begin{figure}[htbp]
  \centering
  \begin{subfigure}[b]{0.3\textwidth}
    \centering
    \includegraphics[width=\linewidth]{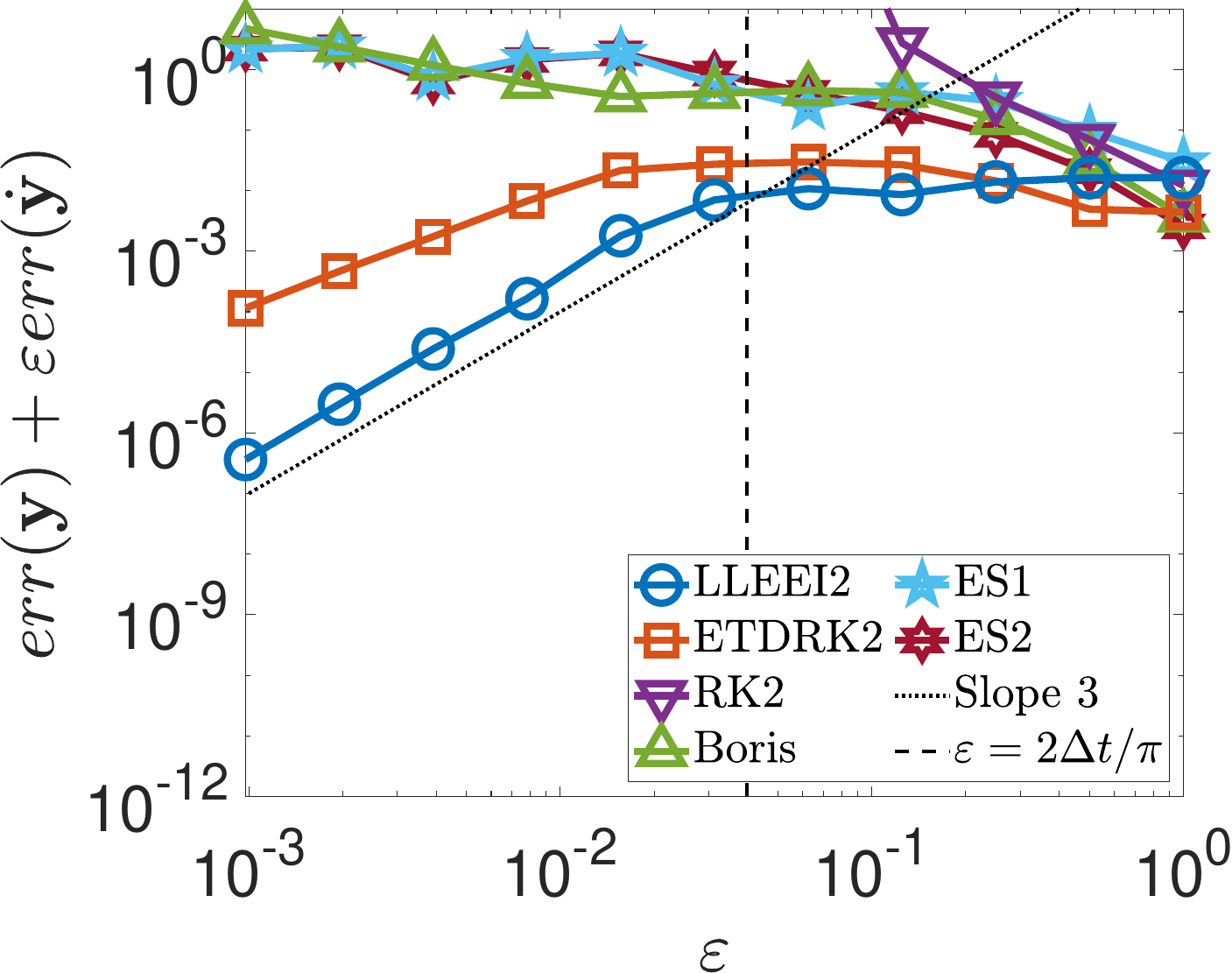}
  \end{subfigure}
  \begin{subfigure}[b]{0.3\textwidth}
    \centering
    \includegraphics[width=\linewidth]{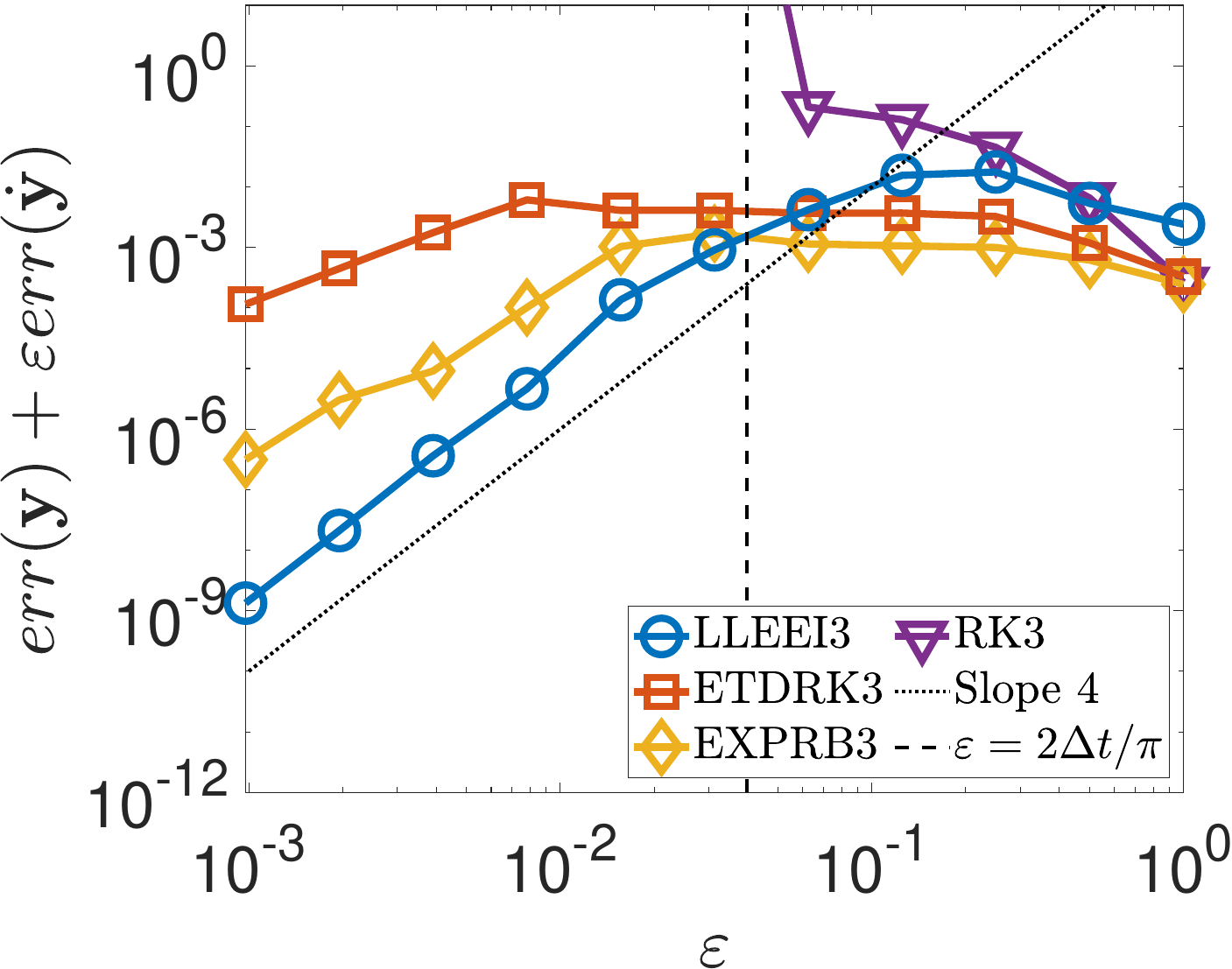}
  \end{subfigure}
  \begin{subfigure}[b]{0.3\textwidth}
    \centering
    \includegraphics[width=\linewidth]{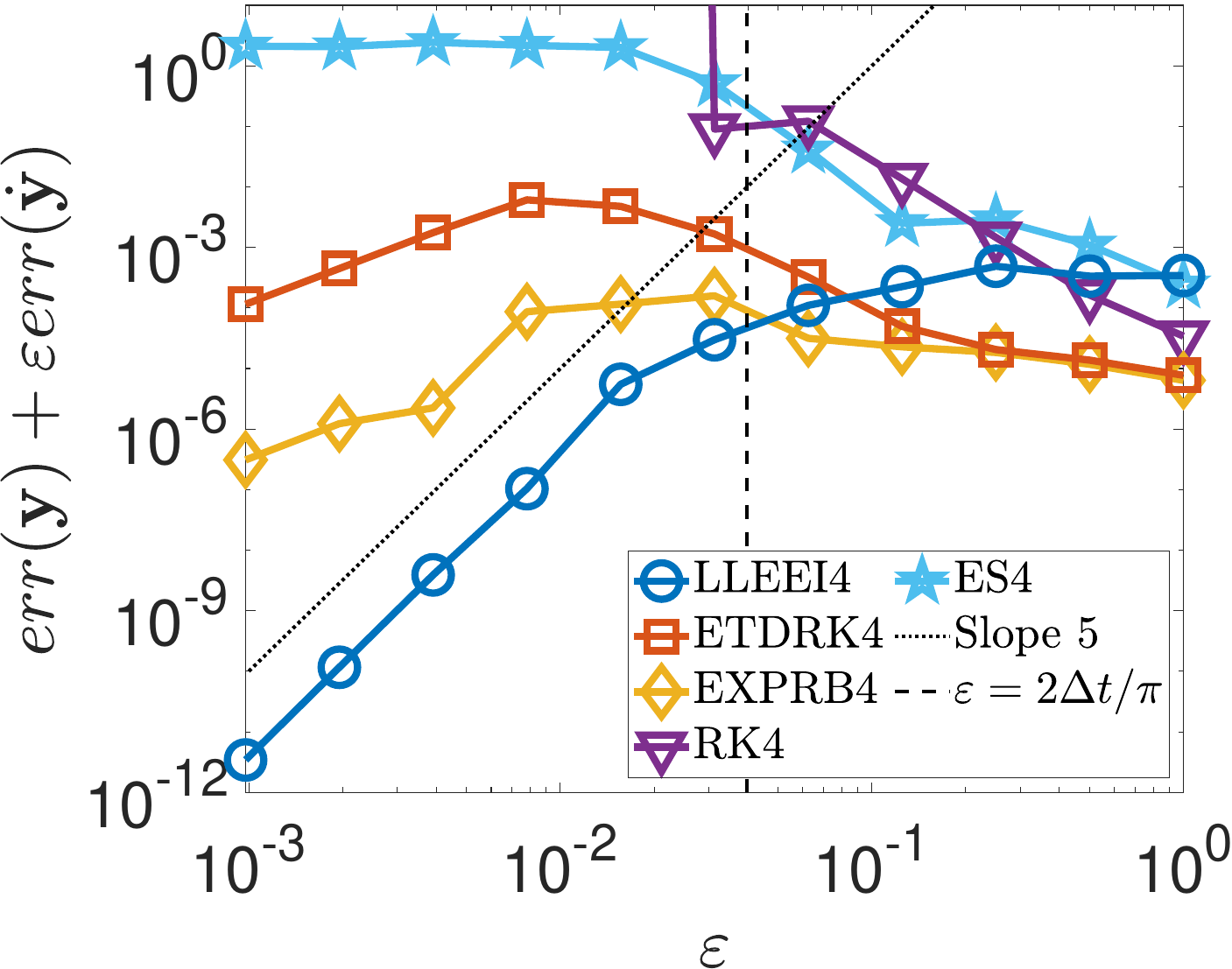}
  \end{subfigure}
  \vspace{1ex}
  \begin{subfigure}[b]{0.3\textwidth}
    \centering
    \includegraphics[width=\linewidth]{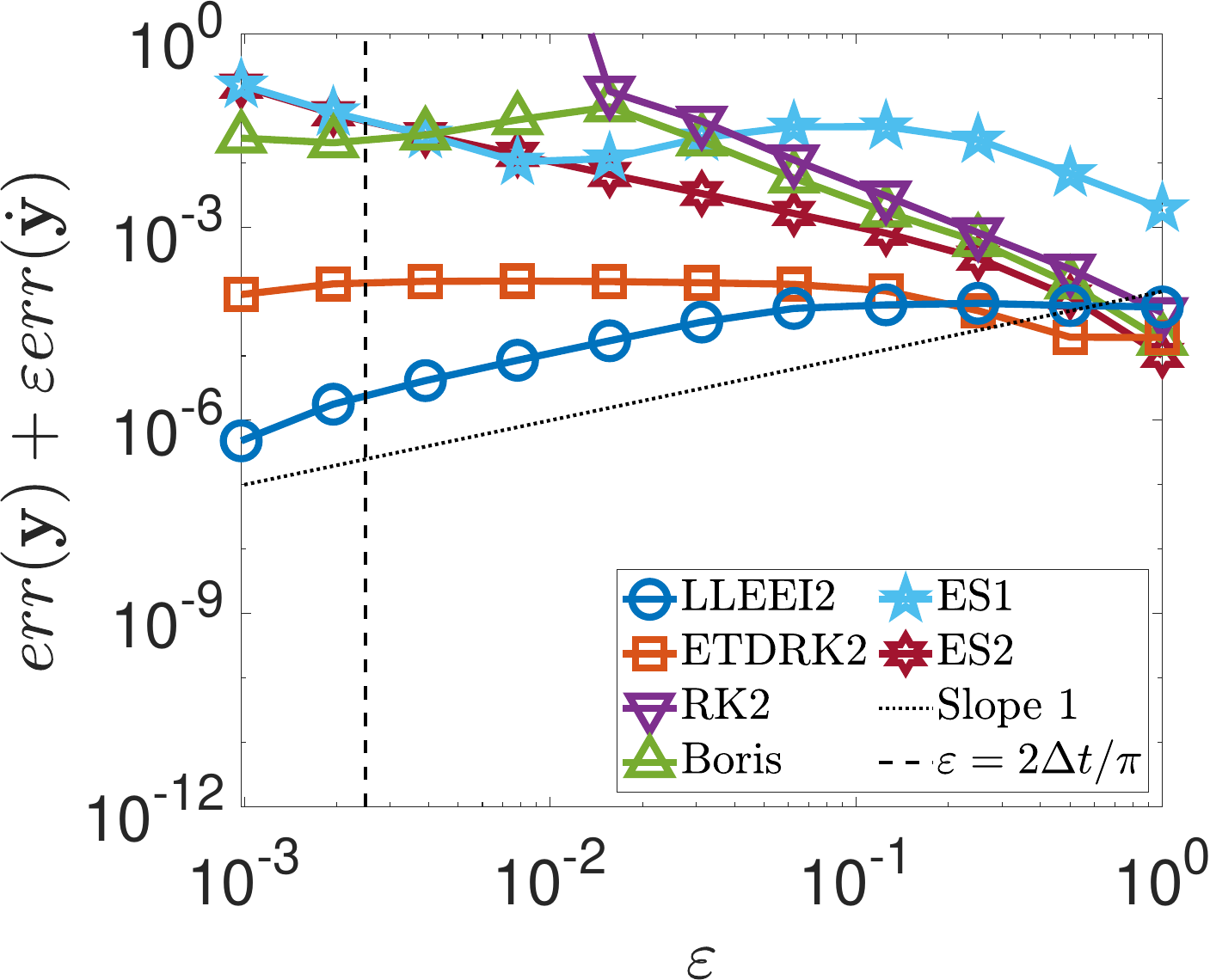}
  \end{subfigure}
  \begin{subfigure}[b]{0.3\textwidth}
    \centering
    \includegraphics[width=\linewidth]{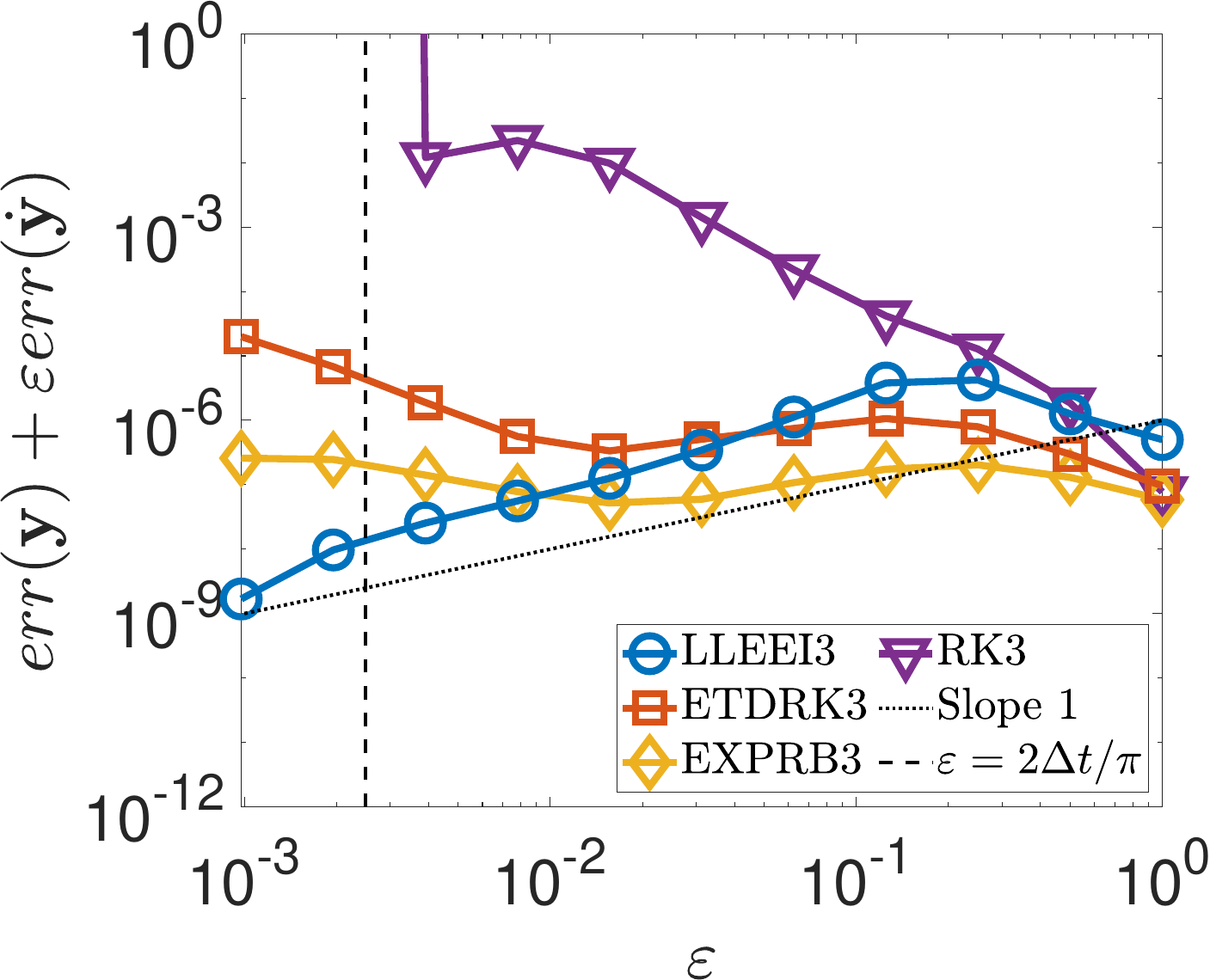}
  \end{subfigure}
    \begin{subfigure}[b]{0.3\textwidth}
    \centering
    \includegraphics[width=\linewidth]{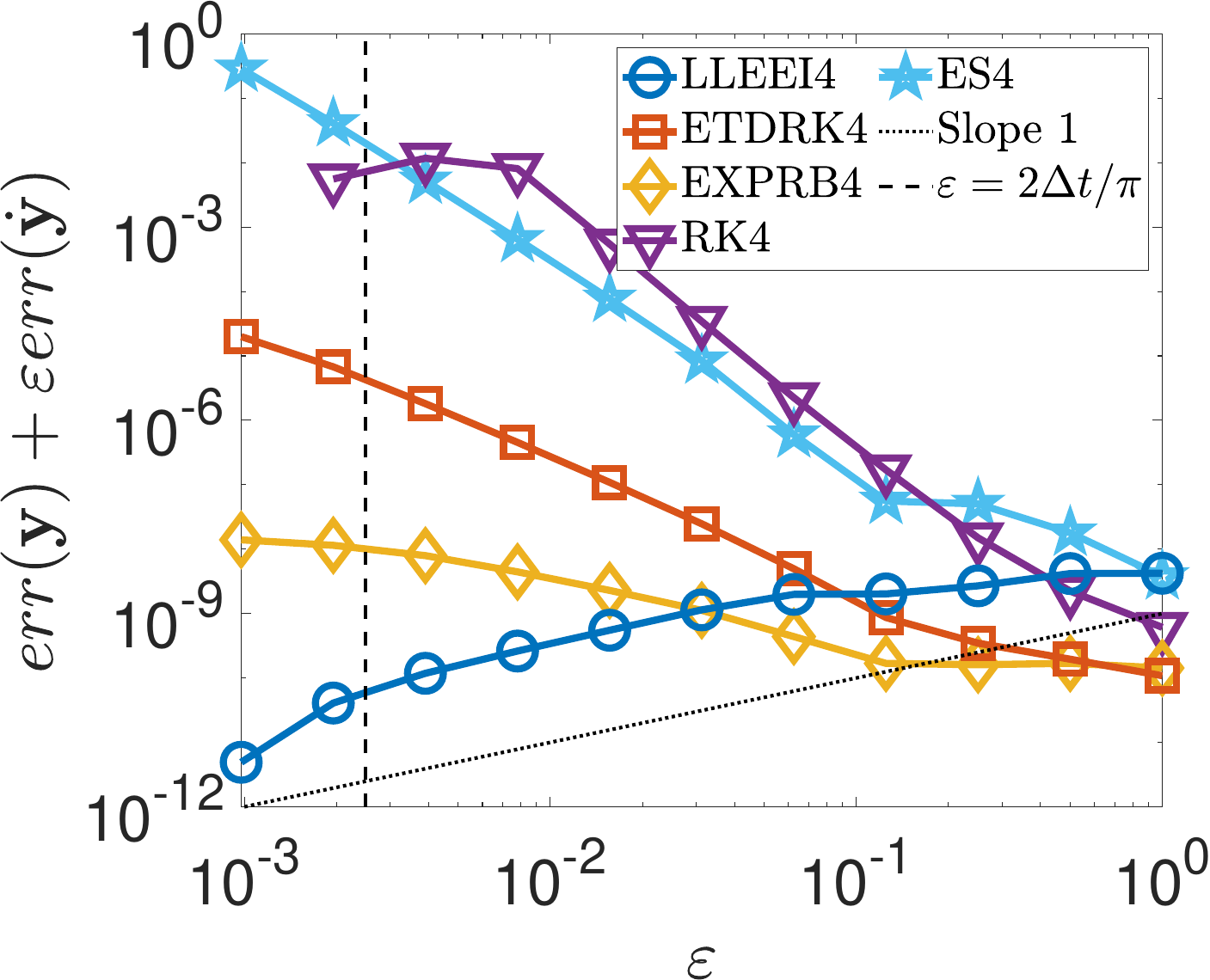}
  \end{subfigure}
  \caption{The error versus $\varepsilon$. In the first row, $\Delta t=1/2^4$; in the second row, $\Delta t=1/2^8$. From left to right: second-order, third-order, and fourth-order schemes, respectively.}
  \label{fig2-2}
\end{figure}

\refstepcounter{NumEx}\label{Ex3}
\textbf{Example \ref{Ex3}.} In this example, we examine the convergence performance of the LLEEI scheme over long time intervals $[0,T/\varepsilon]$ with $T=8$. The electric field is taken as
\begin{equation*}
  \mathbf{E}(\mathbf{y})=\mathrm{e}^{-y_2^2}\begin{bmatrix} \sin y_1\\y_2\cos y_1 \end{bmatrix},
\end{equation*}
which is derived from the potential $U(\mathbf{y})=\frac{1}{2}\mathrm{e}^{-y_2^2}\cos y_1$. The initial position and velocity are chosen as $\mathbf{y}_0=\left[\frac{3}{4},0\right]^\top$ and $\dot{\mathbf{y}}_0=\left[-\frac{5}{2},1\right]^\top$. Since the simulation interval becomes very long when $\varepsilon$ is small, computing a reference solution is extremely expensive. Therefore, we instead use the error in the total energy (\ref{Energy}), denoted by $err(H)$, as our performance metric, since this quantity is invariant.

We first fix several values of $\varepsilon$ and vary $\Delta t$. The dependence of $err(H)$ on $\Delta t$ is shown in Figure \ref{fig3-1}. Note that although we vary $\Delta t$, the actual step size used here is large $(h=\Delta t/\varepsilon)$. We observe that the LLEEI$(k+1)$ method exhibits a convergence order of $O(\Delta t^{k+1})$, which is consistent with the prediction of Theorem \ref{Thm-CPD2D-HomoB-LongTime}. One also observes that, as $\Delta t$ decreases, the error curves corresponding to smaller values of $\varepsilon$ deviate successively from the slope $k+1$. This is because the step size at this stage is already on the $O(1)$ scale and falls below the threshold $\tilde{c}_2$; consequently, the error remains bounded, as already verified in the previous two examples for step sizes $h\in[\tilde{c}_1\varepsilon,\tilde{c}_2]$.

\begin{figure}[htbp]
  \centering
  \begin{subfigure}[b]{0.3\textwidth}
    \centering
    \includegraphics[width=\linewidth]{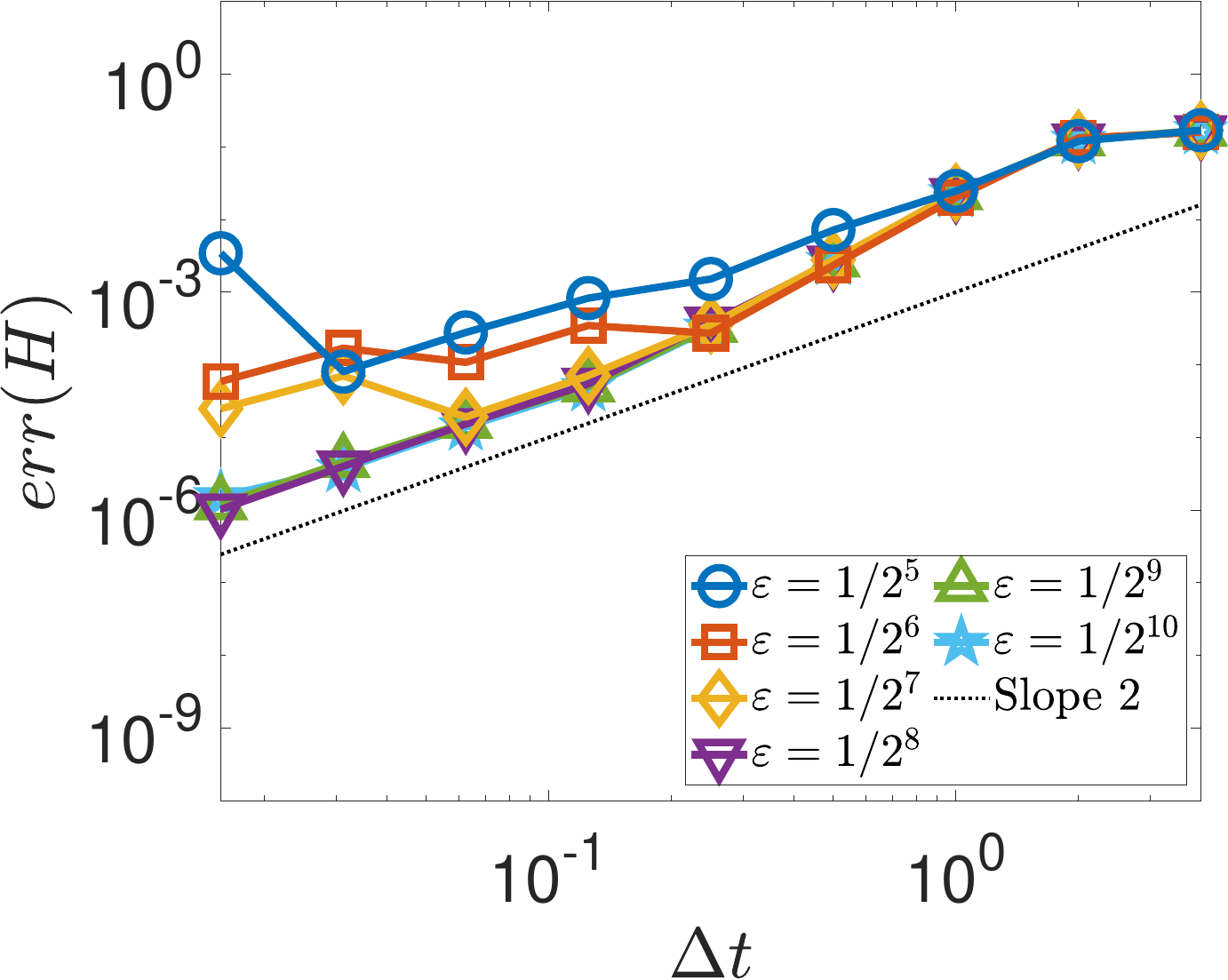}
  \end{subfigure}
  \begin{subfigure}[b]{0.3\textwidth}
    \centering
    \includegraphics[width=\linewidth]{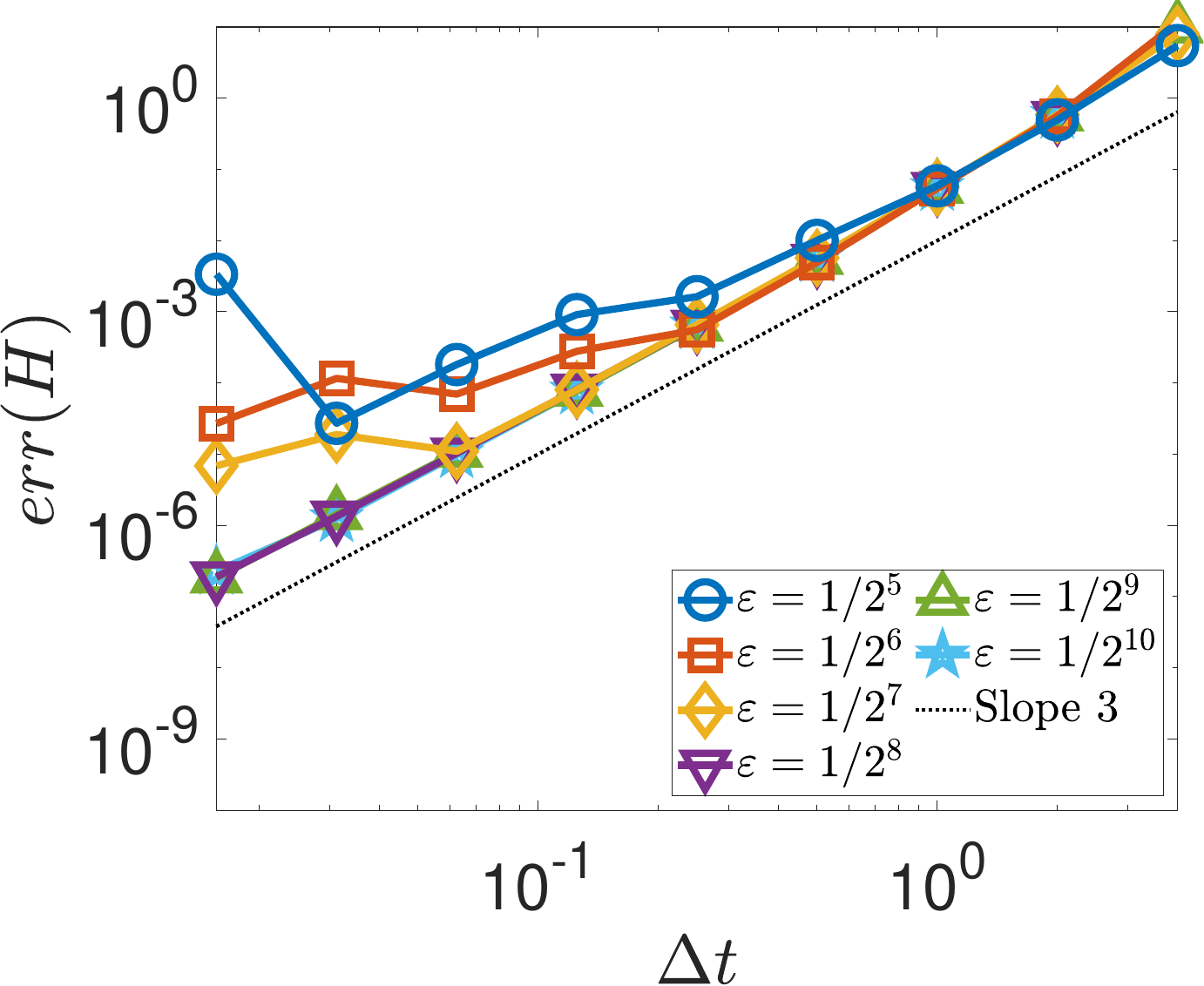}
  \end{subfigure}
  \begin{subfigure}[b]{0.3\textwidth}
    \centering
    \includegraphics[width=\linewidth]{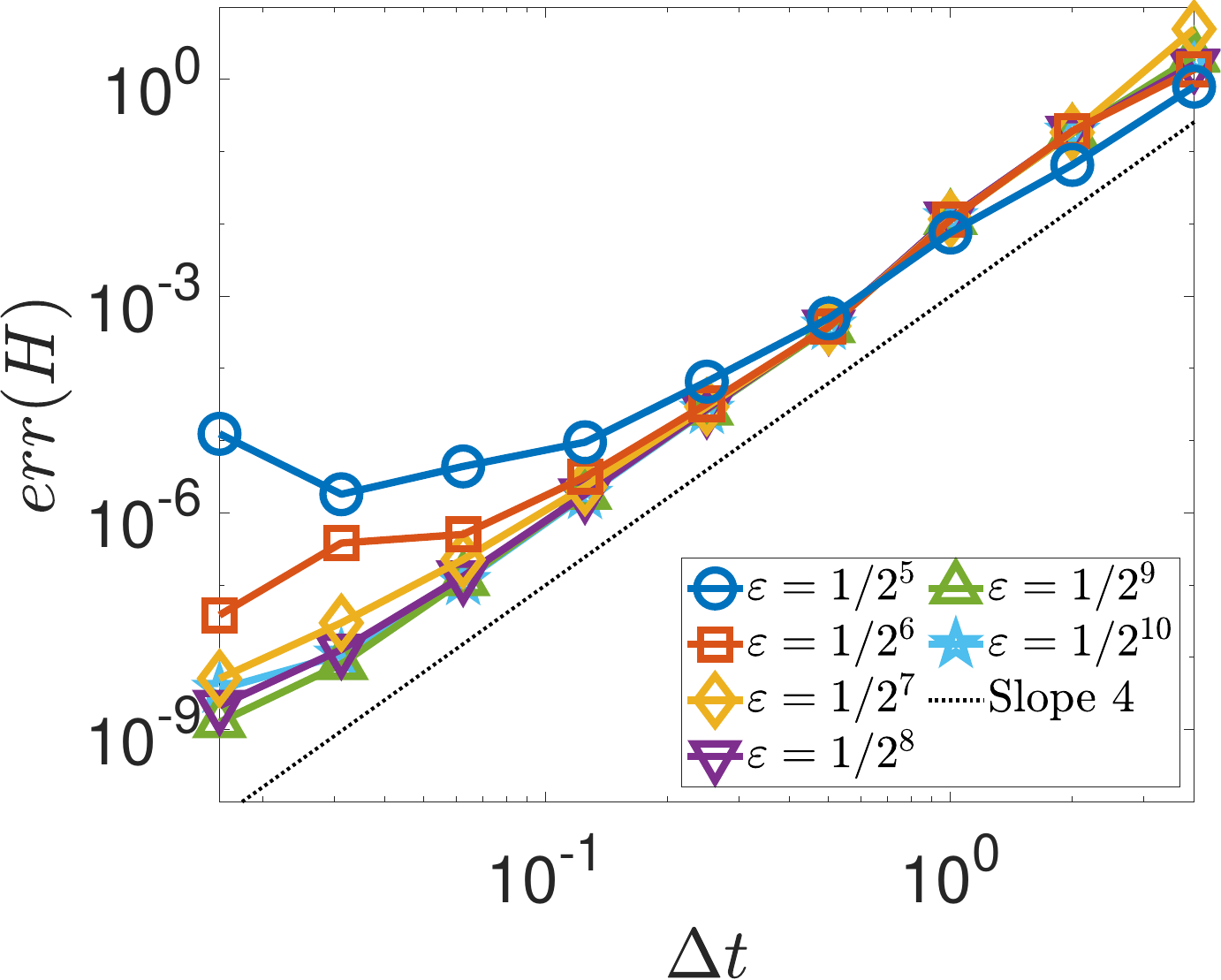}
  \end{subfigure}
  \caption{The energy error versus $\Delta t$. From left to right: $k=1,2,3$, respectively.}
  \label{fig3-1}
\end{figure}

We next fix several values of $\Delta t$ and vary $\varepsilon$ to assess the uniformity of the error. As shown in Figure \ref{fig3-2}, the error curves first exhibit a decaying trend as $\varepsilon$ decreases and then level off. This initial decay is again due to the step size lying within the interval $[\tilde{c}_1\varepsilon,\tilde{c}_2]$, which causes the error to decrease as a power of $\varepsilon$. The figure further confirms that the numerical error remains uniform with respect to $\varepsilon$, even when step sizes of order $O(\varepsilon^{-1})$ are used.

\begin{figure}[htbp]
  \centering
  \begin{subfigure}[b]{0.3\textwidth}
    \centering
    \includegraphics[width=\linewidth]{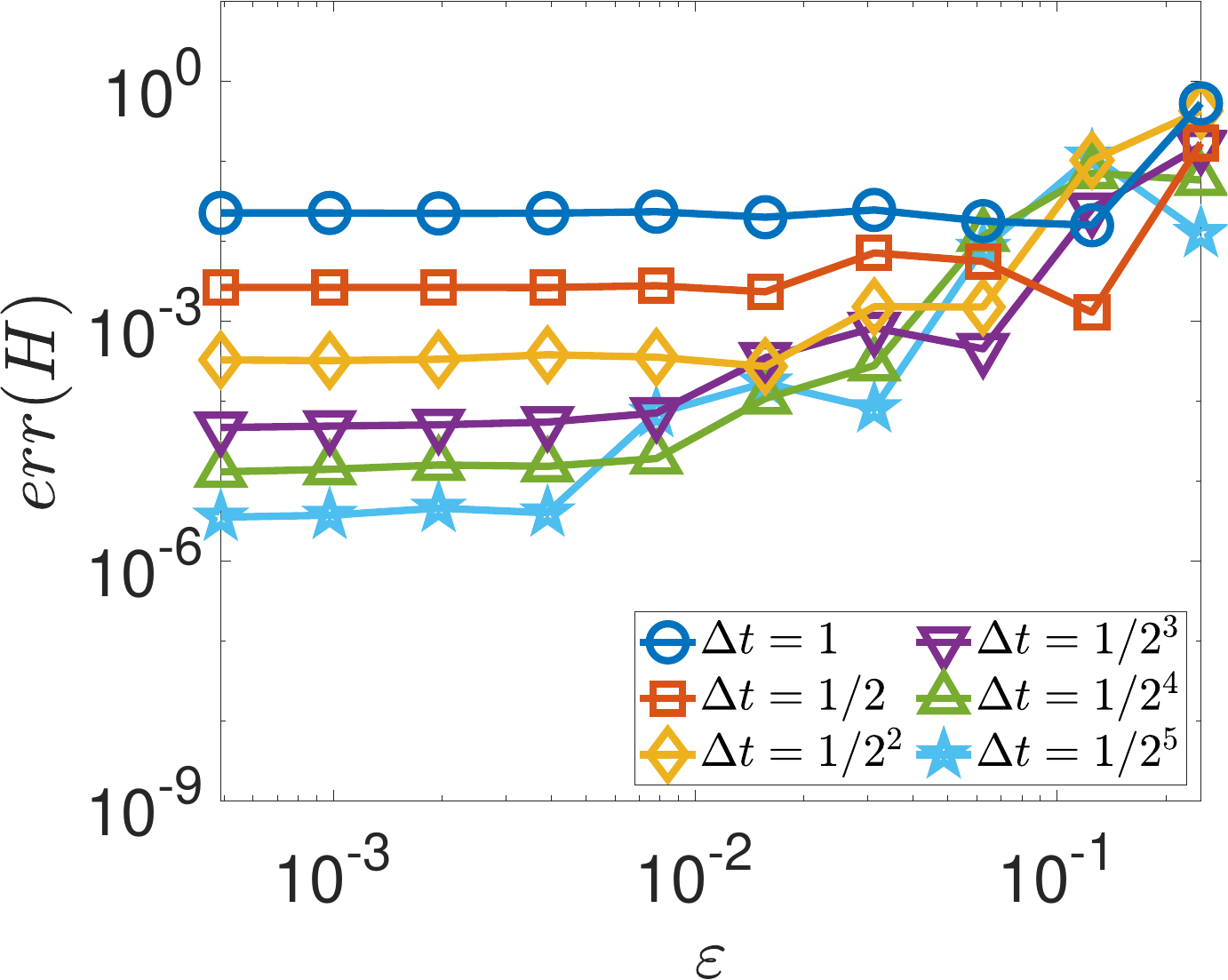}
  \end{subfigure}
  \begin{subfigure}[b]{0.3\textwidth}
    \centering
    \includegraphics[width=\linewidth]{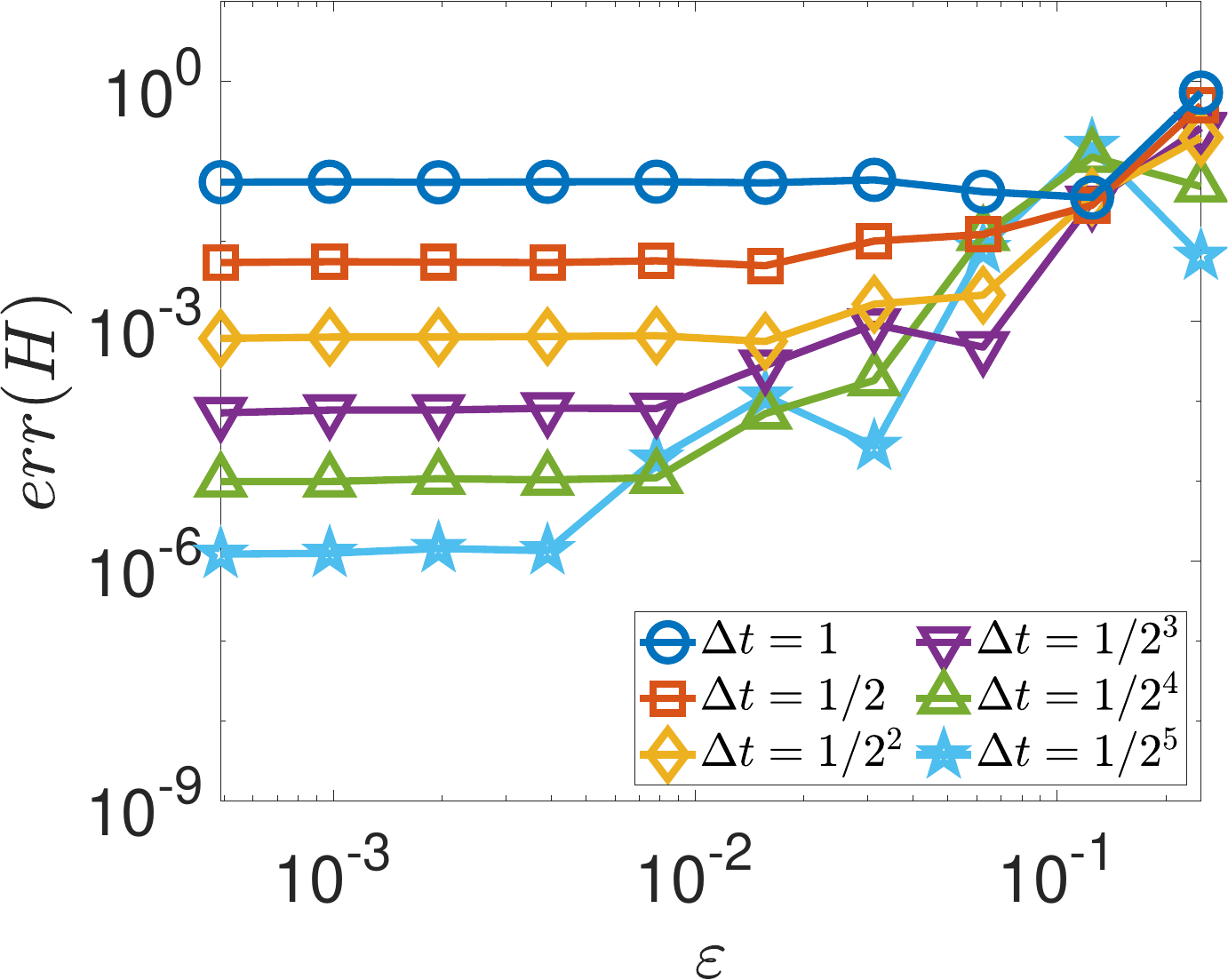}
  \end{subfigure}
  \begin{subfigure}[b]{0.3\textwidth}
    \centering
    \includegraphics[width=\linewidth]{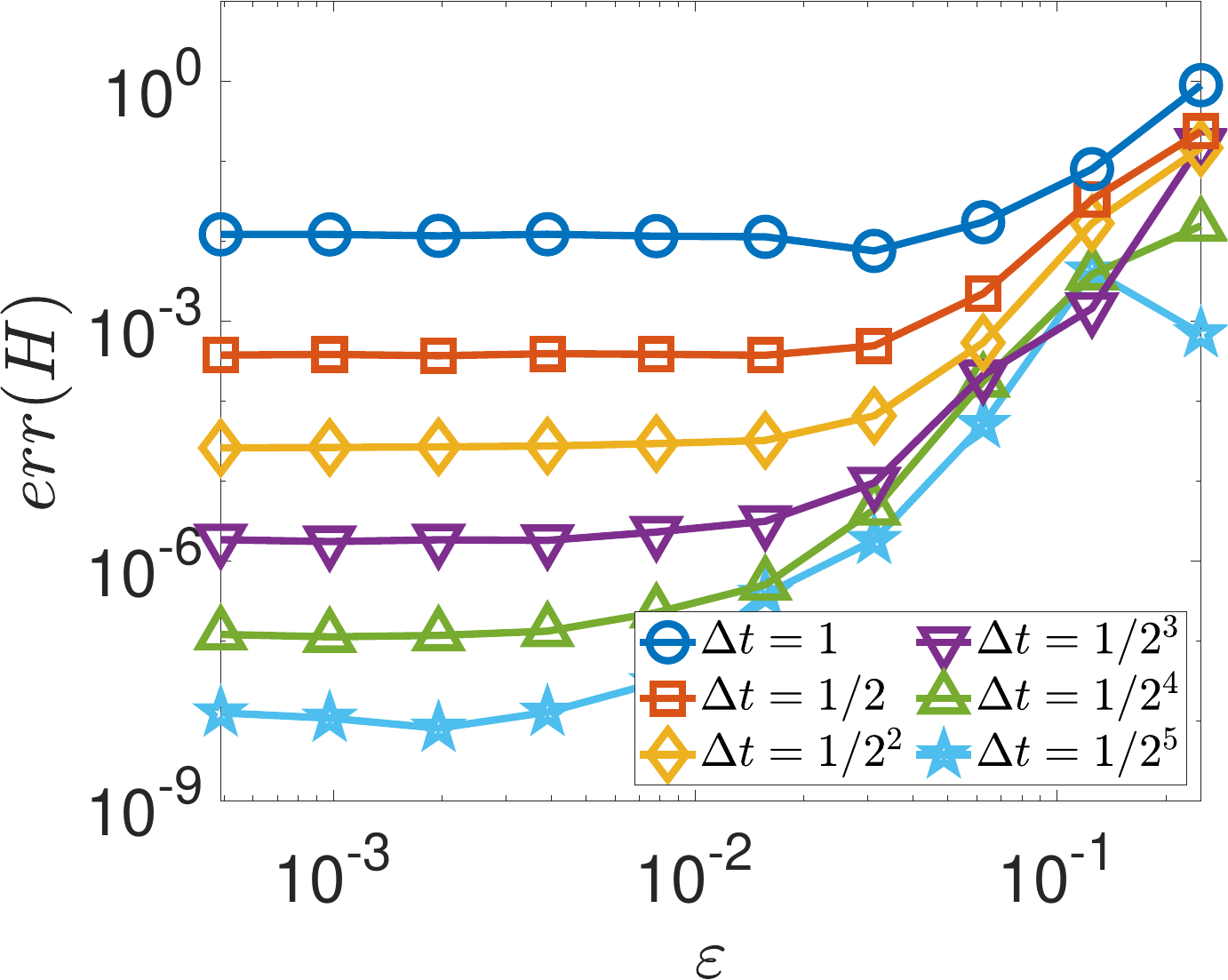}
  \end{subfigure}
  \caption{The energy error versus $\varepsilon$. From left to right: $k=1,2,3$, respectively.}
  \label{fig3-2}
\end{figure}

In practical numerical simulations, it is more common to vary the step size $h$ while keeping the final time fixed. To this end, we set $T=1024$, choose step sizes $h=T/2^j(j=0,\cdots,10)$, and conduct experiments for three different values of $\varepsilon:\varepsilon=1/2^4,1/2^6,1/2^8$. The results are shown in Figure \ref{fig3-3}. The LLEEI methods clearly maintain high-order convergence. The empirical threshold in this example is approximately $\tilde{c}_2\approx 3.96$, which successfully distinguishes the convergence behavior across different convergence regimes in the figure. Moreover, comparing the three cases, we again observe that smaller $\varepsilon$ yields more accurate numerical results, further confirming that the LLEEI scheme exhibits improved uniform convergence with respect to the magnetic flux density.

\begin{figure}[htbp]
  \centering
  \begin{subfigure}[b]{0.3\textwidth}
    \centering
    \includegraphics[width=\linewidth]{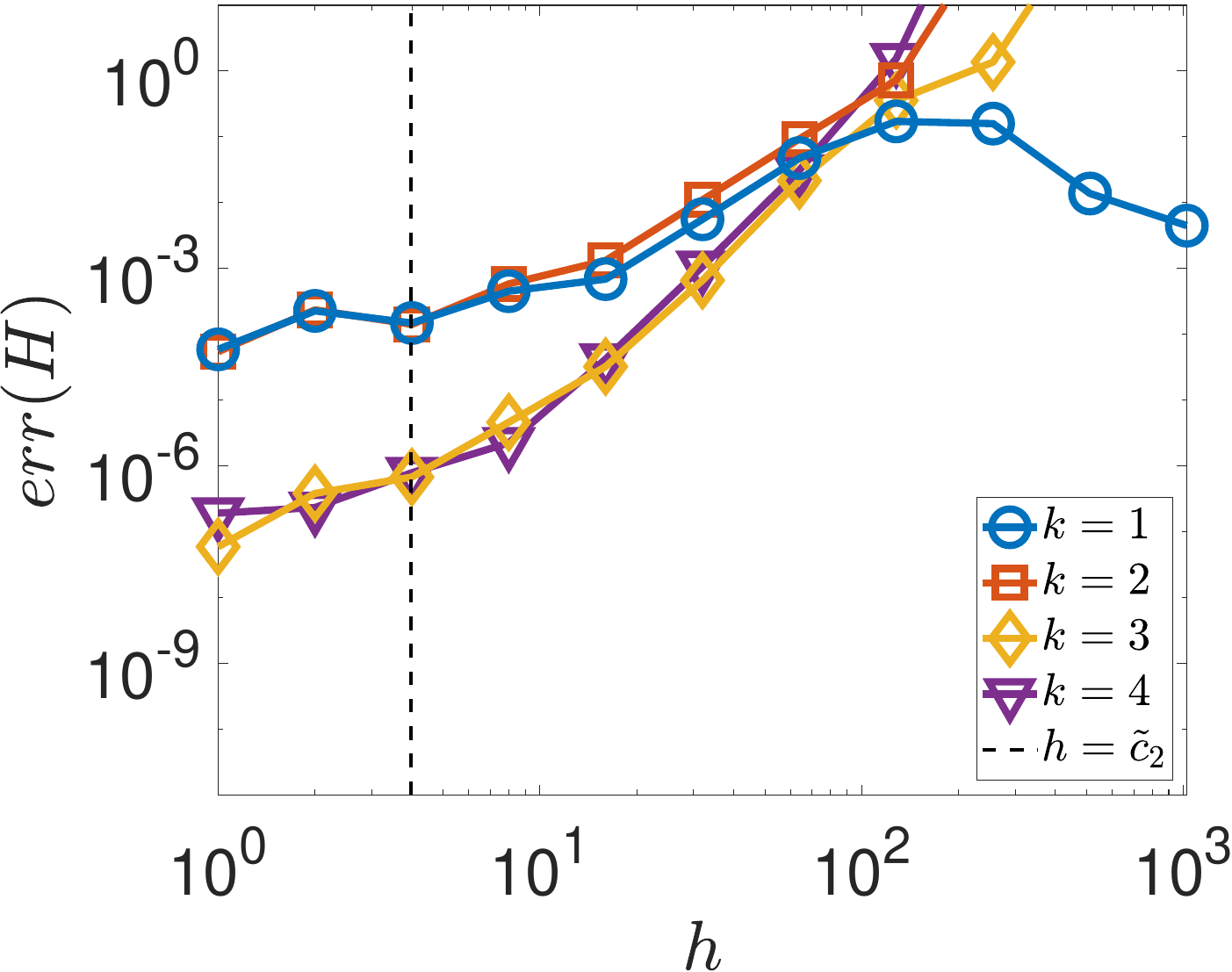}
  \end{subfigure}
  \begin{subfigure}[b]{0.3\textwidth}
    \centering
    \includegraphics[width=\linewidth]{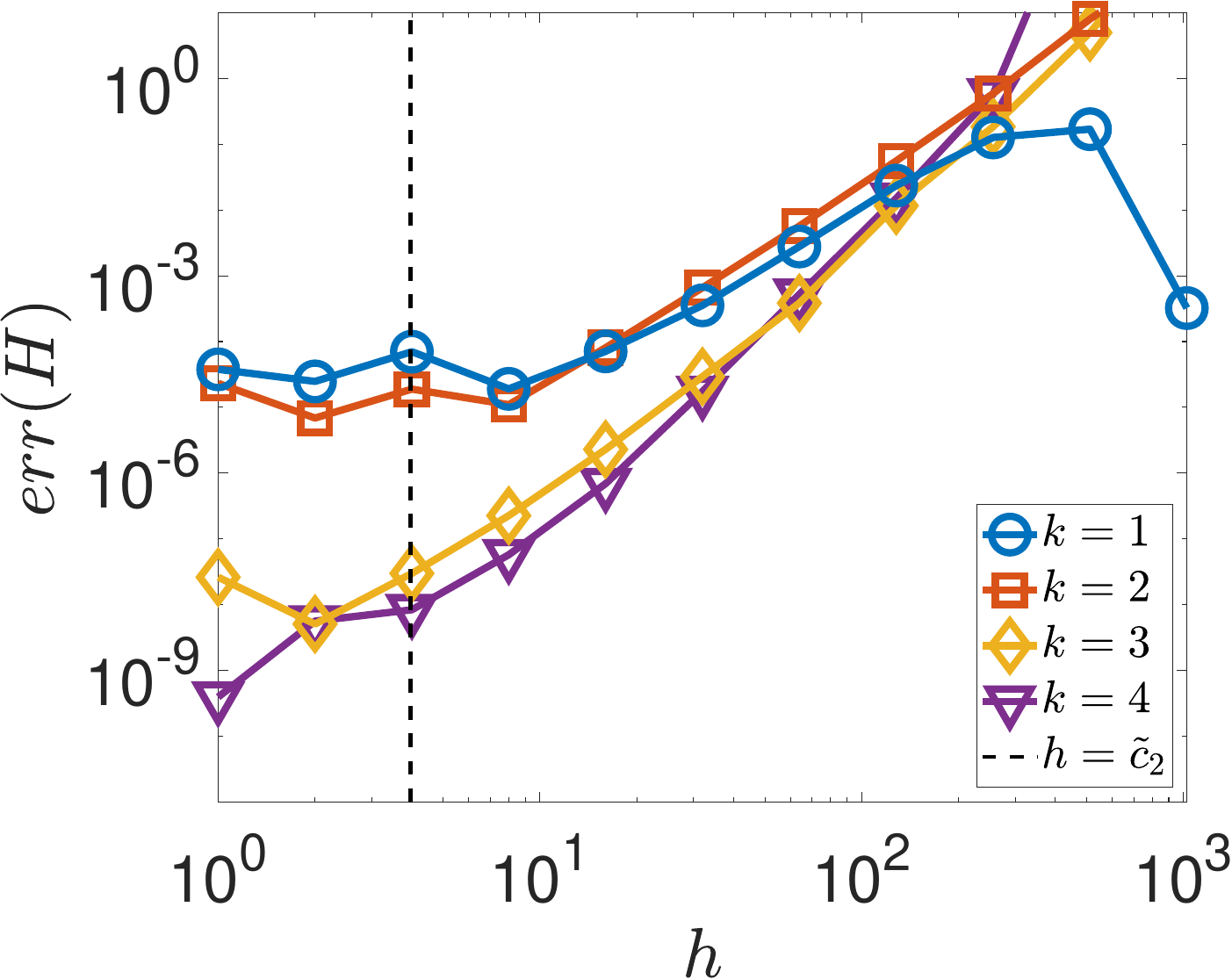}
  \end{subfigure}
  \begin{subfigure}[b]{0.3\textwidth}
    \centering
    \includegraphics[width=\linewidth]{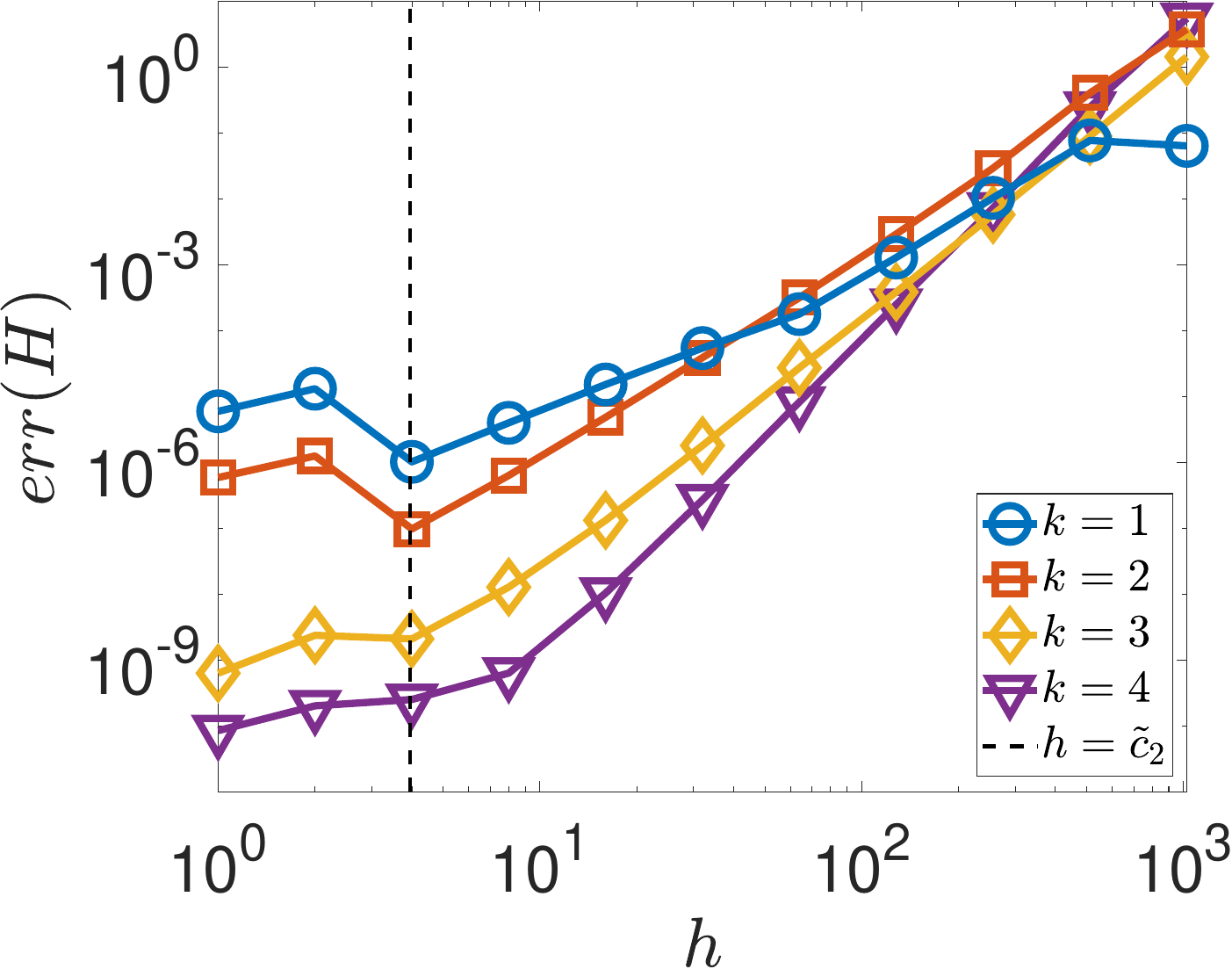}
  \end{subfigure}
  \caption{The energy error versus the step size $h$. From left to right: $\varepsilon=1/2^4,1/2^6,1/2^8$, respectively.}
  \label{fig3-3}
\end{figure}

\begin{Remark}
The uniform convergence numerically verified by $err(H)$ implicitly indicates that the velocity component also has this property, which is stronger than the convergence result for $\dot{\mathbf{y}}$ established in Theorem \ref{Thm-CPD2D-HomoB-LongTime}. Indeed, since $\dot{\mathbf{y}}$ is a highly oscillatory component in an energy-bounded system, one could further carry out an improved uniform convergence analysis using the technique in \cite{QDZ2025}. However, this procedure is cumbersome and is therefore not pursued in the present paper.
\end{Remark}

\section{Conclusion}\label{Sec5}
In this paper, we propose an exponential integrator for the CPD equation (\ref{CPD2D-homo}) based on the local linear extension technique. By means of the dimension-raising technique, we achieve a high-order linearization of this nonlinear oscillatory problem. Specifically, we introduce additional polynomial variables at a fixed point as auxiliary variables to define the local linear extension variable, and then derive their governing equations, thereby establishing the local linear extension system. Applying standard exponential integrators to this augmented system yields the new class of integrators.

The proposed methods possess three characteristic features: (\romannumeral1) theoretically, they attain arbitrarily high-order accuracy as the degree of the auxiliary polynomials increases; (\romannumeral2) they admit two independent error bounds for short-time simulations—one expressed in powers of $\Delta t$ and the other in powers of $\varepsilon$, with the latter ensuring that the numerical scheme still produces very small errors in the highly oscillatory regime even when step sizes of $O(1)$ are used; and (\romannumeral3) they allow time steps of scale $O(\varepsilon^{-1})$ while maintaining high-order uniform accuracy.

We rigorously establish the error bounds of the scheme for three step-size regimes: $O(\varepsilon)$, $O(1)$, and $O(\varepsilon^{-1})$. Two useful empirical thresholds are introduced to distinguish the convergence behavior across these regimes. Finally, numerical experiments are conducted to verify the effectiveness of the algorithm and the theoretical findings.

In future work, building on the present study, we will extend our approach to more complex CPD problems, including those involving finite Larmor radius formulations \cite{CCZ2018}, time-dependent electric fields, spatially non-uniform magnetic fields, and three-dimensional configurations.


\bibliography{ref.bib}

\end{document}